\documentclass{article}
\usepackage[utf8]{inputenc}
\usepackage{amsmath, amssymb,amsthm}
\usepackage{mathrsfs}
\usepackage{enumitem}
\usepackage{fullpage}
\usepackage{xcolor}
\usepackage{dsfont}
\usepackage{graphicx}
\usepackage[font=small,skip=0pt]{caption}
\theoremstyle{plain}

\usepackage[labelfont=bf,font=sf]{caption}
\usepackage[font={sf,small},labelfont=md]{subfig}

\usepackage[noadjust,nosort]{cite}

\usepackage{amsmath, amsthm, amssymb, tikz, mathtools, array, mathrsfs, tensor}
\usepackage{relsize}

\usetikzlibrary{shapes.geometric}
\tikzset{cross/.style={cross out, draw=black, minimum size=2.5*(#1-\pgflinewidth), inner sep=2pt, outer sep=0.5pt},
	cross/.default={1pt}}
\usetikzlibrary{calc,arrows}
\usepackage{pgfplots}
\pgfplotsset{compat=1.15}
\usepackage{polynom}
\usepackage{esint}
\usepackage{multicol}
\usepackage{dsfont}
\newcommand{\one}{\mathds{1}}
\usepackage{framed}
\usepackage{upref}

\usepackage{hyperref}
\hypersetup{colorlinks=true,linkcolor=blue,citecolor=blue,urlcolor=blue} 

\allowdisplaybreaks


\def\eps{\varepsilon}

\newcommand{\E}{\mathbb{E}}

\newcommand{\N}{\mathbb{N}}
\renewcommand{\P}{\mathbb{P}}

\newcommand{\R}{\mathbb{R}}

\newcommand{\Z}{\mathbb{Z}}

\newcommand{\CC}{\mathcal{C}}

\newcommand{\GG}{\mathcal{G}}
\newcommand{\HH}{\mathcal{H}}

\newcommand{\NN}{\mathcal{N}}

\newcommand{\PP}{\mathcal{P}}

\newcommand{\RR}{\mathcal{R}}
\renewcommand{\SS}{\mathcal{S}}
\newcommand{\TT}{\mathcal{T}}

\newcommand{\AAA}{\mathscr{A}}

\newcommand{\CCC}{\mathscr{C}}

\newcommand{\EEE}{\mathscr{E}}
\newcommand{\FFF}{\mathscr{F}}
\newcommand{\GGG}{\mathscr{G}}

\newcommand{\LLL}{\mathscr{L}}
\newcommand{\MMM}{\mathscr{M}}

\newcommand{\PPP}{\mathscr{P}}

\newcommand{\TTT}{\mathscr{T}}

\newcommand{\ZZZ}{\mathscr{Z}}

\newcommand{\ttt}{\mathfrak{t}}

\newcommand{\fb}{\mathfrak{b}}
\newcommand{\ufb}{\underline{\mathfrak{b}}}

\newcommand{\fff}{\mathfrak{f}}

\newcommand{\bDelta}{\boldsymbol{\Delta}}

\newcommand{\naesat}{\textsc{nae-sat}}
\newcommand{\knaesat}{k\textsc{-nae-sat}}
\newcommand{\cond}{\textsf{cond}}
\newcommand{\sat}{\textsf{sat}}
\def\Unif{{\sf Unif}}
\newcommand{\sol}{{\sf SOL}}
\newcommand{\rsb}{{\tiny{1}}\textsc{rsb}}
\newcommand{\size}{\textsf{size}}
\newcommand{\fav}{\mathscr{X}_{\sf fav}}

\newcommand{\g}{{\sf g}}
\newcommand{\sfb}{{\sf b}}
\newcommand{\atyp}{{\sf atyp}}
\newcommand{\typ}{{\sf typ}}
\newcommand{\Law}{\sf Law}
\newcommand{\rev}[1]{\noindent{{#1}}}
\newcommand{\bZ}{\textnormal{\textbf{Z}}}

\newcommand{\bG}{\textnormal{\textbf{G}}}
\newcommand{\bGG}{\boldsymbol{\mathcal{G}}}
\newcommand{\bI}{\textnormal{\textbf{I}}}
\newcommand{\sig}{\underline{\sigma}}

\newcommand{\uL}{\underline{\texttt{L}}}

\newcommand{\bux}{\underline{\textbf{x}}}
\newcommand{\ubz}{\underline{\textbf{z}}}
\newcommand{\uz}{\underline{z}}
\newcommand{\buz}{\boldsymbol{\underline{z}}}
\newcommand{\ux}{\underline{x}}

\newcommand{\utau}{\underline{\tau}}
\newcommand{\utheta}{\underline{\theta}}
\newcommand{\uh}{\underline{h}}

\newcommand{\uC}{\underline{C}}

\newcommand{\dul}{\dot{\underline{\ell}}}
\newcommand{\hul}{\hat{\underline{\ell}}}

\newcommand{\buL}{\underline{\textbf{L}}}
\newcommand{\bE}{\textbf{E}}
\newcommand{\bL}{\textbf{L}}

\newcommand{\bx}{\mathbf{x}}

\newcommand{\bv}{\mathbf{v}}

\newcommand{\bsigma}{\boldsymbol{\sigma}}

\newcommand{\bsig}{\underline{\boldsymbol{\sigma}}}

\newcommand{\bnu}{\boldsymbol{\nu}}

\newcommand{\lit}{\textnormal{lit}}

\newcommand{\sy}{\sf sy}
\newcommand{\tr}{{\sf tr}}

\newcommand{\lab}{\textnormal{lab}}

\newcommand{\qdot}{\dot{\mathbf{q}}}
\newcommand{\qhat}{\hat{\mathbf{q}}}
\newcommand{\dotq}{\dot{q}}

\newcommand{\rr}{{{\scriptsize{\texttt{R}}}}}
\newcommand{\bb}{{{\scriptsize{\texttt{B}}}}}

\newcommand{\fs}{{\scriptsize{\texttt{S}}}}
\newcommand{\ff}{\textnormal{\small{\texttt{f}}}}
\newcommand{\tz}{\small{\texttt{z}}}
\newcommand{\mm}{{{\texttt{m}}}}
\newcommand{\fF}{\scriptsize{\texttt{F}}}

\newcommand{\tL}{\texttt{L}}

\newcommand{\la}{\lambda}

\newcommand{\cyc}{{\sf cyc}}

\newcommand{\ee}{\mathfrak{E}}

\newcommand{\tns}{\textnormal{s}}

\newcommand{\wt}[1]{\widetilde{#1}}

\newcommand{\norm}[1]{\left\|{#1}\right\|} 
\newcommand{\tsq}{\scalebox{0.55}{$\square$}}

\newcommand{\dHH}{\dot{\HH}}
\newcommand{\hHH}{\hat{\HH}}
\newcommand{\bHH}{\bar{\HH}}

\newcommand{\given}{\,|\,}
\newcommand{\bgiven}{\,\,\Big|\,\,}
\newcommand{\bbgiven}{\,\,\bigg|\,\,}
\newcommand{\lekt}{\lesssim_{k,t}}
\def\de{{\rm d}}

\DeclareMathOperator{\Var}{Var}
\DeclareMathOperator{\Cov}{Cov}

\DeclareMathOperator{\tv}{TV} 

\newtheorem{thm}{Theorem}[section]
\newtheorem{prop}[thm]{Proposition}
\newtheorem{cor}[thm]{Corollary}
\newtheorem{lemma}[thm]{Lemma}

\theoremstyle{definition}
\newtheorem{defn}[thm]{Definition}

\newtheorem{remark}[thm]{Remark}
\newtheorem{obs}[thm]{Observation}

\title{Local geometry of NAE-SAT solutions in the condensation regime}

\author{
Allan Sly\thanks{Department of Mathematics, Princeton University. Email: \textup{\tt asly@math.princeton.edu}}\and 
Youngtak Sohn \thanks{Department of Mathematics Massachusetts Institute of Technology. Email: \textup{\tt youngtak@mit.edu}}}

\begin{document}
	\bibliographystyle{acm}
	
	\renewcommand{\thefootnote}{\arabic{footnote}} \setcounter{footnote}{0}
	
	 \maketitle
	
\begin{abstract}
The local behavior of typical solutions of random constraint satisfaction problems (\textsc{csp})
describes many important phenomena including
clustering thresholds, decay of correlations, and the behavior of message passing algorithms.  When the constraint density is low, studying the planted model is a powerful technique for determining this local behavior which in many examples has a simple Markovian structure. \rev{The} work of Coja-Oghlan, Kapetanopoulos, M\"{u}ller (2020) showed that for a wide class of models, this description applies up to the so-called condensation threshold.  

Understanding the local behavior after the condensation threshold is more complex due to long-range correlations. In this work, we revisit the random regular \textsc{nae-sat} model in the condensation regime and determine the local weak limit which describes a random solution around a typical variable. This limit exhibits a complicated non-Markovian structure arising from the space of solutions being dominated by a small number of large clusters.  This is the first description of the local weak limit in the condensation regime for any sparse random \textsc{csp}s in the one-step replica symmetry breaking (1\textsc{rsb}) class. Our result is non-asymptotic and characterizes the tight fluctuation  $O(n^{-1/2})$ around the limit. Our proof is based on coupling the local neighborhoods of an infinite spin system, which encodes the structure of the clusters, to a broadcast model on trees whose channel is given by the 1\textsc{rsb} belief-propagation fixed point. We believe that our proof technique has broad applicability to random \textsc{csp}s in the 1\textsc{rsb} class.

 \end{abstract}

 \section{Introduction}

A random constraint satisfaction problem (r\textsc{csp}) involves $n$ variables $\underline{z}=\{z_i\}_{i\leq n}\in \mathfrak{X}^n$ drawn from a finite alphabet set $\mathfrak{X}$, satisfying $m\equiv \alpha n$ random constraints. The aim is to analyze the solution space of r\textsc{csp}s as $n$ and $m$ increase, with $\alpha$ constant. Major advances have been made by statistical physicists using deep but non-rigorous theory, which describes a series of phase transitions as the constraint density $\alpha$ grows. Their insights apply to a wide class of r\textsc{csp}s belonging to the so-called one-step replica symmetry breaking (1\textsc{rsb}) class, including $k$-sat, nae-sat, and coloring (\cite{mpz02,kmrsz07}, see also \cite{anp05}, \cite{mm09}). We'll begin by describing some of the main predictions made by the physicists \cite{kmrsz07}. See Figure \ref{fig:phase} for the pictorial description of the conjectured phase diagram.

\begin{figure}
		\centering
		\begin{tikzpicture}[scale=0.9]
		\node[inner sep=0pt] (unique) at (0,0)
		{\includegraphics[width=2cm]{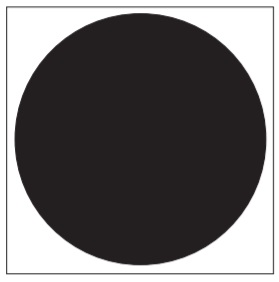}};
		
		\node[inner sep=0pt] (ext) at (2.4,0)
		{\includegraphics[width=2cm]{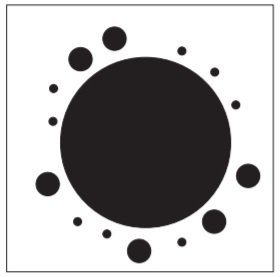}};
		
		\node[inner sep=0pt] (clust) at (4.8,0)
		{\includegraphics[width=2cm]{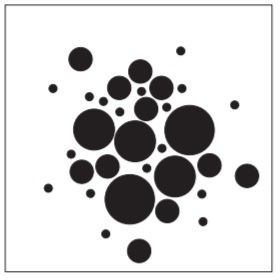}};
		
		\node[inner sep=0pt] (cond) at (7.2,0)
		{\includegraphics[width=2cm]{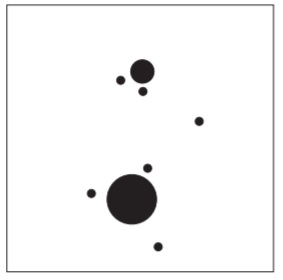}};
		
		\node[inner sep=0pt] (unsat) at (9.6,0)
		{\includegraphics[width=2cm]{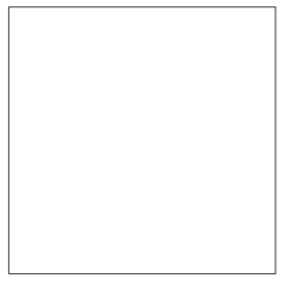}};
		
		\draw (1.2,1.1) -- (1.2,0.5);
		\draw (3.6,1.1) -- (3.6,0.5);
		\draw (6,1.1) -- (6,0.5);
		\draw (8.4,1.1) -- (8.4,0.5);
		\node[anchor=south] (un) at (1.2,1) {$\alpha_{\textsf{uniq}}$};
		\node[anchor=south] (un) at (3.6,1) {$\alpha_{\textsf{clust}}$};
		\node[anchor=south] (un) at (6,1) {$\alpha_{\textsf{cond}}$};
		\node[anchor=south] (un) at (8.4,1) {$\alpha_{\textsf{sat}}$};
		\node[anchor=north] (asd) at (0,-1) {uniqueness};
		\node[anchor=north] (asd) at (2.4,-1) {extremality};
		\node[anchor=north] (asd) at (4.8,-1) {clustering};
		\node[anchor=north] (asd) at (7.2,-1) {condensation};
		\node[anchor=north] (asd) at (9.6,-1) {unsat};
		
		\draw[->] (6.6,-1.72)--(10.5,-1.72);
		\node[anchor=north] (aaas) at (8.8,-1.7) {constraint density $\alpha$};
		\end{tikzpicture}
		\caption{\textit{Figure adapted from} \cite{kmrsz07,dss22}. A pictorial description of the conjectured phase diagram of random constraint satisfaction problems in the one-step replica symmetry breaking class. In the condensation regime $(\alpha_{\textsf{cond}},\alpha_{\textsf{sat}})$, a bounded number of clusters contain most of the solutions and the uniform measure over the solutions fails to be contiguous with the planted model.}\label{fig:phase}
	\end{figure}

When the density of constraints $\alpha$ is below the uniqueness threshold $\alpha_{\textsf{uniq}}$, all of the solutions lie in a single cluster. Here, a cluster is defined to be the connected component of the solution space where two solutions are connected if they differ in 
\rev{one} variable. As $\alpha$ increases, the space of solutions undergoes a shattering threshold $\alpha_{\textsf{clust}}$ after which the space of solutions shatters into exponentially many clusters of solutions, each well separated from each other~ \cite{achlioptas2008algorithmic}. While the space of solutions becomes more complex at this point, the behavior of a typical solution retains a simple description.  In particular, the uniform measure over the solutions is contiguous with respect to the so-called planted model. This was recently established rigorously by Coja-Oghlan, Kapetanopoulos, and Müller~\cite{coja2020replica} up to the condensation threshold $\alpha_{\cond}$ for several models including \textsc{nae-sat} and coloring. 

A second threshold, which is of primary interest in this paper, is the condensation threshold $\alpha_{\cond}$.  For $\alpha\in(\alpha_{\textsf{clust}},\alpha_{\cond})$ each cluster of solutions has only an exponentially small fraction of the total number of solutions while for $\alpha\in(\alpha_{\cond},\alpha_{\sat})$ most of the solutions are contained in $O(1)$ number of clusters.  Indeed, a more refined prediction is that the cluster sizes follow a Poisson-Dirichlet distribution \cite{kmrsz07}. This is the regime in which the model is said to be 1\textsc{rsb}, or in the \textit{condensation regime}.  Formally, this means that if we look at the normalized Hamming distance of two randomly chosen solutions, it is concentrated on two points.  This corresponds to a positive probability of having two solutions in the same cluster, in which case they are close, as well as a positive probability of two solutions in different clusters, in which case they are much further.  While this is predicted in many models\rev{,} it has so far only been established in the regular \textsc{nae-sat} for large $k$~\cite{NSS}.

It is further conjectured that not only does the structure of the space of solutions exhibit a phase transition at $\alpha_{\cond}$, \rev{but} so does the local distribution of the individual solutions themselves.  In particular, given a solution drawn uniformly at random, consider the empirical distribution of the solution in a ball of radius $2t$ around variables $1\leq i\leq n$. Here, a ball of radius $2t$ is with respect to the factor graph induced by the constraints and the variables. For example, if $2$ variables are involved in the same constraint, they have distance $2$.

For $\alpha<\alpha_{\cond}$, because of contiguity, it suffices to study the planted model to determine the limit of such local empirical distribution. Here, the planted model means taking a fixed ``planted'' assignment of the variables and then choosing the constraints conditioned to satisfy the planted assignment.  The local empirical distribution of the planted model admits a simple description as it can be studied with the configuration model.  In the case of \textsc{nae-sat} or colorings on random regular graphs, it is simply the uniform distribution \rev{of} solutions on a regular tree, which is Markovian in the sense that the spins along any path follow\rev{s} a Markov chain whose transition probabilities can be readily calculated.  This then describes the behavior of a random solution up to $\alpha_{\cond}$.

In this paper, we investigate the regime $\alpha>\alpha_{\cond}$ in the random $d$-regular $k$-\textsc{nae-sat} model for large $k\geq k_0$, where $k_0$ is an absolute constant. This regime presents a complex local empirical distribution, deviating from the planted model. The \textsc{nae-sat} problem offers additional symmetries compared to $k$-sat that make it more tractable from a mathematical viewpoint. Nevertheless, it is predicted to belong to the same 1\textsc{rsb} universality class of r\textsc{csp}s as random $k$-\textsc{sat} and random graph coloring, thus sharing similar qualitative behaviors. Let us give an informal statement of our main theorem.
\begin{thm}\label{thm:informal.main}(Informal)
For $k\geq k_0$ and $\alpha \in (\alpha_{\cond}(k),\alpha_{\sat}(k))$, consider a random regular $\knaesat$ solution $\ubz\in \{0,1\}^n$ defined in Section \ref{subsec:def}. For $t\geq 1$, the empirical distribution over balls of radius $2t$ of the solution $\ubz$ converges to an explicit non-Markovian limit $\PP_\star^t$ with $\Theta_k(1/\sqrt{n})$ fluctuations.
\end{thm}
We refer to Theorem \ref{thm:main} below for the formal statement. \rev{An} explicit definition of $\PP_\star^t$ is given in Section \ref{s:lwl}. Unlike in the planted case, $\PP_\star^t$ is non-Markovian as shown in Section \ref{ss:further}.

We emphasize that for $t\geq 2$, characterization of the local weak limit $\PP_\star^t$ with $O(n^{-1/2})$ fluctuation poses significant difficulties, demanding novel methods not covered in earlier works \cite{NSS, nss2}. Indeed, \cite{NSS} has only studied statistics of depth $2$ neighborhoods, and they established concentration in $\ell^{\infty}$-distance of the \textit{free component profile} with larger distance $O(n^{-1/2}\log n)$. In order to establish Theorem \ref{thm:informal.main}, we \rev{build upon the estimates obtained from \cite{NSS} to improve their results to} $\ell^1$-type concentration with optimal fluctuation $O(n^{-1/2})$ which is much stronger since there are typically $n^{\Omega_k(1)}$ types of free trees. \rev{We refer to Theorem~\ref{thm:concentration} for the precise statement.} Further, we construct a delicate coupling in the infinite spin system encoding the clusters, which we call \textit{component coloring} and improve the concentration of depth $2$ neighborhoods to greater distances. Section~\ref{sec:proof} provides a high-level overview of our proof methodology, which we believe has wide applicability to 1\textsc{rsb} class random \textsc{csp}s.

We further remark that the characterization of the local weak limit in the condensation regime $(\alpha_{\cond}, \alpha_{\sat})$ is delicate due to the randomness coming from the weights for the clusters. Indeed, we expect that for a different notion of the local distribution studied in \cite{montanari2012weak}, where a uniformly random vertex is chosen first, and then the marginal of a neighborhood around the vertex is considered, the local weak limit will then be \textit{random}, which is a mixture of extremal Gibbs measures with weights drawn from a Poisson-Dirichlet distribution. Since showing that the relative sizes of the \rev{largest} clusters \rev{(albeit slightly different notion of a cluster, where two solutions are connected if they differ in $O(\log n)$ variables)} follow a Poisson-Dirichlet process in the regime $(\alpha_{\cond}, \alpha_{\sat})$ is open for any r\textsc{csp}'s in 1\textsc{rsb} universality class, we leave this different notion of local weak limit as a conjecture. See Section~\ref{ss:further} for a further discussion.

\subsection{Definitions and Main result} 
\label{subsec:def}
We first define the random regular \textsc{nae-sat} model.  An instance of a $d$-regular $k$-\textsc{nae-sat} problem can be represented by a labeled $(d,k)$-regular bipartite graph as follows. Let $V=\{ v_1, \ldots , v_n \}$ and $F=\{a_1, \ldots, a_m \}$ be the sets of variables and clauses, respectively. Connect $v_i$ and $a_j$ by an edge if the variable $v_i$ participates in the clause $a_j$. Denote this  bipartite graph by $G=(V,F,E)$, and for $e\in E$, let $\tL_e\in \{0,1\}$ denote the literal assigned to the edge $e$. Then, \textsc{nae-sat} instance is defined by $\GG=(V,F,E,\uL)\equiv (V,F,E,\{\tL_{e}\}_{e\in E})$.

For each $e\in E$, we denote the variable (resp. clause) adjacent to it by $v(e)$ (resp. $a(e)$).  Moreover, $\delta v$ (resp. $\delta a$) are the collection of adjacent edges to $v\in V$ (resp. $a \in F$). Then, a $\naesat$ solution is formally defined as follows.

\begin{defn}
		For an integer $l\ge 1$ and $\uz  =(z_i) \in \{0,1\}^l$, define
		\begin{equation}\label{eq:def:INAE}
		I^{\textsc{nae}}(\uz) :=  \mathds{1} \{\uz \textnormal{ is neither identically } 0 \textnormal{ nor }1 \}\,.
		\end{equation} 
		Let $\GG = (V,F,E,\uL)$ be a \textsc{nae-sat} instance. An assignment $\uz \in \{0,1\}^V$ is called a \textbf{solution}  if 
		\begin{equation}\label{eq:def:INAE G}
		I^{\textsc{nae}}(\uz;\GG) :=
		\prod_{a\in F} I^{\textsc{nae}} \big((z_{v(e)} \oplus \tL_e)_{e\in \delta a}\big) =1\,,
		\end{equation}
		where $\oplus$ denotes the addition mod 2. Denote the set of solutions by $\textsf{SOL}(\GG)\subset\{0,1\}^V$.
      \end{defn}

The random regular $\naesat$ instance $\bGG=(V,F,\bE,\buL)$ is then generated by a perfect matching between the set of \textit{half-edges} adjacent to variables and half-edges adjacent to clauses which are labeled from $1$ to $nd=mk$. Thus, $\bE$ is a uniform permutation in $S_{nd}$. Conditioned on $\bE$, the literals $\buL=(\bL_e)_{e\in \bE}$ is drawn i.i.d. from $\Unif(\{0,1\})$. We use the notation $\ubz\sim \Unif(\sol(\bGG))$ for a $\naesat$ solution drawn uniformly at random given a random regular $\naesat$ instance $\bGG$.
    
We next define the distribution in a local neighborhood of $v\in V$ and a depth $t\geq 1$. Hereafter we denote $d(\cdot,\cdot)\equiv d_{\GG}(\cdot,\cdot)$ by the graph distance on the factor graph $\GG$. Let $N_t(v,\GG)$ be the $2t-\frac{3}{2}$ neighborhood around $v$ in $\GG$. That is, $N_t(v,\GG)$ contains the variables $w\in V$ such that the distance $d(v,w)$ in $\GG$ is at most $2t-2$ ($t$ layers of variables including \rev{$v$}), and clauses $a\in F$ such that $d(v,a)\leq 2t-3$ ($t-1$ layers of clauses).

We have used the distance $2t-\frac{3}{2}$ instead of $2t-2$ to include the \textit{boundary} half-edges (half-edges that are not connected within $N_t(v,\GG)$) hanging from the leaf variables $\{w\in V: d(v,w)=2t-2\}$, which will be convenient for the proof. Denote the set of boundary half-edges (resp. full-edges) of $N_t(v,\GG)$ by $\partial N_t(v, \GG)$ (resp. $E_{\sf in}(N_t(v, \GG)$), and let $E(N_t(v,\GG)):= E_{\sf in}(N_t(v,\GG))\sqcup \partial N_t(v,\GG)$. The set of variables (resp. clauses) in $N_t(v,\GG)$ is denoted by $V(N_t(v,\GG))$ (resp. $F(N_t(v,\GG))$).
    
    We take the convention that the full-edges of $e\in E_{\sf in }(N_t(v,\GG))$ store literal information $\tL_e$ while the boundary half-edges $e\in \partial N_t(v,\GG)$ do not. To this end, for $\GG=(V,F,E,\uL)$ and $\uz\in \{0,1\}^V$, denote
    \begin{equation*}
    \uL_{t}(v,\GG):=(\tL_e)_{e\in E_{\sf in}(N_t(v,\GG))}\,,\quad \uz_{t}(v,\GG):=(z_v)_{v\in V(N_t(v,\GG))}\,.
    \end{equation*}
    Note that if $N_t(v,\GG)$ does not contain a cycle, it is isomorphic to $(d,k)$-regular tree. Denote the infinitary $(d,k)$ regular factor tree rooted at a variable $\rho$ by $\TTT_{d,k}\equiv \TTT_{d,k}(\rho)$. Here, we consider $\rho$ as a variable and its $d$ descendants as clauses. Similarly, all the clauses have \rev{$k-1$} descendant variables. Thus, the variables are located at even depths whereas the clauses are located at odd depths. Then, $\TTT_{d,k,t}$ is defined by \rev{the} $2t-\frac{3}{2}$ neighborhood around the root $\rho$ in $\TTT_{d,k}$ (see Figure \ref{fig:tree}). We use the notation $V(\TTT_{d,k,t}), F(\TTT_{d,k,t}), E_{\sf in}(\TTT_{d,k,t}),$ and $\partial \TTT_{d,k,t}$ respectively for the set of variables, clauses, full-edges, and boundary half-edges of $\TTT_{d,k,t}$.

    \begin{figure}[t]
    \centering
		\begin{tikzpicture}[square/.style={regular polygon,regular polygon sides=4},thick,scale=0.64, every node/.style={transform shape}]
		\node[circle,draw, scale=0.7, fill=black] (v1) at (0,0) {};
		\node[square,draw, scale=0.7] (a1) at (-1.5,-1.3) {};
		\node[square,draw, scale=0.7] (a2) at (0,-1.3) {};
		\node[square,draw, scale=0.7] (a3) at (1.5,-1.3) {};
         \node[circle ,draw, scale=0.7, fill=black] (v2) at (-3,-2.6) {};
		\node[circle ,draw, scale=0.7, fill=black] (v3) at (-1.7,-2.6) {};
        \node[circle ,draw, scale=0.7, fill=black] (v4) at (-0.6,-2.6) {};
        \node[circle ,draw, scale=0.7, fill=black] (v5) at (0.6,-2.6) {};
		\node[circle ,draw, scale=0.7, fill=black] (v6) at (1.7,-2.6) {};
        \node[circle ,draw, scale=0.7, fill=black] (v7) at (3,-2.6) {};
		\draw (v1)--(a1);
		\draw (v1)--(a2);
		\draw (v1)--(a3);
        \draw (a1)-- (v2);
        \draw (a1) -- (v3);
        \draw (a2)-- (v4);
        \draw (a2) -- (v5);
        \draw (a3) -- (v6);
        \draw (a3) -- (v7);
		\draw[draw=blue] (v2)--(-3.75,-3.3);
		\draw[draw=blue] (v3)--(-2.25,-3.3);
		\draw[draw=blue] (v4)--(-0.9,-3.3);
		\draw[draw=blue] (v5)--(0.3,-3.3);
		\draw[draw=blue] (v6)--(1.5,-3.3);
		\draw[draw=blue] (v7)--(3,-3.3);
        \draw[draw=blue] (v2)--(-3,-3.3);
		\draw[draw=blue] (v3)--(-1.5,-3.3);
		\draw[draw=blue] (v4)--(-0.3,-3.3);
		\draw[draw=blue] (v5)--(0.9,-3.3);
		\draw[draw=blue] (v6)--(2.25,-3.3);
		\draw[draw=blue] (v7)--(3.75,-3.3);
		\end{tikzpicture}
     \vspace{3mm}
		\caption{$\TTT_{d,k,t}$ for $d=k=3$ and $t=2$. Variables and clauses are drawn by the circular and square nodes, respectively. The boundary half-edges in $\partial \TTT_{d,k,t}$ are highlighted in blue. }\label{fig:tree}
	\end{figure}
 
    For $(\uz_t,\uL_t) \in \{0,1\}^{V(\TTT_{d,k,t})}\times \{0,1\}^{E_{\sf in} (\TTT_{d,k,t})}$ and $v\in V$, define the random variable $X_v^t \equiv X_v^t [\bGG,\ubz, \uz_t,\uL_t]$ by
    \begin{equation*}
        X_v^t \equiv X_v^t [\bGG,\ubz, \uz_t,\uL_t] :=\one\{N_t(v,\bGG)\textnormal{ is a tree, ~$\buL_{t}(v,\bGG)=\uL_t$,~ and }~~\ubz_{t}(v, \bGG)=\uz_t\}\;.
    \end{equation*}
Then the depth $t$ empirical distribution is given by
\[
\P_n^t[\bGG, \ubz](\uz_t,\uL_t)=\frac{1}{n}\sum_{v\in V}X_v^t\,.
\]

Note that $\P_n^t[\bGG,\ubz]$ is a random measure on \rev{$\{0,1\}^{V(\TTT_{d,k,t})}\times \{0,1\}^{E_{\sf in} (\TTT_{d,k,t})}$}. The total mass is $1-O(n^{-1})$ with high probability because of rare neighborhoods which contain a cycle. Our main theorem below determines the limit of $\P_n^t[\bGG,\ubz]$ with $O(n^{-1/2})$ fluctuation, thus such cyclic neighborhoods may be neglected. \rev{Throughout, $\PPP(A)$ denotes the set of probability measures on a measurable space $A$.}
    \begin{thm}\label{thm:main}
    For $k\geq k_0$ and $\alpha \in (\alpha_{\cond}(k),\alpha_{\sat}(k))$, let $\boldsymbol{\mathcal{G}}$ be a random regular $\knaesat$ instance. Given $\bGG$, draw a uniformly random $\naesat$ solution $\ubz \sim \Unif(\sol(\bGG))$. For any $t\geq 1$, there is an explicit non-random \rev{$\PP^\star_t\equiv \PP^\star_t[\alpha,k]\in \PPP\big(\{0,1\}^{V(\TTT_{d,k,t})}\times \{0,1\}^{E_{\sf in} (\TTT_{d,k,t})}\big)$, which is} defined in Section~\ref{s:lwl}, such that the following holds. For any $\eps>0$, there exists $C\equiv C(\eps,\alpha,k,t)>0$ such that we have with probability at least $1-\eps$ over $(\bGG, \ubz)$,
    \begin{equation}\label{eq:thm:main}
      d_{\tv}\Big(\P_n^{t}[\bGG, \ubz\,]\;,\;\PP_\star^t\Big)\leq \frac{C}{\sqrt{n}}\,.
    \end{equation}
    \end{thm}
    \begin{remark}\label{rmk:alpha:regime}
    For the random regular $k$-\textsc{nae-sat} model for $k\geq k_0$, the thresholds $\alpha_{\sat}(k)$ and $\alpha_{\cond}(k)$ were established respectively in \cite{dss16} and \cite{ssz22}. Further, they showed that the set  $(\alpha_{\cond}(k),\alpha_{\sat}(k))$ is contained in $[\alpha_{\textsf{lbd}}, \alpha_{\textsf{ubd}}]$, where $\alpha_{\textsf{lbd}} \equiv (2^{k-1}-2)\log 2$ and $\alpha_{\textsf{ubd}}\equiv 2^{k-1}\log 2$. Thus, we assume that $\alpha\in [\alpha_{\textsf{lbd}}, \alpha_{\textsf{ubd}}]$ throughout the paper. We also remark that the large $k$ expansion of $\alpha_{\cond}(k)$ and $\alpha_{\sat}(k)$ \cite{ssz22} is given by 
    \[
   \alpha_{\cond}(k)= 2^{k-1}\log 2 - \log 2 + o_k(1)\,,\quad \alpha_{\sat}(k)= 2^{k-1}\log 2 - \frac12\log 2 - \frac14 + o_k(1)\,.
    \]
    The quantities $\alpha_{\cond}(k)$ and $\alpha_{\sat}(k)$ are formally defined in equation \eqref{eq:def:alpha:cond:sat} below.
    \end{remark}
 We emphasize that the fluctuation $\frac{C}{\sqrt{n}}$ is optimal as local neighborhood frequencies have central limit theorem fluctuations.\footnote{Our proof method also shows that there exists a constant $C^\prime(\eps,\alpha,k,t)$ such that $d_{\tv}\big(\P_n^{t}[\bGG, \ubz\,],\PP_\star^t\big)\geq C^\prime n^{-1/2}$ with probability $1-\eps$. This lower bound follows straightforwardly from our analysis by considering fluctuation that comes from variables that are in a free tree with a single free variable (see Section \ref{subsec:proof:freecomp} for the definition). Thus, we focus on achieving a tight upper bound.} In particular, it is arguably much stronger than the asymptotic statement, where one only specifies the limit $\PP^\star_t$, and not the rate $O(n^{-1/2})$ in \eqref{eq:thm:main}. For example, Theorem \ref{thm:main} imply that $d_{\tv}\big(\P_n^{t}[\bGG, \ubz\,]\;,\;\PP_\star^t\big)\leq a_n/\sqrt{n}$ with probability $1-o_n(1)$, for any $a_n$ which diverges to $\infty$.

\subsection{Related works}

Local weak convergence of graphs, also known as Benjamini-Schramm convergence, gives a way of describing the local distributional limit of a sequence of (possibly random) graphs.  It was developed independently by Aldous~\cite{aldous2001zeta} to study the Assignment Problem and by Benjamini and Schramm~\cite{benjamini2001recurrence} to study recurrence of random walks on planar graphs. 

The notion of local weak convergence generalizes naturally to the graphs that are labeled by a spin system, which we study in this work, and it is conjectured that many global properties of the spin systems are determined by the local weak limit. Indeed, in the context of r\textsc{csp}s, statistical physics predictions~\cite{zdeborova2007phase,montanari2008clusters,kmrsz07} that describe a series of phase transitions r\textsc{csp} in the 1\textsc{rsb} universality class are based on the local weak limit. The shattering and freezing thresholds are described explicitly in terms of transitions in the behavior of the broadcast model, which corresponds to the local weak limit in this regime. The condensation corresponds to the point at which the local weak limit stops being given by the simple broadcasting model description.

Earlier mathematical literature on local weak limits for r\textsc{csp}s was centered around understanding the shattering and freezing thresholds.  The shattering threshold is conjectured to coincide with the reconstruction threshold of the local weak limit, which asks whether the spin at the root is dependent on asymptotically far away vertices.  Several results relate tree thresholds to the analogous graphs.  Gerschenfeld and Montanari~\cite{gerschenfeld2007reconstruction} showed that on locally treelike graphs, reconstruction on the graph is equivalent to that of the tree if two independent solutions have approximately independent empirical distributions. Montanari,  Restrepo, and Tetali~\cite{montanari2011reconstruction} extended this to a wider class of r\textsc{csp}s. Later, Coja-Oghlan, Efthymiou, and Jaafari established the local weak convergence in the $k$-colouring model up to the condensation threshold for large enough $k$ in~\cite{coja2018local}. On the other hand, the freezing threshold for coloring models was established by Molloy ~\cite{molloy2018freezing} and for a wider class of models in~\cite{mr13}. 

As noted above, studying the planted model has been a powerful tool to study the local weak limit of r\textsc{csp}s below condensation. A spectacular application of this method was by Achlioptas and  Coja-Oghlan, who established regimes of clustering for solutions for the coloring, $k$-\textsc{sat} and \textsc{nae-sat}~\cite{achlioptas2008algorithmic}.  This was established via the second moment method which shows that most graphs have similar numbers of solutions.  A more refined picture can be established by Robinson and Wormald's small graph conditioning method~\cite{robinson1994almost} to show that the planted model is contiguous with respect to the original distribution~\cite{janson1995random}.

Another setting where local weak convergence has played a key role is in the study of the stochastic block model.  This is a model of inhomogeneous graphs that contains communities and is an important test-bed for the statistical theory of community detection in networks. Understanding the local weak limit of the \rev{stochastic block model} has led to sharp information-theoretic bounds on when the detection of communities is possible~\cite{mossel2015reconstruction,MSS22}.

\subsection{Clusters}
\label{subsec:clusters}
\rev{Recall that a cluster is defined by the connected component of the solution space where two solutions are connected if they differ in one variable.} A central role in understanding the $\naesat$ model and sparse r\textsc{csp}s, in general, is to study how the space of solutions splits into small rigid clusters.  In a typical solution, a small but constant fraction of variables can be flipped between $0$ and $1$ without violating any constraints giving rise to exponentially many nearby solutions.  In order to give a combinatorial definition of a cluster, the so-called \emph{coarsening} algorithm inductively maps variables taking values in $\{0,1\}$ to $\ff$ free if they can be flipped without violating any constraints \cite{Parisi02, dss16}.  A constraint is considered satisfied if one of its variables is free and the algorithm continues until no more variables can be set to $\ff$ resulting in a $\{0,1,\ff\}$ valued configuration called a \emph{frozen configuration}.  Every valid frozen configuration satisfies the following properties.

      \begin{defn}[Coarsening and Frozen configuration]\label{def:frozenconfig}
      Given a $\naesat$ instance $\GG=(V,F,E, \uL)$ and a $\naesat$ solution $\uz \in \sol(\GG)$, the \textbf{coarsening} $\ux\in \{0,1,\ff\}^V$ of $\uz$ is defined by the following algorithm. For each variable $v\in V$, whenever the value of $z_v\in \{0,1\}$ can be flipped between $0/1$ without violating any constraints, change $z_v$ to $\ff$. Iterate until no more variable can be set to $\ff$. Note that the resulting configuration must satisfy the following. We call $\ux\in\{0,1,\ff \}^V$ a (valid) \textbf{frozen configuration} if it satisfies
		\begin{itemize}
			\item No \textsc{nae-sat} constraints are violated for $\ux$. That is, $I^{\textsc{nae}}(\ux;\GGG)=1$.
			
			\item For $v\in V$, $x_v\in \{0,1\}$ if and only if it is forced to be so. That is, $x_v\in \{0,1\}$ if and only if there exists $e\in \delta v$ such that $a(e)$ becomes violated if $\tL_e$ is negated, i.e., $I^{\textsc{nae}} (\ux; \GGG\oplus\one_e )=0$ where $\GGG\oplus\one_e$ denotes $\GGG$ with $\tL_e$ flipped. Conversely, $x_v=\ff$ if and only if no such $e\in \delta v$ exists.
		\end{itemize}
	\end{defn}
\rev{The term `frozen configuration' originates from the second item: $x_v = z\in \{0,1\}$ if and only if they are frozen to be $z$, i.e. cannot be flipped to $1-z$ without violating a clause. For this reason, $v$ is called a frozen variable (w.r.t. $\ux\in \{0,1,\ff\}^V$) if $x_v\in \{0,1\}$.
}

Frozen model solutions are themselves a random \textsc{csp} but without clusters because they are, in effect, clusters of solutions projected to a single point. \cite{dss16} showed that for $\alpha\in [\alpha_{\textsf{lbd}}, \alpha_{\textsf{ubd}}]$, with high probability all solutions map via coarsening algorithm to frozen configurations with a low density (less than $7/2^k$) of free variables. \rev{Thus, intuitively, the connected components of free variables form small subcritical branching trees. Such intuition was confirmed rigorously in the works~\cite{ssz22, NSS, nss2}. In particular, \cite{NSS} showed that with high probability, there is a one-to-one correspondence between all the clusters -- except those of negligible size -- and the frozen configurations with a low density of free variables and at most one cycle in their connected components. For the precise definition of the connected components of free variables, we refer to the definition of \textit{free components} in Section~\ref{subsec:proof:freecomp}.}

An alternative definition of the cluster model is in terms of fixed points of the Belief Propagation (BP) equations.  On each directed edge $e$ we define a pair of messages $\mm_e\equiv (\dot{\mm}_e,\hat{\mm}_e)$ taking values in the set of probability distributions on $\{0,1\}$.  The messages satisfy the BP equations if
\[
	\dot{\mm}_e(z)
= \frac{\prod_{e'\in\delta v(e)\setminus e} \hat{\mm}_{e'}(z)}
{\sum_{z^\prime \in\{0,1\}}\prod_{e'\in\delta v(e)\setminus e} \hat{\mm}_{e'}(z^\prime)}\,,
\]
and
\[
\hat{\mm}_e(z) = \frac{\sum_{\underline{z}_{\delta a(e)} \in\{0,1\}^d} \rev{\one}(z_e=z) I^\textsc{nae}((\uz\oplus L)_{\delta a(e)}) \prod_{e'\in\delta a(e)\setminus e} \dot{\mm}_{e'}(z_{e'})}
{\sum_{\underline{z}_{\delta a(e)} \in\{0,1\}^d} I^\textsc{nae}((\uz\oplus L)_{\delta a(e)}) \prod_{e'\in\delta a(e)\setminus e} \dot{\mm}_{e'}(z_{e'})}\,.
\]
The interpretation of $\dot{\mm}_e(z)$ is the probability that \rev{$v(e)$} is equal to $z$ in a random solution in that cluster after the edge $e$ is removed.  Frozen variables correspond to those that have at least one incoming message $\hat{\mm}$ that is a point mass. A solution to the BP equations $\underline{\mm}\equiv (\mm_e)_{e\in E}$ can be arrived from a $\naesat$ solution $\underline{z}$ by starting with $\dot{\mm}_{e}(y)=I(z_{v(e)}=y)$ for $e\in E$ and $y\in \{0,1\}$ and then iteratively calculating $\hat{\mm}$ from $\dot{\mm}$ \rev{around every clause,} and then $\dot{\mm}$ from $\hat{\mm}$ \rev{around every variable. With high probability over $(\bGG,\buz)$, where $\ubz\sim \Unif(\sol(\bGG))$, these iterations converge in $O(\log n)$ steps. This fact, though not subsequently employed, follows from \cite{NSS} which shows that typically, the largest connected components of free variables have size $O(\log n)$.}

The number of solutions in a configuration corresponding to a cluster $\ux$, or equivalently $\underline{\mm}$, is given by
\[
	\size(\underline{\mm}, \GG)
	=\prod_{v\in V}
	\dot{\varphi}(\hat{\mm}_{\delta v})
	\prod_{e\in E}\bar{\varphi}(\dot{\mm}_e,\hat{\mm}_e)\prod_{a\in F}
	\hat{\varphi}^{\lit}((\rev{\dot{\mm}}
	\oplus L)_{\delta a})\,.
 \]
for explicit functions $(\dot{\varphi},\bar{\varphi},\hat{\varphi}^{\lit})$ defined in equation (28) of \cite{ssz22} or equation \eqref{eq:def:phi} below. The challenge in the condensation regime is that the typical number of solutions is much smaller than the expected number of solutions.  This is because the largest contribution to the expected value comes from rare clusters, which are large. \rev{However,} it is exactly these clusters whose local distribution behaves like the planted model. Their absence in a typical realization results in a different empirical distribution.  Instead, following the physics heuristics of~\cite{kmrsz07} implemented rigorously in~\cite{ssz22,NSS}, we weight frozen model configurations according to $\size(\underline{\mm}, \GG)^\lambda$ and tune $\lambda$ to give clusters that correspond to the largest size that appears.  To this end, we define the following measure-valued functions. For $\mu \in \PPP([0,1])$, let
\begin{align*}
	\hat{\RR}_\lambda\mu(B) &\equiv 
	\frac1{\hat{\mathscr{Z}}(\mu)}
		\int \bigg(2-\prod_{i=1}^{k-1}x_{i}-\prod_{i=1}^{k-1}(1-x_{i})\bigg)^{\lambda}
	\rev{\one}\bigg\{
			\frac{1-\prod_{i=1}^{k-1}x_{i}}
			{2- \prod_{i=1}^{k-1}x_{i} - \prod_{i=1}^{k-1}(1-x_{i})}
		\in B\bigg\}
		\,
		\prod_{i=1}^{k-1}{\mu}(\de x_{i})\,,\\ 
	\dot{\RR}_\lambda\mu(B) &\equiv
	\frac1{\dot{\mathscr{Z}}(\mu)}
		\int\bigg(\prod_{i=1}^{d-1}y_{i}+\prod_{i=1}^{d-1}(1-y_{i})\bigg)^{\lambda}
		\rev{\one}\bigg\{
			\frac{\prod_{i=1}^{d-1}y_{i}}
			{\prod_{i=1}^{d-1}y_{i}+\prod_{i=1}^{d-1}(1-y_{i})}
		\in B
		\bigg\}
		\,
		\prod_{i=1}^{d-1}\mu(\de y_{i})\,,
\end{align*}
where $\hat{\mathscr{Z}}(\mu)$ and $\dot{\mathscr{Z}}(\mu)$ are normalizing constants to make $\hat{\RR}_\lambda\mu$ and $\dot{\RR}_\lambda\mu$ a probability measure. Denote $\RR_{\la} \equiv \dot{\RR}_{\la}\circ \hat{\RR}_{\la}:\PPP([0,1])\to \PPP([0,1])$. The fixed point of $\RR_{\la}$ was established in~\cite{ssz22}.
\begin{prop}[Proposition 1.2 in \cite{ssz22}]\label{prop:ssz:physics:bp:fixed}
Fix $k\geq k_0$. For $\la \in [0,1]$ and $\alpha \in [\alpha_{\textsf{lbd}}, \alpha_{\textsf{ubd}}]$, the following holds. Let $\dot{\mu}_{\lambda,0}=\frac12\delta_0 + \frac12 \delta_1$ and for $t\geq 0$, recursively set $\dot{\mu}_{\lambda,t+1}=\RR_\lambda \dot{\mu}_{\lambda,t}$. \rev{As $t\to\infty$,} $\dot{\mu}_{\la,t}$ converges to $\dot{\mu}_{\la}\in \PPP([0,1])$ in total variation distance. Moreover, $\dot{\mu}_{\la}\equiv \dot{\mu}_{\la}[\alpha,k]$ satisfies $\dot{\mu}_{\la}(dx)=\dot{\mu}_{\la}(d(1-x))$ and $\dot{\mu}_{\la}((0,1))\leq 7/2^k$.
\end{prop}
Denote $\hat{\mu}_{\la}\equiv \hat{\RR}\dot{\mu}_{\la}$. Define the measures $\dot{w}_{\la}, \hat{w}_{\la}, \bar{w}_{\la} \in \PPP([0,1])$ by
\begin{align*}
		\dot{w}_{\lambda}(B)	
		&= ({\dot{\ZZZ}}_{\lambda}^\prime)^{-1}
		\int\bigg(\prod_{i=1}^{d}y_{i}+\prod_{i=1}^{d}(1-y_{i})\bigg)^{\lambda}
		 \rev{\one}\bigg\{
		 \prod_{i=1}^{d}y_{i}
		 +\prod_{i=1}^{d}(1-y_{i})\in B
		 \bigg\}
		 \prod_{i=1}^{d}\hat{\mu}_{\lambda}(\de y_{i})\,,\\
		\hat{w}_{\lambda}(B)	
		&= 
		({\hat{\ZZZ}}_{\lambda}^\prime)^{-1}
		\int\bigg(1-\prod_{i=1}^{k}x_{i}-\prod_{i=1}^{k}(1-x_{i})\bigg)^{\lambda}
			\rev{\one}\bigg\{
			1-\prod_{i=1}^{k}x_{i}-\prod_{i=1}^{k}(1-x_{i})\in B
			\bigg\}
			\prod_{i=1}^{k}\dot{\mu}_\lambda(\de x_{i})\,,\\
		\bar{w}_{\lambda}(B)	
		&=
		({\bar{\ZZZ}}_{\lambda}^\prime)^{-1}
			\iint\bigg(xy+(1-x)(1-y)\bigg)^{\lambda}
			\rev{\one}\Big\{
			xy+(1-x)(1-y)\in B
			\Big\}
			\dot{\mu}_\lambda(dx)\hat{\mu}_{\lambda}(\de y)\,,\\
\end{align*}
where $\dot{\ZZZ}_{\la}^\prime ,\hat{\ZZZ}_{\la}^\prime ,\bar{\ZZZ}_{\la}^\prime $ are normalizing constants. Let
\begin{align*}
\mathfrak{F}(\lambda;\alpha)&= \log \dot{\ZZZ}_\lambda + \alpha \log \hat{\ZZZ}_\lambda - \alpha k \log \bar{\ZZZ}_{\lambda},\\
s(\lambda;\alpha)&=\int \log(x) \dot{w}_{\lambda}(\de x) + \alpha\int \log(x) \hat{w}_{\lambda}(\de x) - \alpha k \int \log(x) \bar{w}_{\lambda}(\de x).
\end{align*}
Then, $\alpha_{\cond}\equiv \alpha_{\cond}(k)$ and $\alpha_{\sat}\equiv \alpha_{\sat}(k)$ are defined by
\begin{equation}\label{eq:def:alpha:cond:sat}
\alpha_{\cond}\equiv \sup\Big\{\alpha \in [\alpha_{\textsf{lbd}},\alpha_{\textsf{ubd}}]:\mathfrak{F}(1;\alpha)> s(1;\alpha)\Big\}\,,\quad\alpha_{\sat}\equiv \sup\Big\{\alpha \in [\alpha_{\textsf{lbd}},\alpha_{\textsf{ubd}}]:\mathfrak{F}(0;\alpha)>0\Big\}\,,
\end{equation}
\rev{where we recall that $\alpha_{\textsf{lbd}} \equiv (2^{k-1}-2)\log 2$ and $\alpha_{\textsf{ubd}}\equiv 2^{k-1}\log 2$ from Remark~\ref{rmk:alpha:regime}.} It was shown in \cite[Proposition 1.4]{ssz22} that $\alpha_{\cond}(k)$ and $\alpha_{\sat}(k)$ \rev{are} well defined for $k\geq k_0$. Then, for $\alpha \in (\alpha_{\cond},\alpha_{\sat})$, we set $\lambda^\star \equiv \la^\star[\alpha,k]$ and $s^\star\equiv s^\star[\alpha,k]$ so that
\begin{equation}\label{eq:opt:la}
\lambda^\star :=\sup\{\lambda\in[0,1]:\mathfrak{F}(\lambda;\alpha)-\lambda s(\lambda;\alpha)>0\}\,\rev{,}\quad\quad s^\star:=s(\la^\star;\alpha).
\end{equation}
We remark that \cite{NSS, nss2} proved that in the condensation regime $\alpha \in (\alpha_{\cond},\alpha_{\sat})$, both the number of solutions and the number of solutions in the largest cluster are of size
$\Theta\big(\exp\big(ns^\star-\frac{1}{2\la^\star}\log n\big)\big)$.

\subsection{Local weak limit}\label{s:lwl}

We now specify the distribution over solution\rev{s} given the literals $\uL_t\in \{0,1\}^{E_{\sf in}(\TTT_{d,k,t})}$, where we recall that $\TTT_{d,k}$ is the infinite $(d,k)$ regular factor tree rooted at a variable $\rho$, and $\TTT_{d,k,t}$ is the sub-tree of $\TTT_{d,k}$ up to depth $2t-\frac{3}{2}$. First, we choose a random cluster in terms of its BP messages $\mm=(\dot{\mm},\hat{\mm})$.  Note that if we set the incoming messages $\hat{\mm}_e$ at the boundary edges $e\in \partial \TTT_{d,k,t}$, then there is a unique extension to the internal edges  $E_{\sf in}(\TTT_{d,k,t})$ solving the BP equations, which gives $\underline{\mm}_t \equiv (\mm_e)_{e\in E(\TTT_{d,k,t})}$. With abuse of notation, identify $\dot{\mm}_e, \hat{\mm}_e \in \PPP(\{0,1\})$ with $\dot{\mm}_e(1), \hat{\mm}_e(1) \in [0,1]$. For $\uL_t\in \{0,1\}^{E_{\sf in}(\TTT_{d,k,t})}$, we assign the weight
\[
\nu_{\la}(\de \underline{\mm}_t ; \uL_t)= (\mathcal{Z}_{\la})^{-1} \bigg( \prod_{v\in V(\TTT_{d,k,t})}
	\dot{\varphi}(\rev{\underline{\hat{\mm}}_{\delta v}})
	\prod_{e\in E_{\sf in}(\TTT_{d,k,t})}\bar{\varphi}(\dot{\mm}_e,\hat{\mm}_e)	
	\prod_{a\in F(\TTT_{d,k,t})}
	\hat{\varphi}^{\lit}(\rev{(\underline{\dot{\mm}}
    	\oplus \uL)_{\delta a}}) \bigg)^\lambda \prod_{e\in \partial \TTT_{d,k,t}} \hat{\mu}_\lambda(\de \hat{\mm}_e)\,,
\]
where the normalization constant $\mathcal{Z}_{\la}$ is given by
\begin{equation*}
    \mathcal{Z}_{\la}=\int \bigg( \prod_{v\in V(\TTT_{d,k,t})}
	\dot{\varphi}(\rev{\underline{\hat{\mm}}_{\delta v}})
	\prod_{e\in E_{\sf in}(\TTT_{d,k,t})}\varphi(\dot{\mm}_e,\hat{\mm}_e)	
	\prod_{a\in F(\TTT_{d,k,t})}
	\hat{\varphi}^{\lit}(\rev{(\underline{\dot{\mm}}
    	\oplus \uL)_{\delta a}}) \bigg)^\lambda \prod_{e\in \partial \TTT_{d,k,t}} \hat{\mu}_\lambda(\de \hat{\mm}_e)\,.
\end{equation*}
Given $\underline{\mm}_t$, let $\ux(\underline{\mm}_t)\equiv (x_v)_{v\in V(\TTT_{d,k,t})}$ be the frozen configuration associated with $\underline{\mm}_t$. That is, if there exists $e\in \delta v$ such that $\hat{\mm}_e=\delta_{z}$ for some $z \in \{0,1\}$, then set $x_v=z$, and otherwise, set $x_v=\ff$. For $\uz_t\in \{0,1\}^{V(\TTT_{d,k,t})}$, we write $\uz_t\sim_{\uL_t} \ux(\underline{\mm}_t)$ if $\uz_v=\ux_v$ whenever $\ux_v\in\{0,1\}$ and $\uz_t$ is a valid $\naesat$ configuration for literals $\uL_t$.  In other words, $\uz_t$ is a valid assignment of the spins in the free variables of $\ux(\underline{\mm}_t)$. For $\la^\star\equiv \la^\star[\alpha,k]$ in \eqref{eq:opt:la}, define the probability measure $\PP_{\star}^t\equiv \PP_{\star}^t[\alpha,k]\in \PPP\big( \{0,1\}^{V(\TTT_{d,k,t})}\times \{0,1\}^{E_{\sf in }(\TTT_{d,k,t})}\big)$ by
\begin{equation}\label{eq:def:P:star}
\PP_\star^t(\uz_t,\uL_t)= 2^{-|E_{\sf in} (\TTT_{d,k,t})|} \int \frac{ I\big(\uz_t\sim_{\uL_t} \ux(\rev{\underline{\mm}_t})\big) \prod_{e\in \partial \TTT_{d,k,t}} \hat{\mm}_e(\uz_{v(e)}) }{\sum_{\uz'_t} I\big(\uz'_t\sim_{\uL_t} \ux(\rev{\underline{\mm}_t})\big) \prod_{e\in \partial \TTT_{d,k,t}} \hat{\mm}_e(\uz'_{v(e)}) }\nu_{\la^\star}(\de \underline{\mm}_t;\uL_t)\,.
\end{equation}
This construction picks a frozen configuration $\ux$ according to the $\lambda^\star$-weighted measure and then picks a random solution properly weighted by the effect of the $\ux$ outside of the neighborhood.

\subsection{Further discussion}\label{ss:further}
Having established the local weak limit $\PP_{\star}^t\equiv \PP_{\star}^t[\alpha,k]$ for $\alpha \in (\alpha_{\cond},\alpha_{\sat})$, natural questions arise: is $\PP_{\star}^t[\alpha,k]$ a Gibbs measure? Can $\PP_{\star}^t[\alpha,k]$ be described in a Markovian fashion? 

As one might expect, the answer to the first question is yes in a rather simple manner. Note that given $\underline{\mm}_t$, the integrand in equation \eqref{eq:def:P:star} defines a Gibbs measure over $\uz_t$ with BP messages $(\hat{\mm}_e,\dot{\mm}_e)_{e\in E(\TTT_{d,k,t})}$. Thus, $\PP_{\star}^t$ is a mixture of such measures, which is again a Gibbs measure. Furthermore, let $F_0\subset F(\TTT_{d,k,t})$ be a subset of clauses and $V_0:= \{v\in V(\TTT_{d,k,t}):d(F_0,v)=1\}$ be the variables adjacent to them. Denote the boundary variables in $V_0$ by $\partial V_0:= \{v\in V_0:d(\rho,v)=2t-2\textnormal{ or }d(F_0^{\sf c}, v)=1\}$. Then, conditional on $\uz_{\partial V_0\sqcup V_0^{\sf c}}$ and $\uL_t$, it follows from the definition that $\PP_{\star}^t(\cdot\given \uz_{\partial V_0\sqcup V_0^{\sf c}}, \uL_t)$ is simply a uniform measure over the $\naesat$ solutions in $V_0\cup F_0$.\footnote{Here, note that for two $\naesat$ solutions $\uz_t, \uz_t^\prime$ such that $\uz_{\partial V_0\sqcup V_0^{\sf c}}=\uz_{\partial V_0\sqcup V_0^{\sf c}}^\prime$, $\uz_t \sim_{\uL_t}\ux(\mm_t)$ if and only if $\uz_t^\prime \sim _{\uL_t} \ux(\mm_t)$.} Thus, in this sense, all the interesting aspects of $\PP_{\star}^t$ come from the boundary conditions $(\hat{\mm}_e)_{e\in \partial\TTT_{d,k,t}}$.

 For the second question, the measure $\PP^t_{\star}$ is non-Markovian. \rev{Here, we emphasize that the non-Markovian property only applies to the limit of the joint law $(\ubz_t,\buL_t)\sim \P_n^{t}[\bGG,\ubz]$. Indeed, if we look only at the marginal $\ubz_t$, then for \emph{any} choice of $\alpha$, it is simply the product measure. Thus, the local weak limit is non-trivial only when the joint law $(\ubz_t,\buL_t)$ is considered. 
}

 Let us first show \rev{that $\PP^t_{\star}$ is non-Markovian} in the limiting case of $t=\infty$.  Note that the BP messages induced by $\uz$ are measurable with respect to $\uz$.  We condition on $z_\rho=1$ and on the literals $\uL_t$ and edges around the root $e\in \delta \rho$ satisfies $\hat{\mm}_e[1] \in \{0, \frac{1}{2}\}$. Assume it was Markovian.  Then, conditional on the root, the messages to the root from each subtree are independent.
Let $Y_i$ be the indicator that the $i$-th edge of $\rho$ is forcing.  Then taking a ball of radius 1 around the root, if any of the clauses are forcing, then $\mathcal{Z}_{\la}=1$ while if they are all separating, then the root is a free singleton and $\mathcal{Z}_{\la}=2^\lambda$.  Hence we have that with $p=\frac{\hat{\mu}_{\la^\star}(0)}{\hat{\mu}_{\la^\star}(0)+ \hat{\mu}_{\la^\star}(1/2)}$,
\[
\P(Y_{\partial \rho} = y_{\partial \rho}) = \frac1{z} p^{\sum_{i}y_i}(1-p)^{d-\sum_{i}y_i} 2^{(\lambda^\star-1)I(y\equiv 0)}\,,
\]
and so it is not a product measure.  The factor of $2^{(\lambda^{\star}-1)}$ comes from $\mathcal{Z}_{\la}$ and the probability $2^{-1}$ of $z_\rho=1$. As the measure is non-Markovian for $t=\infty$, it must also be non-Markovian for some large fixed $t$.

\rev{From a high-level, the peculiar phenomenon of the non-Markovian nature of $\PP^t_{\star}$ may be explained as follows. Recall that in the definition of $\PP^t_{\star}$ in Section~\ref{s:lwl}, we first defined the distribution $\nu_{\la}$ of the BP messages $\underline{\textbf{\mm}}_t \equiv (\textbf{\mm}_e)_{e\in E(\TTT_{d,k,t})}$, and then $(\ubz_t,\buL_t)\sim \PP_{\star}^t$ was drawn according to a weighting scheme concerning the weights $\underline{\textbf{\mm}}_t$. Here, if we consider the \textit{joint distribution} of $(\ubz_t,\buL_t, \underline{\textbf{\mm}}_t)$, then the limit is Markovian. That is, $(\ubz_t,\buL_t)$ needs to be augmented with the BP messages $\underline{\textbf{\mm}}_t$ in order to be Markovian. However, the BP messages $\underline{\textbf{\mm}}_t$ depend on an extensive range of $(\ubz_{\infty},\buL_{\infty})$. Therefore, $(\ubz_t,\buL_t)$ alone, without the BP messages, is non-Markovian, as functions of a Markov chain generally do not preserve the Markov property. 
}

We remark that our description of the local weak limit holds even below $\alpha_{\cond}$. Note that $\dot{\mu}_{\la}\equiv \dot{\mu}_{\la}[\alpha,k]$ in Proposition \ref{prop:ssz:physics:bp:fixed} is well-defined even for $\alpha \in [\alpha_{\textsf{lbd}},\alpha_{\cond}]$. Thus, in such regime, $\PP_{\star}^t[\alpha,k]$ is well-defined with equation \eqref{eq:def:P:star}, where $\lambda^{\star}[\alpha,k]\equiv 1$ for $\alpha\leq \alpha_{\cond}$, and (a modification) of our proof shows that Theorem \ref{thm:main} holds for $\alpha \in [\alpha_{\textsf{lbd}},\alpha_{\cond}]$. However, when $\la^{\star}=1$, it can be shown that $\PP_{\star}^t$ is just a uniform measure over $(\uz_t,\uL_t)$, which is a $\naesat$ solution and can be described by a simple broadcast model. This coincides with the description of the local weak limit obtained in \cite{coja2020replica} for $\alpha\in (0,\alpha_{\cond})$, thus below condensation, our method is just a more complicated way of determining the local weak limit.

We further remark that there are two different notions of local weak convergence of a solution ~\cite{montanari2012weak}.  In the terminology of~\cite{montanari2012weak}, our result is called \emph{convergence locally on average} as we have taken the empirical distribution averaged over all the vertices of the graph.  A stronger notion is \emph{convergence in probability locally} which asks that at almost all fixed variables $i$, the distribution of solutions in a ball of radius $t$ around $i$ converges.  This stronger notion was proved by~\cite{coja2020replica} for $\alpha<\alpha_{\cond}$.  In contrast, it is not true for $\alpha>\alpha_{\cond}$ because the local distribution is itself random depending on the clusters and their relative weights. To be more precise, let $\widetilde{\P}_{v,t}[\GG]$ denote the distribution of $(\buz_t(v,\GG), \uL_t(v,\GG))$, where $\buz\sim \Unif(\sol(\GG))$. Then, consider $\widetilde{\P}_{\bv,t}[\bGG]$, where $\bv$ is drawn uniformly at random from $V$ and $\bGG$ is a random regular \textsc{nae-sat} instance. Then, we conjecture that in the condensation regime $\alpha\in (\alpha_{\cond},\alpha_{\sat})$, the random element $\widetilde{\P}_{\bv,t}[\bGG]$ in $\PPP\big( \{0,1\}^{V(\TTT_{d,k,t})}\times \{0,1\}^{E_{\sf in }(\TTT_{d,k,t})}\big)$ equipped weak star topology converges weakly to 
\[
\widetilde{\P}_{\bv,t}[\bGG]\stackrel{w}{\longrightarrow}\frac{1}{2}\Big(\widetilde{\P}_{\star,+}^{t}+\widetilde{\P}_{\star,-}^{t}\Big),
\]
where the random elements $\widetilde{\P}_{\star,+}^{t}, \widetilde{\P}_{\star,-}^{t}$ are defined as follows. For $(\uz_t, \uL_t) \in  \{0,1\}^{V(\TTT_{d,k,t})}\times \{0,1\}^{E_{\sf in }(\TTT_{d,k,t})}$,
\[
\widetilde{\P}_{\star,+}^{t}(\uz_t,\uL_t)=
\sum_{i=1}^{\infty} \omega_i\cdot \frac{ I\big(\uz_t\sim_{\uL_t} \ux(\rev{\underline{\mm}_t^{(i)}})\big) \prod_{e\in \partial \TTT_{d,k,t}} \hat{\mm}_e^{(i)}(\uz_{v(e)}) }{\sum_{\uz'_t} I\big(\uz'_t\sim_{\uL_t} \ux(\rev{\underline{\mm}_t^{(i)}})\big) \prod_{e\in \partial \TTT_{d,k,t}} \hat{\mm}_e^{(i)}(\uz'_{v(e)}) }\,,\quad\widetilde{\P}_{\star,-}^{t}(\uz_t,\uL_t):=\widetilde{\P}_{\star,+}^{t}(\neg\,\uz_t,\uL_t)\,.
\]
Here, $(\omega_i)_{i\geq 1}$ follows Poisson-Dirichlet distribution with parameter $\la^\star\equiv \la^\star[\alpha,k]$ (see Chapter 2 of \cite{panchenko13a} for the definition of Poisson-Dirichlet process) and $\big(\rev{\underline{\mm}_t^{(i)}}\big)_{i\geq 1}\stackrel{i.i.d.}{\sim} \nu_{\la^\star}(\cdot\,;\,\uL_t)$. Also, $\neg\,\uz_t$ is obtained from $\uz_t$ by flipping $0$ and $1$. Establishing this conjecture is equivalent to showing that the cluster sizes follow Poisson-Dirichlet distribution, which is a major open problem for r\textsc{csp}'s in 1\textsc{rsb} class.

    \section{Proof overview}\label{sec:proof}
    In this section, we give an overview of the proof of Theorem \ref{thm:main} which is equivalent to 
    \begin{equation}\label{eq:thm:equiv}
        \limsup_{C\to\infty}\limsup_{n\to\infty} \P\left(\; \bigg|\frac{1}{n}\sum_{v\in V} X_v^t -\PP_\star^{t}(\uz_t,\uL_t)\bigg|\geq \frac{C}{\sqrt{n}}\;\right)=0,\quad\rev{\forall} (\uz_t,\uL_t) \in\rev{\{0,1\}^{V(\TTT_{d,k,t})}\times \{0,1\}^{E_{\sf in }(\TTT_{d,k,t})}}.
    \end{equation}
    Based on $\rsb$ heuristics, we analyze the law of $\ubz \sim \Unif(\sol(\bGG))$ by first conditioning on \rev{its} coarsening $\bux\in \{0,1,\ff\}^{V}$: we will see in Section \ref{subsec:proof:freecomp} that conditioned on $(\bGG, \bux)$, the law of $\ubz$ can be described in a relatively simple manner based on \textit{belief propagation}. We then divide the cases into whether the frozen configuration $(\bGG,\bux)$ is \textit{favorable} or not. By Chebyshev's inequality, it suffices to show that for any $\eps>0$, there exists a set $\fav\equiv \fav(\eps)$ of favorable frozen configurations such that $\P\left((\bGG,\bux)\rev{\notin} \fav \right) \leq \eps$ and 
    \begin{align}
        \sup_{(\GG,\ux)\in \fav}\Var\Big(&\sum_{v\in V} X_v^{t} \bgiven (\bGG,\bux)=(\GG, \ux) \Big)\leq C \sqrt{n}\;,
        \label{eq:goal:variance}\\
        \sup_{(\GG,\ux)\in \fav}\bigg|\frac{1}{n}\E\Big[&\sum_{v\in V} X_v^t\bgiven (\bGG,\bux)=(\GG, \ux) \Big]-\PP_{\star}^{t}(\uz_t,\uL_t)\bigg|\leq \frac{C}{\sqrt{n}}\;,
        \label{eq:goal:bias}
    \end{align}
    where $C\equiv C(\eps,k,t)>0$ is a constant that depends only on $\eps,k,t$. In the subsequent subsections, we explain the main ideas on how to establish \eqref{eq:goal:variance} and \eqref{eq:goal:bias} on some typical event $\fav$. In Section \ref{subsec:proof:freecomp}, we define the notion of \textit{free components}, which intuitively are the connected components of the subgraph formed by the free variables. Free components play a crucial role in understanding the law of $\ubz \sim \Unif(\sol(\bGG))$ conditioned on its coarsening $\bux$. Indeed, we show in Section \ref{subsec:proof:exp:decay} that the variance control \eqref{eq:goal:variance} follows from the exponential decay of the frequencies of the free components, which was established in \cite{NSS}. Obtaining the bias control \eqref{eq:goal:bias} is where most of the challenges lie in. In Sections \ref{subsec:proof:concentration} and \ref{subsec:proof:coupling}, we explain the main ideas to establish \eqref{eq:goal:bias}.

    \textbf{Notations:} Throughout, we let $\ux \in \{0,1,\ff\}^{V}$ be a valid frozen configuration on a $\naesat$ instance $\GG=(V,F,E,\uL)$. We often identify $V\equiv \{1,2,\ldots, n\}$ for convenience. \rev{
    For each $e\in E$, we denote the variable (resp. clause) adjacent to it by $v(e)$ (resp. $a(e)$).  Moreover, $\delta v$ (resp. $\delta a$) denotes the collection of adjacent edges to $v\in V$ (resp. $a \in F$). For a vector $\sig \equiv (\sigma_i)_{i\in X} \in A^X$ that takes values in the set $A$ and a subset of indices $S\subset X$, we use the subscript $\sig_{S}\equiv (\sigma_i)_{i\in S}$ to denote the restriction of $\sig$ to S. An exception is made for the subscript $t$, e.g. $\sig_t$, which is reserved to denote vectors that depend on the depth or distance $t\geq 1$.} For non-negative quantities $f=f_{d,k, n,t}$ and $g=g_{d,k,n,t}$, we use any of the equivalent notations $f=O_{k,t}(g), g= \Omega_{k,t}(f), f\lesssim_{k,t} g$ and $g \gtrsim_{k,t} f $ to indicate that there exists a constant $C_{k,t}$ which depends only on $k$ and $t$ \rev{such that $f\leq C_{k,t}g$ holds}. We drop the subscript $t$ \rev{if} the constant $C$ only depends on $k$.
    
    \subsection{Free components}
    \label{subsec:proof:freecomp}
     Given $(\GG,\ux)$, a variable $v\in V$ is called \textbf{frozen} in $\ux$ if $x_v\in \{0,1\}$. If $x_v=\ff$, it is \textbf{free}. A clause $a\in F$ is called \textbf{separating} if there are $2$ adjacent frozen variables that evaluate $0$ and $1$, i.e. there exist $e, e^\prime \in \delta a$ such that $\tL_{e}\oplus x_{v(e)} = 0, \quad \tL_{e^\prime} \oplus x_{v(e^\prime)}=1.$ A clause $a\in F$ is \textbf{non-separating} if it is not separating. Observe that a separating clause can never be violated no matter how the free variables in $\ux$ are filled with $0$ or $1$.
    
    An edge $e\in E$ is called \textbf{forcing} if flipping the value of $x_{v(e)}$ invalidates $a(e)$, i.e. $\tL_{e}\oplus x_{v(e)} \oplus 1 = \tL_{e'}\oplus x_{v(e')}\in \{0,1\}$ for all $e'\in \delta a \setminus e$. A clause $a\in F$ is \textbf{forcing}, if there exists $e\in \delta a$ which is a forcing edge. In particular, a forcing clause is also separating.
	\begin{defn}\label{def:free:comp}
	    Given $(\GG, \ux)$, a \textbf{free piece}, denoted by $\fff^{\textnormal{in}}$, is a connected component of the subgraph induced by the free variables and non-separating clauses in $\ux$. A \textbf{free component} \rev{$\fff$} is a union of $\fff^{\textnormal{in}}$ and the half-edges adjacent to $\fff^{\textnormal{in}}$. Thus, $\fff$ is composed of the free piece $\fff^{\textnormal{in}}$ and the \textit{boundary half-edges} hanging from $\fff^{\textnormal{in}}$. Moreover, the (half-)edges of $\fff$ are labeled as follows.
 \begin{itemize}
     \item Denote by $V(\fff), F(\fff), E(\fff)$ the set of variables, clauses, and full-edges of $\fff$ (i.e. edges of $\fff^{\textnormal{in}}$), respectively. Then, each $e\in E(\fff)$ is labeled by its literal $\tL_e$.
     \item Let $\dot{\partial} \fff$ (resp. $\hat{\partial}\fff$) be the set of boundary half-edges adjacent to $F(\fff)$ (resp. $V(\fff)$), and write $\partial{\fff}:=\dot{\partial}\fff \sqcup \hat{\partial}\fff$. Then, $e\in \dot{\partial} \fff$ is labeled with the information $(x_{v(e)},\tL_e)\in \{0,1\}^2$. Here $x_{v(e)}\in \{0,1\}$ is guaranteed since it must be frozen. The label $x_{v(e)}$ is called \textbf{spin-label} whereas $\tL_e$ is called \textbf{literal-label}. The other boundary half-edges $e\in \hat{\partial}\fff$, which must be adjacent to separating clauses in $\GG$, are unlabeled.
 \end{itemize}
A \textbf{free tree} is a free component $\fff$ which does not contain a cycle. We often use the notation $\ttt$ to denote a free tree and the notation $\fff$ for a generic free component, which may contain a cycle. We denote the collection of free components inside $(\ux, \GG)$ by $\mathscr{F}(\ux,\GG)$ and the collection of free trees by $\mathscr{F}_{\tr}(\ux,\GG)\subseteq \mathscr{F}(\ux,\GG)$. We use the notation 
\rev{
\begin{equation*}
v_{\fff}=|V(\fff)|\,,~~~f_{\fff}=|F(\fff)|\,,~~~e_{\fff}=|E(\fff)|
\end{equation*}
}
for the number of variables, clauses, and edges of a free component $\fff$.
	\end{defn}
We remark that an equivalent labeling scheme was also used in \cite[Definition 2.18]{NSS}. However, the notion of `free tree' in \cite{NSS} (see Definition 2.16) is slightly different than the one in Definition \ref{def:free:comp}. Namely, \cite{NSS} further introduced an equivalence relation and defined a `free tree' as an equivalence class. It is crucial that we do not make this reduction for the purpose of the coupling described in Section \ref{subsec:proof:coupling}.

Note that a free component $\fff\in \FFF(\ux, \GG)$ is embedded in $\GG$ by definition. However, it can also be treated as a separate labeled graph. To this end, we denote the set of possible free components (up to graph and label isomorphisms) by $\FFF$ and the set of possible free trees by $\FFF_{\tr}\subsetneq \FFF$. For $\fff\in \FFF$, define the \textit{weight} of $\fff$ as
\begin{equation}\label{eq:weight:fff}
    w_{\fff}:=\left|\left\{\uz\in \{0,1\}^{V(\fff)}:\textnormal{ the coarsening of the $\naesat$ solution $\uz$ is $\fff$}\right\}\right|,
\end{equation}
where \rev{$\uz\in \{0,1\}^{V(\fff)}$} is a $\naesat$ solution if it satisfies every clause in $\fff$ (recall that the every (half-)edges adjacent to clauses store literal information) and the coarsening is taken with respect to the clauses of $\fff$ as the same manner as described in Definition \ref{def:frozenconfig}. \rev{That is, $w_{\fff}$ counts the number of $\naesat$ solutions $\uz\in \{0,1\}^{V(\fff)}$ such that applying the coarsening algorithm to $\uz$ sets every variable in $V(\fff)$ to $\ff$.} A crucial observation then follows.
\begin{obs}\label{obs:sampling}
Since a separating clause can never be violated, we have $\size(\ux,\GG)=\prod_{\fff\in \FFF(\ux,\GG)}w_{\fff}$. Thus, sampling a random regular $\naesat$ instance $\bGG$ and a uniformly random $\naesat$ solution $\ubz\in \Unif(\sol(\bGG))$ is \textit{equivalent} to the following sampling procedure. 
    \begin{enumerate}[label=(\alph*)]
        \item Sample a random regular $\naesat$ instance $\bGG$. Then, sample a frozen configuration, or equivalently a cluster, $\bux \in \{0,1,\ff\}^{V}$ with probability proportional to its weight $\size(\bux, \bGG)$, namely $\P(\bux= \ux \given \bGG)=\size(\ux, \bGG)/|\sol(\bGG)|$. 
        \item Given $(\bGG,\bux)$, sample a $\naesat$ solution $\ubz$ uniformly at random among those which are coarsened to $\bux$. Equivalently, for each free component $\fff \in \FFF(\bux,\bGG),$ independently sample $\ubz_{V(\fff)}\in \{0,1\}^{V(\fff)}$ uniformly at random among those which are coarsened to $\fff$.
    \end{enumerate}
\end{obs}
We remark that for a free tree $\ttt\in \FFF_{\tr}$, every $\naesat$ solution $\uz_{V(\ttt)}$ of $\ttt$ is coarsened to $\ttt$, so $\size(\ttt)$ is the number of $\naesat$ solutions of $\ttt$. Moreover, sampling a $\naesat$ solution $\ubz_{V(\ttt)}\in \{0,1\}^{V(\ttt)}$ of $\ttt$ uniformly at random can be analyzed in a simple manner. That is, any marginals of the law of $\ubz_{V(\ttt)}$ can be described by \textit{belief propagation} (see Section \ref{sec:BP}).
 
For the rest of this subsection, we review the notion of \textit{component coloring} in \cite{NSS}. Define the set $\CCC$ as
	\begin{equation}\label{eq:def:Omegcom1}
	\CCC:= \{\rr_0, \rr_1, \bb_0, \bb_1, \fs \} \cup \{ (\fff, e): \fff \in \mathscr{F}, \, e\in E(\fff)  \}\,.
	\end{equation}
	Here, we take the convention that $(\fff,e)=(\fff^\prime,e^\prime)$ if there exists a graph isomorphism from $\fff$ to $\fff^\prime$ that keeps the labels unchanged and maps $e$ to $e^\prime$. The symbols $\rr_0, \rr_1$ (resp. $\bb_0,\bb_1$) represent `red' (resp. `blue') spins, and $\fs$ represent `separating' spin\rev{s}. \rev{Roughly speaking, red spins are assigned to edges that are forcing while blue spins are assigned to non-forcing edges adjacent to frozen variables. Separating spins are assigned to edges between a separating clause and a free variable.} The terminologies were introduced in \cite{ssz22}, building on the works \cite{cp13stoc, coja14stoc}.
	\begin{defn}\label{def:compcol}
		Let $\ux\in \{0,1,\ff\}^{V}$ be a (valid) frozen configuration on $\GG$. The \textbf{component coloring} $\sig\equiv (\sigma_e){} \in \CCC^E$ corresponding to $\ux$ is defined by the following procedure.
		\begin{enumerate}
			\item For each frozen variable $v \in V$, i.e. $x_v\in\{0,1\}$, and an adjacent edge $e\in \delta v$, assign $\sigma_e = \rr_{x_v}$ if $e$ is forcing, $\sigma_e=\bb_{x_v}$ otherwise.
			\item For each separating clause $a$ and an adjacent edge $e\in \delta a$, assign $\sigma_e = \fs$ if $x_{v(e)}= \ff$.
			\item For the other edges $e\in E$, which must be contained in a free component of $\ux$, let $\fff_e\in \mathscr{F}$ be the free component that contains $e$. We then set $\sigma_e = (\fff_e,e)\in \CCC$.
		\end{enumerate}
     \rev{With a slight abuse of notation, we use the symbol $\rr$ (resp. $\bb$) to denote a spin that can take a value either $\rr_0$ or $\rr_1$ (resp. $\bb_0$ or $\bb_1$).}
	\end{defn}
	Given $\GG$, we call a component coloring $\sig$ valid on $\GG$ if there exists a valid frozen configuration $\ux$ on $\GG$ such that it maps to $\sig$ with the above procedure. Then, it is straightforward to see that the procedure in Definition \ref{def:compcol} gives a one-to-one correspondence between the valid frozen configurations and the valid component colorings. The following remark plays an important role later in Section \ref{subsec:proof:coupling}.
    \begin{remark}\label{rmk:comp:col}
    Given $(\GG,\ux)$, suppose $e\in E$ is contained in a free component. Note that by Definition \ref{def:free:comp}, $(\fff_e,e)$ contains the information on literals of $\fff_e$ and spin labels on boundary half-edges of $\fff$. Therefore, $\sigma_e=(\fff_e,e)$ completely determines the colors of adjacent edges, $(\sigma_{e^\prime})_{e^\prime \in \delta v(e)\setminus e\,\sqcup\, \delta a(e)\setminus e}$.
    \end{remark}
    \subsection{Exponential decay of the frequencies of free components}
    \label{subsec:proof:exp:decay}
    By Observation \ref{obs:sampling}, conditional on $(\bGG, \bux)=(\GG, \ux)$, we have that $X_i^{t}$ and $X_j^t$ are independent if $N_t(i,\GG)\cap N_t(j,\GG)=\emptyset$ and there is no free component intersecting both $N_t(i,\GG)$ and $N_t(j,\GG)$. Thus, the critical component in establishing the variance control \eqref{eq:goal:variance} is to show that the large free components are rare in a typical frozen configuration, which enables us to argue that for most $i,j\in V$, $X_i^t$ and $X_j^t$ are (conditionally) independent. We formalize this idea in this subsection. We start with the definition of \textit{boundary profile} and \textit{free component profile} introduced in \cite[Definition 3.2]{NSS}.

   \begin{defn}\label{def:empirical:boundary}
	Given $(\GG, \ux)$, the unnormalized free component profile of $\ux$ is the sequence $(n_\fff[\GG,\ux])_{\fff\in \mathscr{F}}$, where $n_\fff[\GG,\ux]$ is the number of free component\rev{s} $\fff$ inside $(\GG,\ux)$. The \textbf{free component profile} is then $\{p_{\fff}[\GG, \ux]\}_{\fff\in \FFF}:=\Big\{\frac{n_{\fff}[\GG, \ux]}{n}\Big\}_{\fff\in \FFF}$. The \textbf{boundary profile} of $(\GG, \ux)$ is the tuple $B[\GG,\ux] \equiv (\dot{B},\hat{B},\bar{B})$ defined as follows. Let $\sig \in \CCC^{E}$ be the component coloring corresponding to $\ux$ in Definition \ref{def:compcol}. Then, $\dot{B}$, $\hat{B},$ and $\bar{B}$ are respectively measures on $\{\rr,\bb\}^d$, $\{\rr,\bb, \fs\}^k$ and $\{\rr,\bb, \fs\}$ defined as
		\begin{equation*}
		\begin{split}
		&\dot{B}(\underline{\tau})
		:=
		|\{v\in V: \sig_{\delta v}=\underline{\tau} \} | / |V| \quad
		\textnormal{for all } \underline{\tau}\in \{\rr,\bb\}^d\;,\\
		 &\hat{B}(\underline{\tau})
		:=
		|\{a\in F: \sig_{\delta a}=\underline{\tau} \} | / |F| \quad
		\textnormal{for all } \underline{\tau}\in \{\rr,\bb, \fs\}^k\;,\\
		 &\bar{B}(\tau)
		:=
		|\{e\in E: \sigma_e=\tau \} | / |E| \quad
		\textnormal{for all } {\tau}\in \{\rr,\bb, \fs\}\;,
		\end{split}
		\end{equation*}
       \rev{where we recall that $\sig_{\delta a}$ (resp. $\sig_{\delta v})$ denotes $\sig$ restricted to $\delta a$ (resp. $\delta v$).} Finally, the \textbf{($1$-neighborhood) coloring profile} of $(\GG,\ux)$ is the collection of the boundary profile and the free component profile, which we denote by $\xi[\GG,\ux]:=(B[\GG, \ux], \{p_{\fff}[\rev{\GG}, \ux]\}_{\fff\in \FFF})$.
	\end{defn}
    For $r>0$, let $\ee_{r}$ be the collection of free component profiles satisfying the exponential decay of frequencies in its number of variables with rate $2^{-rk}$. That is,
	\begin{equation}\label{eq:def:exp:decay:profile}
	    \ee_{r}:= \Big\{(p_\fff)_{\fff \in \FFF}: \sum_{\fff \in \FFF,\rev{v_{\fff}}=v} p_\fff \leq 2^{-rkv}, ~~\forall v\geq 1\;\Big\}\,.
	\end{equation}
    With slight abuse of notation, we also denote $\xi=(B,\{p_{\fff}\}_{\fff\in \FFF})\in \ee_r$ if $\{p_{\fff}\}_{\fff\in \FFF}\in \ee_r$. In Section \ref{sec:concentration:free:tree}, we show that $(\bGG,\bux)$ is contained in the set $\ee_{\frac{1}{4}}$ and there are no multi-cyclic free components, i.e. a free component with more than one cycle, with high probability.
    \begin{lemma}\label{lem:exp:decay}
       For $k\geq k_0$ and $\alpha \in (\alpha_{\cond}(k),\alpha_{\sat}(k))$, we have that
       \begin{equation*}
       \P\left(\{p_{\fff}[\bGG,\bux\,]\}_{\fff\in \FFF}\notin \ee_{\frac{1}{4}}~~\textnormal{or}~~p_{\fff}[\bGG,\bux]\neq 0~~\textnormal{for some multi-cyclic $\fff\in \FFF$}\right)=o_n(1)\,.
       \end{equation*}
    \end{lemma}
    Moreover, we show that if $\big\{p_{\fff}[\bGG,\bux\,]\big\}_{\fff\in \FFF}\in \ee_{\frac{1}{4}}$ holds, then our desired variance control \eqref{eq:goal:variance} holds.
    \begin{lemma}\label{lem:var:control}
    Consider a frozen configuration $(\GG, \ux)$ which satisf\rev{ies} $\big\{p_{\fff}[\GG,\ux\,]\big\}_{\fff\in \FFF}\in \ee_{\frac{1}{4}}$. Then, we have
    \begin{equation}\label{eq:lem:var:control}
        \Var\Big(\sum_{i=1}^{n}X_i^{t} \bgiven (\bGG,\bux)=(\GG, \ux) \Big)\lesssim_{k,t} \sqrt{n}\,.
    \end{equation}
    \end{lemma}
    \begin{proof}
    Given $(\GG, \ux)$ with $\big\{p_{\fff}[\GG,\ux\,]\big\}_{\fff\in \FFF}\in \ee_{\frac{1}{4}}$, denote by $I_{\g}[\GG, \ux]$ the collection of $(i,j)\in V^2$ such that $N_t(i,\GG)\cap N_t(j,\GG)=\emptyset$ and there exists no free component $\fff\in \FFF(\ux,\GG)$ which intersect\rev{s} both $N_t(i,\GG)$ and $N_t(j,\GG)$. Then, by Observation \ref{obs:sampling}, we have that 
    \begin{equation*}
        (i,j)\in I_{\g}[\GG,\ux] \implies \Cov\Big(X_i^{t}, X_j^{t} \bgiven (\bGG,\bux)=(\GG,\ux)\Big)=0\,.
    \end{equation*}
    Note that for a fixed $i\in V$, there are at most $(kd)^{2t-2}$ $j\in V$ with $d(i,j)\leq 4t-4$. Moreover, for a fixed free component $\fff\in \FFF(\ux,\GG)$, if $N_t(i,\GG)$ intersects with $\fff$, then there must be a variable $v\in V(\fff)$ with $d(i,v)\leq 2\rev{t}-2$ (the boundary of $N_t(i,\GG)$ are formed by variables, not clauses), so there are at most $v_{\fff}(kd)^{t-1}$ such variables $i\in V$. Hence, it follows that
    \begin{equation*}
        \left|V^2 \setminus I_{\g}[\GG, \ux]\right|\leq n\bigg((kd)^{2t-2}+\sum_{\fff\in \FFF} v_{\fff}^2(kd)^{2t-2}\cdot p_{\fff}[\GG,\ux]\bigg)\lekt n\,,
    \end{equation*}
    where the last inequality holds since $\sum_{\fff\in \FFF}v_{\fff}^2 \cdot p_{\fff}[\GG,\ux]\leq \sum_{v}v^2 2^{-\frac{kv}{4}}$ holds due to $\big\{p_{\fff}[\GG,\ux\,]\big\}_{\fff\in \FFF}\in \ee_{\frac{1}{4}}$. Therefore, our claim \eqref{eq:lem:var:control} follows.
    \end{proof}

   \subsection{Tight concentration of free component profile}\label{subsec:proof:concentration}
   We now consider the bias control in \eqref{eq:goal:bias}. First, let us examine the case where $t=1$. Then, $N_1(i,\bGG)$ is always isomorphic to $\TT_{d,k,1}$, which consists of a single variable $\rho$ and boundary half-edges adjacent to it (with no literals), thus $z_1 \in \{0,1\}^{\TT_{d,k,1}}$ is given by $z_1=z_{\rho}\in \{0,1\}$. Thus, if a variable $i\in V$ is frozen with respect to $(\GG,\ux)$, then $X_i^{1}\equiv X_i^{1}[\GG,\ux,\uz_1,\uL_1]$ is deterministic, which equals the indicator that $x_i=z_{\rho}$.
   
   On the other hand, if a variable $i\in V$ is \rev{free} with respect to $(\GG,\ux)$, then recalling Observation \ref{obs:sampling}, $\E[X_i^{1}\given (\bGG,\bux)=(\GG,\ux)]$ is a function which only depends on $\fff$, the free component which contains $i\in V$. More precisely, $\E[X_i^{1}\given (\bGG,\bux)=(\GG,\ux)]$ equals the probability that $\boldsymbol{z}_i=z_{\rho}$, where $\ubz_{V(\fff)}\in \{0,1\}^{V(\fff)}$ is sampled u.a.r. among those which are coarsened to $\fff$. When $\fff$ is a free tree, this probability can be evaluated using \textit{belief propagation}.

   Therefore, if we establish the tight $O(n^{-1/2})$ concentration of the number of frozen variables of $(\bGG, \bux)$ and the analogue $\ell^1$- type concentration on the free component profile $\{p_{\fff}[\bGG,\bux]\}_{\fff\in \FFF}$, then we can establish tight bias control \eqref{eq:goal:bias} for the simplest case $t=1$. In fact, the former concentration of the number of frozen variables can be established using the concentration of the boundary profile $B[\bGG,\bux]$ around the \textit{optimal} boundary profile $B^\star$ using the results of \cite{NSS}. We review the definition of $B^\star$ \rev{from} \cite{NSS} in Section \ref{sec:BP}.

   However, establishing $O(n^{-1/2})$ concentration of the free component profile $\{p_{\fff}[\bGG,\bux]\}_{\fff\in \FFF}$ in $\ell^1$-type distance poses significant challenges. Indeed, the results of \cite{NSS} only imply a much weaker concentration in $\ell^{\infty}$-distance of the free component profile\rev{, i.e. $\sup_{\fff\in \FFF}|p_{\fff}-p_{\fff}^\star|$ for some $(p_{\fff}^\star)_{\fff\in \FFF}$ (see Definition~\ref{def:optimal:coloring}), }with larger distance $O(\frac{\log n}{\sqrt{n}})$, and it is a priori not clear if the stronger $\ell^1$-type concentration can be established, let alone removing the $\log n$ factor to obtain the optimal fluctuation as in Theorem \ref{thm:main}. Note that there are \textit{unbounded} number of types of free components in a typical frozen configuration $(\bGG,\bux)$\footnote{From the branching process heuristics described in \cite{dss16}, a typical frozen configuration has the largest free tree with $\Theta_k(\log n)$ variables as one can infer from Lemma \ref{lem:exp:decay}, and there are exponentially many types of free trees with a given number of variables $v$, so there are typically $n^{\Omega_k(1)}$ types of free trees.}\rev{.}

   Nevertheless, we show in Section \ref{sec:concentration:free:tree} that by a delicate use of a local central limit theorem for triangular arrays \cite{Borokov17} and the exponential decay of the free component profile (cf. Lemma \ref{lem:exp:decay}), the free component profile concentrates in $\ell^1$-type distance on the optimal scale $O(n^{-1/2})$. To be precise, consider the following distance of $2$ coloring profiles $\xi^1=(B^1, \{p_{\fff}^1\}_{\fff\in \FFF})$ and $\xi^2=(B^2, \{p_{\fff}^2\}_{\fff\in \FFF})$:
   \begin{equation}\label{eq:def:tsq:norm}
       \norm{\xi^1-\xi^2}_{\tsq}:=\norm{B^1-B^2}_1+\sum_{\fff\in \FFF}\left|p_{\fff}^1-p_{\fff}^2\right|(v_{\fff}+f_{\fff})\,.
   \end{equation}
   We show in Section \ref{sec:concentration:free:tree} that the following \rev{t}heorem holds.
   \begin{thm}\label{thm:concentration}
    For $k\geq k_0$ and $\alpha \in (\alpha_{\cond}(k),\alpha_{\sat}(k))$ there exists an explicit $\xi^{\star}\equiv \xi^\star[\alpha,k] \equiv (B^\star,\{p_{\fff}^\star \}_{\fff\in \FFF})$ in Definition \ref{def:optimal:coloring} such that the following holds. For any $\eps>0$, there exists a constant $C(\eps,k)>0$ such that  
    \begin{equation*}
        \P\left(\big\|\xi[\bGG,\bux\,]-\xi^\star\big\|_{\tsq}\geq \frac{C}{\sqrt{n}}\right)\leq \eps\,.
    \end{equation*}
   \end{thm}

   \subsection{From $1$-neighborhoods to $t$-neighborhoods}
   \label{subsec:proof:coupling}
   In this subsection, we discuss the ideas for establishing the bias control \eqref{eq:goal:bias} for general $t\geq 2$.

   Given $(\GG, \ux)$, let $\sig\in \CCC^{E}$ be the corresponding component coloring (cf. Definition \ref{def:compcol}). Note that by Observation \ref{obs:sampling}, the quantity $\E[X_i^{t}\given (\bGG,\bux)=(\GG,\ux)]$ is completely determined by the configuration of free components intersecting with $N_t(i,\GG)$. Further, note that \rev{for $e\in E(N_t(i,\GG))$}, $\sigma_e$ encodes the free component that $e$ is contained in. Here, we emphasize that $E(N_t(i,\GG))$ contains the boundary half-edges $\partial N_t(i,\GG)$. Therefore, $N_t(i, \GG)$, $\sig_{t}(i, \GG)\equiv (\sigma_e)_{e\in E(N_t(i,\GG))}$, and $\uL_{t}(i,\GG)$ determine $\E[X_i^{t}\given (\bGG,\bux)=(\GG,\ux)]$ (see Remark \ref{rmk:p} below) with one exception: namely, \rev{when two distinct leaves of $N_t(i,\GG)$ are connected by a path inside a free component $\fff\in \FFF(\ux,\GG)$ and the path is contained in the complement of $N_t(i,\GG)$.}
   Fortunately, as shown below, this exceptional case is very rare and, thus, can be neglected.  
   \begin{figure}[t]
   \begin{center}
\includegraphics[width=0.4\textwidth]{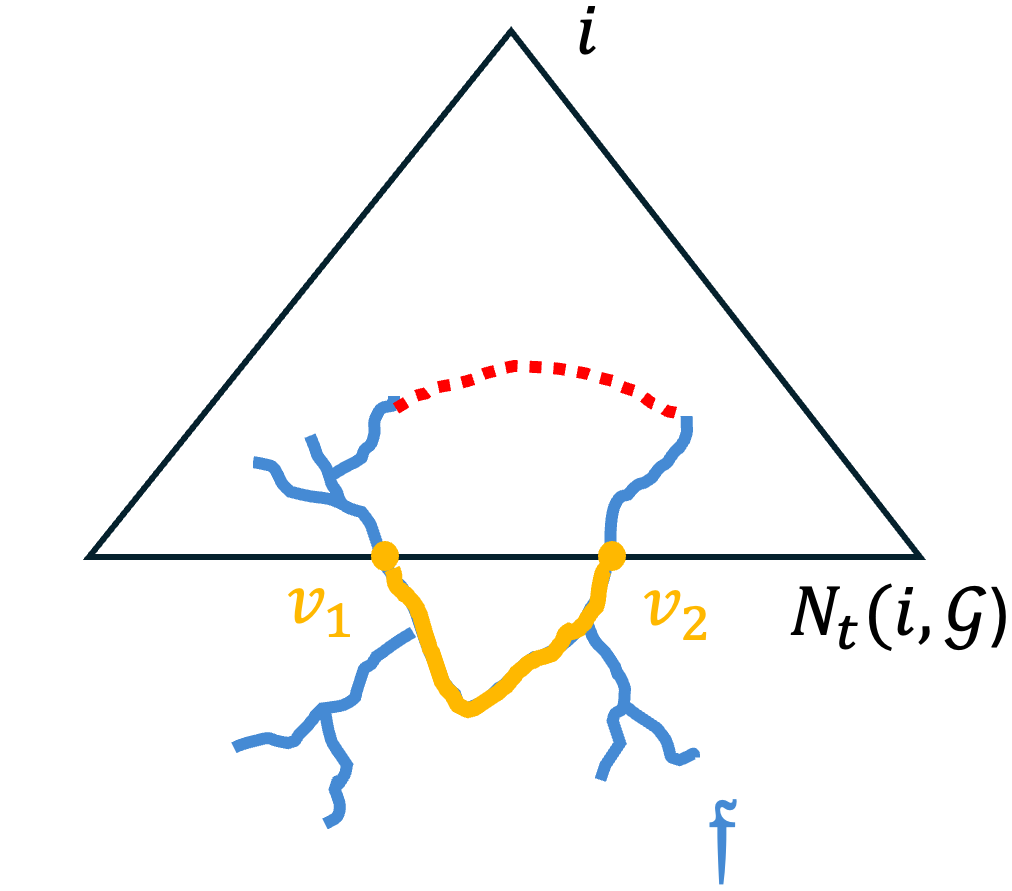}
\end{center}
\hspace{0.5cm}
\caption{A pictorial description of a scenario where two distinct leaves $v_1,v_2$ of $N_t(i,\GG)$ are connected by a path (colored yellow), which lies inside $\fff$ and outside $N_t(i,\GG)$. The blue curves represent the free component $\fff$ intersecting with $N_t(i,\GG)$. The red dashed line represents the edges colored $\{\rr_0,\rr_1,\bb_0,\bb_1,\fs\}$ under the component coloring $\sig$. Lemma~\ref{lem:bdry:cycle} shows that most of the cycles contain more than $2t$ edges colored $\{\rr_0,\rr_1,\bb_0,,\bb_1,\fs\}$.
}

\label{fig:cycle}
\end{figure}

   Note that in the exceptional case described above, there must be a self-avoiding cycle, i.e. a cycle that does not self-intersect, within $\fff\cup N_t(i,\GG)$. That is, consider $2$ variables $v_1,v_2$ that are \rev{leaves of $N_t(i,\GG)$ and are contained in $\fff$.} Then, there must exist $2$ distinct self-avoiding paths that connect $v_1,v_2$, where the first one \rev{is contained in $\fff\cap N_t(i,\GG)^{\textsf{c}}$} and the second one lies within $N_t(i,\GG)$. Concatenating \rev{these two paths} produces a self-avoiding cycle. Observe that the length of this cycle might be as large as $\Theta_k(\log n)$ given the exponential decay in Lemma \ref{lem:exp:decay} due to the case where $\fff$ is large. However, if we only count the edges \rev{in this cycle} that are not contained in free components, \rev{i.e. the edges that have colors $\{\rr_0, \rr_1,\bb_0,\bb_1, \fs\}$ under the component coloring $\sig$, }
  the number of such edges is at most $2t$. \rev{See Figure~\ref{fig:cycle} for a pictorial description.} The lemma below shows that \rev{such self-avoiding cycles that have at most $2t$ edges outside of free components are rare. In other words, most of the cycles contain more than $2t$ edges that are colored $\{\rr_0, \rr_1,\bb_0,\bb_1, \fs\}$.}
   \begin{lemma}\label{lem:bdry:cycle}
       Given $(\GG,\ux)=(V,F,E,\uL,\ux)$ and a cycle $\CC$ inside the bipartite graph $(V,F,E)$, \rev{let} $N_{\cyc}^{\sf b}(2t;\GG,\ux)$ \rev{denote} the number of self-avoiding cycles \rev{which have at most $2t$ edges that are not contained in free components of $\ux$, i.e. the edges that are colored $\{\rr_0,\rr_1,\bb_0,\bb_1,\fs\}$ under the component coloring $\sig$ corresponding to $\ux$.} Uniformly over $\xi=(B,\{p_{\fff}\}_{\fff\in \FFF})$ such that $\{p_{\fff}\}_{\fff\in \FFF}\in \ee_{\frac{1}{4}}$ and $p_{\fff}=0$ if $\fff$ is multi-\rev{cyclic}, and $\norm{B-B^\star}_{1}\leq n^{-1/3}$, we have for each fixed $t\geq 1$ that
       \begin{equation*}
           \P_{\xi}\Big(N_{\cyc}^{\sfb}(2t;\bGG,\bux)\geq n^{1/3}\Big)\equiv \P\Big(N_{\cyc}^{\sfb}(2t;\bGG,\bux)\geq n^{1/3}\bgiven \xi[\bGG,\bux]=\xi\Big)=o_n(1)\,.
       \end{equation*}
   \end{lemma}
   \begin{remark}\label{rmk:P:xi}
       Hereafter, we denote $\P_{\xi}$ (resp. $\E_{\xi}$) \rev{as} the probability (resp. expectation) with respect to the conditional measure $\P(\cdot\given \xi[\bGG,\bux]=\xi)$. Since $\xi=\xi[\GG,\ux]$ determines the size of $\ux$, \rev{i.e. }$\size(\ux,\GG)$ (cf. Observation \ref{obs:sampling}), $\P_{\xi}$ is a uniform measure over $(\bGG,\bux)$ \rev{that satisfies} $\xi[\bGG,\bux]=\xi$. Thus, $\P_{\xi}$ can be described in terms of a configuration model with colored edges. This fact is further explained in Observation \ref{obs:config} below. We remark that a similar idea was also used in \cite{Bordenave15}.
   \end{remark}

   Having Lemma \ref{lem:bdry:cycle} in hand, we focus our attention on the empirical measure of $\big(\sig_{t}(i, \GG), \uL_t(i,\GG)\big)_{i\in V}$.
   \begin{defn}\label{def:t:nbd:empirical}
   We say $(\sig_t,\uL_t)\in \CCC^{E(\TTT_{d,k,t})}\times\{0,1\}^{E_{\sf in}(\TTT_{d,k,t})}$ is a valid (literal-reinforced-) \textbf{$t$-coloring} if it can be realized as $\big(\sig_t(i,\GG), \uL_t(i,\GG)\big)=(\sig_t,\uL_t)$ for some valid component coloring $(\GG,\sig)$ and $i\in V$.
   
   We take the convention that the $2$ valid $t$-colorings $(\sig^{(\ell)}_t, \uL^{(\ell)}_t)\equiv \big((\sigma^{(\ell)}_e)_{e\in E(\TTT_{d,k,t})},(\tL^{(\ell)}_e)_{e\in E_{\sf in}(\TTT_{d,k,t})}\big)$, $\ell =1,2,$ are the same if there is an automorphism $\phi$ of $\TTT_{d,k,t}$ such that $\sigma^{(1)}_e=\sigma^{(2)}_{\phi(e)}, e\in E(\TTT_{d,k,t})$, and $\tL^{(1)}_e=\tL^{(2)}_{\phi(e)}, e\in E_{\sf in}(\TTT_{d,k,t})$. Denote \rev{by} $\Omega_t$ the set of all valid $t$-colorings up to automorphisms. Then, the \textbf{t-coloring profile} of a valid frozen configuration $(\GG,\ux)$ is the probability measure $\nu_t\equiv \nu_t[\GG,\ux]$ on the set $\Omega_{t}\sqcup \{\cyc\}$ defined by
       \begin{equation*}
       \begin{split}
           &\nu_t(\utau_t,\uL_t):=\big|\{i\in V:\;\sig_{t}(i,\GG)=\utau_t~~\textnormal{and}~~\uL_{t}(i,\GG)=\uL_t\}\big|\,/\,|V|\quad\textnormal{for all}\quad (\tau_t,\uL_t)\in \Omega_t\;,\\
           &\nu_t(\cyc):=\big|\{i\in V:\;N_t(i,\GG)\textnormal{  contains a cycle}\}\big|\,/\,|V|\;.
       \end{split}
       \end{equation*}
   \end{defn}
   \begin{remark}\label{rmk:p}
    Given $(\GG,\ux)$ and $i\in V$, suppose that $\fff\cap N_t(i,\GG)$ is connected for every $\fff\in \FFF(\ux,\GG)$. Then,
    \begin{equation*}
        \E\big[X_i^t[\bGG,\ubz, \uz_t,\uL_t] \;\big|\; (\bGG,\bux)=(\GG,\ux)\big]=p_{\uz_t}\big( \sig_t(i,\GG),\uL_t\big)\one\left(\textnormal{$N_t(i,\GG)$ is a tree and $\uL_{t}(i,\GG)=\uL_t$}\right).
    \end{equation*}
    Here, $p_{\uz_t}(\sig_t,\uL_t)$ is defined as follows for $\uz_t \in \{0,1\}^{V(\TTT_{d,k,t})}, \uL_t\in \{0,1\}^{E_{\sf in}(\TTT_{d,k,t})}$, $\sig_t\equiv (\sigma_e)_{e\in E(\TTT_{d,k,t})}\in \CCC_t$. Note that $(\sig_t,\uL_t)$ not only determines the $\{0,1,\ff\}$ configuration on $V(\TTT_{d,k,t})$, but also the \textit{hanging} free trees on $\partial \TTT_{d,k,t}$. That is, for $e\in \partial \TTT_{d,k,t}$, $\sigma_e$ encodes the free component $\fff_e$, if any, that $e$ resides in. Assuming that all the free components $\fff_e$ are disjoint for $e\in \partial \TTT_{d,k,t}$, $p_{\uz_t}(\sig_t,\uL_t)$ is the probability of obtaining $\uz_t$ when we independently sample for each $\fff\in \{\fff_e\}_{e\in E(\TTT_{d,k,t})}$, a $\naesat$ solution $\ubz_{V(\fff)}\in \{0,1\}^{V(\fff)}$ uniformly at random among those which are coarsened to $\fff$. When $\{\fff_e\}_{e\in E(\TTT_{d,k,t})}$ are all free trees, the value $p_{\uz_t}(\sig_t, \rev{\uL_t})$ can be expressed in terms of \textit{belief-propagation}.
   \end{remark}
   The main component in obtaining the bias control \eqref{eq:goal:bias} is to enhance the concentration of \rev{the} coloring profile as stated in Theorem \ref{thm:concentration} to that of \rev{the} $t$-coloring profile. \begin{defn}\label{def:Xi}
       For $C>0$, let $\Xi_C\equiv \Xi_{C,n}$ be the set of coloring profile\rev{s} $\xi$ which satisfy both $\norm{\xi-\xi^\star}_{\tsq}\leq \frac{C}{\sqrt{n}}$ and $\E_{\xi}\big[N_{\cyc}(2t;\GG)\big]\leq C$, where $N_{\cyc}(2t;\GG)$ denotes the number of cycles in $\GG$ of length at most $2t$.
   \end{defn}
   \begin{thm}\label{thm:coupling}
    For $k\geq k_0$ and $\alpha \in (\alpha_{\cond}(k),\alpha_{\sat}(k))$ there exists an explicit $\nu^{\star}_t\equiv \nu^\star_t[\alpha,k]\in \PPP(\Omega_t)$ in Definition \ref{def:optimal:t} below, such that the following holds. For any $\eps>0, C>0$, and $t\geq 1$, there exists a constant $K\equiv K(C,\eps,k,t)>0$ such that uniformly over any vector $w \in [-1,1]^{\Omega_t\cup\{\cyc\} }$ and $\xi\in \Xi_C$, we have
   \begin{equation*}
       \P_{\xi}\bigg(\Big|\big\langle \nu_t[\bGG, \bux] -\nu_t^\star \,,\, w\big\rangle\Big|\geq \frac{K}{\sqrt{n}}\bigg)\leq \eps\,.
   \end{equation*}
   \end{thm}
 
   The proof of Theorem \ref{thm:coupling} is provided in Section \ref{sec:coupling}. For the remainder of this subsection, we discuss the high-level ideas in proving Theorem \ref{thm:coupling}.
   
   The measure $\nu_t^\star$ in Theorem \ref{thm:coupling} can be described by a so-called \textit{broadcast} model based on $\xi^\star$ (see Section \ref{subsec:broadcast}). Indeed, this description is what we will use to prove Theorem \ref{thm:coupling}. Let $\E_{\xi}\bnu_t\in \PPP(\Omega_t\cup\{\cyc\})$ be the conditional expectation given $\xi$ of the random element $\bnu_t\equiv \nu_t[\bGG,\bux\,]$ in $\PPP(\Omega_t\cup\{\cyc\})$. Then, by Chebyshev's inequality, Theorem \ref{thm:coupling} is implied by
   \begin{equation}\label{eq:coupling}
       d_{\tv}\left(\,\E_{\xi} \bnu_t \,,\,\nu_t^\star \,\right)\lekt n^{-1/2}\,,\quad\quad d_{\tv}\left(\,\E_{\xi}[\bnu_t\otimes \bnu_t]\,,\, \E_{\xi}\bnu_t\otimes \E_{\xi}\bnu_t\,\right)\lekt n^{-1} \,,
   \end{equation}
   uniformly over $\xi\in \Xi_C$, where $\bnu_t\otimes \bnu_t\in \PPP\left((\Omega_t\cup \{\cyc\})^2\right)$ is the product measure of $\bnu_t$'s.
   
   We establish the first claim in \eqref{eq:coupling} by coupling the measures $\E_{\xi}[\bnu_t]$ and $\nu_t^\star$. Since $\E_{\xi}[\bnu_t]$ can be described by a certain \textit{sampling with replacement} procedure based on $\xi$ and $\nu_\star^t$ is the analog \textit{sampling without replacement} procedure based on $\xi^\star$, we construct a coupling of the two procedures based on a coupling of $\xi$ and $\xi^\star$. The $O_{k,t}(n^{-1/2})$ probability of error comes from the coupling of $\xi$ and $\xi^\star$ since the difference between sampling with or without \rev{replacement} only induces error of probability $O_{k,t}(n^{-1})$. Similarly, we establish the second claim in \eqref{eq:coupling}, where the dominant probability of error now comes from the latter difference.

   In the coupling argument above, rare spins might cause problems. However, recalling Remark \ref{rmk:comp:col}, if $e\in E$ is contained in a free component, $\sigma_e=(\fff_e,e)$ completely determines the colors of adjacent edges, $(\sigma_{e^\prime})_{e^\prime \in \delta v(e)\setminus e\,\sqcup\, \delta a(e)\setminus e}$. Moreover, the condition $\norm{\xi-\xi^\star}_{\tsq}=o_n(1)$ guarantees that the number of boundary spins $\{\rr,\bb,\fs\}$ are at least $\Omega_k(n)$. Therefore, Remark \ref{rmk:comp:col} rules out the difficulties coming from the rare spins, and this is the sole reason we work with the more complicated component colorings instead of the simpler `coloring configuration' defined in \cite{ssz22}.

   \begin{proof}[Proof of Theorem \ref{thm:main}]
   It suffices to prove \eqref{eq:thm:equiv} for any fixed $(\uz_t, \uL_t)\in \{0,1\}^{V(\TTT_{d,k,t})}\times \{0,1\}^{E_{\sf in} (\TTT_{d,k,t})}$. To this end, fix any $\eps>0$. Then, recalling the set $\Xi_C$ from Definition \ref{def:Xi}, there exists $C=C(\eps,k,t)$ such that
   \begin{equation}\label{eq:Xi:whp}
   \P\left(\xi[\bGG,\bux]\rev{\notin} \Xi_C\right)\leq \frac{\eps}{3}\,.
   \end{equation}
   \rev{This is because} for large enough $C>0$, Theorem \ref{thm:concentration} guarantees that $\norm{\xi[\bGG,\bux]-\xi^\star}_{\tsq}\leq C n^{-1/2}$ holds with probability at least $1-\eps/6$. \rev{Moreover, since} $\E\big[\E\big[N_{\cyc}(2t;\GG)\given \xi[\bGG,\bux]\big]\big]=\E[N_{\cyc}(2t;\GG)]\lekt 1$ holds by tower property, Markov's inequality shows that $\E_{\xi}\big[N_{\cyc}(2t;\GG)\big]\leq C$ is satisfied with probability at least $1-\eps/6$.

   Now, with $C=C(\eps,k,t)$ that satisfies \eqref{eq:Xi:whp}, Theorem \ref{thm:coupling} \rev{yields} that there exists $K=K(\eps,t,k)$ such that 
   \begin{equation}\label{eq:fav:1}
   \sup_{\xi\in \Xi_C}\P_{\xi}\bigg(\Big|\big\langle \nu_t[\bGG, \bux] -\nu_t^\star \,,\, p_{\uz_t,\uL_t} \big\rangle\Big|\geq \frac{K}{\sqrt{n}}\bigg)\leq \frac{\eps}{3}\,,
   \end{equation}
   where $p_{\uz_t,\uL_t}\in [0,1]^{\Omega_t\sqcup\{\cyc\}}$ is defined by  $p_{\uz_t,\uL_t}(\sig_t,\uL^\prime_t)=p_{\uz_t}(\sig_t,\uL_t)\one(\uL^\prime_t=\uL_t)\in [0,1]$ for $(\sig_t,\uL^\prime_t)\in \Omega_t$ (see Remark \ref{rmk:p} for the definition of $p_{\uz_t}(\sig_t,\uL_t)$), and $p_{\uz_t,\uL_t}(\cyc)\equiv 0$. Moreover, the result of Appendix B in \cite{ssz22} imply that $\PP^t_{\star}$ defined in Section \ref{s:lwl} satisfies $\PP_\star^t(\uz_t,\uL_t)=\sum_{\sig_t\in \CCC_t} \nu_t^\star(\sig_t,\uL_t)p_{\uz_t}(\sig_t,\uL_t)$, and for completeness we have provided its proof in Appendix~\ref{sec:appendix:B}. To this end, we define $\fav \equiv \fav(\eps,k,t)$ to be the favorable set of $(\GG,\ux)$ which satisfies the following 4 conditions.
   \begin{equation*}
       (i)~\xi[\GG,\ux]\in \ee_{\frac{1}{4}}~~~(ii)~\xi[\GG,\ux]\in \Xi_C~~~ (iii)~N_{\cyc}^{\sf b}(2t;\GG,\ux)\leq n^{1/3}~~~~(iv)~\big|\big\langle \nu_t[\bGG, \bux]\,,\, p_{\uz_t,\uL_t} \big\rangle-\PP_{\star}^t(\uz_t,\uL_t)\big|\leq \frac{K}{\sqrt{n}}
   \end{equation*}
    Then, Lemma \ref{lem:exp:decay}, Lemma \ref{lem:bdry:cycle}, and equations \eqref{eq:Xi:whp}, \eqref{eq:fav:1} show that \begin{equation*}
    \P\big((\GG,\ux)\notin \fav\big)\leq \rev{\frac{2\eps}{3}}+o_n(1)\,.
    \end{equation*}
    \rev{Note that by} the first condition $\xi[\GG,\ux]\in \ee_{\frac{1}{4}}$, the variance control for $\fav$ in \eqref{eq:goal:variance} holds by Lemma \ref{lem:var:control}. With regards to the bias control \eqref{eq:goal:bias}, let $V_{\g}[\GG,\ux]\subset V$ be the set of $v\in V$ such that $N_t(v,\GG)$ is a tree and \rev{that} $N_t(v,\GG)\cap \fff$ is connected \rev{for every $\fff\in \FFF(\ux,\GG)$}. Then for $v\in V_\g[\GG,\ux]$, we have by Remark \ref{rmk:p} that $\E\big[X_i^t[\bGG,\ubz, \uz_t,\uL_t] \;\big|\; (\bGG,\bux)=(\GG,\ux)\big]=p_{\uz_t,\uL_t}\big(\sig_t(i,\GG),\uL_{N_t(i,\GG)}\big)$. Thus, for $(\GG,\ux)\in \fav$,
   \begin{equation*}
   \bigg|\frac{1}{n}\E\Big[\sum_{i=1}^{n}X_i^t\bgiven (\bGG,\bux)=(\GG, \ux) \Big]-\PP_{\star}^{t}(\uz_t,\uL_t)\bigg| \leq \frac{K}{\sqrt{n}}+\frac{\big|(V_{\g}[\GG,\ux])^{\textsf{c}}\big|}{n}\,.
   \end{equation*}
   \rev{Observe that} if $v\notin V_{\g}[\GG,\ux]$, then \rev{there exists a cycle $\CC$ that overlaps with $N_t(v,\GG)$, and have at most $2t$ edges that are not contained in free components of $\ux$} (see the paragraph above Lemma \ref{lem:bdry:cycle}). Thus, we have for $(\GG,\ux)\in \fav$ that
   \rev{
   \begin{equation*}
   \big|(V_{\g}[\GG,\ux])^{\textsf{c}}\big|\leq 2t(kd)^t \cdot \big|N_{\cyc}^{\sf b}(2t;\GG,\ux)\big|\lekt n^{1/3}\,.
   \end{equation*}
   }
   Therefore, both \rev{of the} conditions \eqref{eq:goal:variance} and \eqref{eq:goal:bias} hold on the set $\fav$ with $\P((\GG,\ux)\notin \fav)\leq \eps$. By Chebyshev's inequality, this concludes the proof of \eqref{eq:thm:equiv}, \rev{proving} Theorem \ref{thm:main}.
   \end{proof}

    \section{Concentration of coloring profile in $\ell^1$-type distance}
    \label{sec:concentration:free:tree}
    In this section, we prove Lemma \ref{lem:exp:decay} and Theorem \ref{thm:concentration}. While Lemma \ref{lem:exp:decay} follows from the results of \cite{NSS, nss2}, Theorem \ref{thm:concentration} requires careful control \rev{of} the frequencies of large free trees. Throughout, we consider $k$ large enough so that the results of \cite{NSS, nss2} hold and \rev{let} $\alpha \in (\alpha_{\cond}(k),\alpha_{\sat}(k))$.
    
    \subsection{Exponential decay of free component profile}
    Recall the notation of coloring profile $\xi[\GG,\ux]:=(B[\GG, \ux], \{p_{\fff}[\ux]\}_{\fff\in \FFF})$ from Definition \ref{def:empirical:boundary}. We denote $s_\circ(C)\equiv s_{\circ}(C,n)\equiv s^\star-\frac{1}{2\la^\star}\frac{\log n}{n}+\frac{C}{n}$ for $C\in \R$, where $s^\star\equiv s^\star(\alpha,k)>0$ and $\la^\star\equiv \la^\star(\alpha,k)\in (0,1)$ are defined in \eqref{eq:opt:la}. Throughout, we let $(\bGG, \bux)$ be a random frozen configuration drawn with probability proportional to its size (cf. $(a)$ of Observation \ref{obs:sampling}). The proof of Lemma \ref{lem:exp:decay} \rev{follows from} the \rev{p}roposition below, which is established using the estimates from \cite{NSS, nss2}.
    \begin{prop}\label{prop:nss}
    For $\uC\equiv (C_1,C_2) \in \R^2$ with $C_2>0$, define $\Gamma(\uC) \equiv \Gamma_{n}(\uC)$ by the set of frozen configurations $(\GG,\ux)$ which satisfy the following conditions:
    \begin{itemize}
        \item $\left|\sol(\GG)\right|\in [e^{n s_{\circ}(-C_2)}, e^{n s_{\circ}(C_2)]}$ and $\size(\ux;\GG)\geq e^{ns_{\circ}(C_1)}$.
        \item Let $\Xi_0$ be the set of coloring profile $\xi=(B,\{p_{\fff}\}_{\fff\in \FFF})$ such that $\xi\in \ee_{\frac{1}{4}}, \sum_{\fff\in \FFF\setminus \FFF_{\tr}} p_{\fff}[\GG,\ux]\leq \frac{\log n}{n}$, and $p_{\fff}[\GG,\ux]=0$ if $\fff\in \FFF$ is multi-cyclic, i.e. contains at least $2$ cycles. Then, $\xi(\GG,\ux) \in \Xi_0$ holds.
    \end{itemize}
    Then, for any $\eps>0$, there exists $\uC\equiv \uC(\eps,\alpha,k)$ such that $(\bGG,\bux)\in \Gamma(\uC)$ with probability at least $1-\eps$.
    \end{prop}
    \begin{proof}
        Fix $\eps>0$. By Theorem 1.1 of \cite{NSS}, there exists $C_2\equiv C_2(\eps,\alpha,k)>0$ such that
        \begin{equation*}
            \P\Big(\big|\sol(\bGG)\big|\notin \big[e^{n s_\circ(-C_2)},e^{n s_\circ(C_2)}\big]\Big)\leq \frac{\eps}{3}\,.
        \end{equation*}
         Moreover, since $\P(\bux=\ux\given \bGG)=\size(\ux;\bGG)/|\sol(\GG)|$, we have
        \begin{equation}\label{eq:tech:1}
            \P\Big(\size(\bux;\bGG)\leq e^{ns_\circ(C_1)}~\,\textnormal{and}~~ \big|\sol(\bGG)\big|\in \big[e^{n s_\circ(-C_2)},e^{n s_\circ(C_2)}\big]\bgiven \bGG \Big)\leq e^{-ns_\circ(-C_2)} \overline{\bZ}_{1,s_{\circ}(C_1)}\,,
        \end{equation}
        where 
        \begin{equation*}
            \overline{\bZ}_{1,s_{\circ}(C_1)}\equiv \overline{\bZ}_{1,s_{\circ}(C_1)}(\bGG):= \sum_{\ux\in \{0,1,f\}^{V}}\size(\ux;\bGG)\one\big(\size(\ux;\bGG)\leq e^{n s_{\circ}(C_1)}\big)\,.
        \end{equation*}
        In the proof of Theorem 1.1-$(a)$ in \cite{NSS} (see equation (3.79) therein), it was shown that
        \begin{equation*}
            \E  \overline{\bZ}_{1,s_{\circ}(C_1)}\lesssim_{k} n^{-\frac{1}{2\la^\star}}e^{ns^\star+(1-\la^\star)C_1}\,.
        \end{equation*}
        We remark that in \cite{NSS}, they introduced a truncation of free and red variables, but this truncation only induces a difference that is exponentially small in $n$ (see Lemma 2.12 of \cite{NSS}, or Lemma 3.3 of \cite{ssz22}). Thus, taking expectation in \eqref{eq:tech:1} shows that
        \begin{equation*}
             \P\Big(\size(\bux;\bGG)\leq e^{ns_\circ(C_1)}~\,\textnormal{and}~~ \big|\sol(\bGG)\big|\in \big[e^{n s_\circ(-C_2)},e^{n s_\circ(C_2)}\big]\Big)\leq e^{(1-\la^\star)C_1+C_2}\leq \frac{\eps}{3}\,,
        \end{equation*}
        where we took $C_1\equiv C_1(\eps,\alpha,k)$ small enough so that $e^{(1-\la^\star)C_1+C_2}\leq \eps/3$.
        
        We now establish the second condition of $\Gamma(\uC)$. Let
        \begin{equation*}
        \overline{\bZ}_{\la^\star}\big[(\Xi_0)^{\textsf c}\big]\equiv \overline{\bZ}_{\la^\star}\big[(\Xi_0)^{\textsf c}\big](\bGG):=\sum_{\ux\in \{0,1,f\}^{V}}\size(\ux;\bGG)^{\la^\star}\one\big(\,\xi(\bGG,\ux)\notin \Xi_0\,\big)\,, 
        \end{equation*}
        and similarly, let $\overline{\bZ}_{\la^\star}\equiv \overline{\bZ}_{\la^\star}(\bGG)$ defined by the same equation, but without the indicator term above. Then, it was shown in Proposition 3.5 of \cite{NSS} that
        \begin{equation*}
            \E \overline{\bZ}_{\la^\star}\big[(\Xi_0)^{\textsf c}\big]\lesssim_k \frac{\log^3 n}{n^2} \E \overline{\bZ}_{\la^\star}\lesssim_k \frac{\log^3 n}{n^2} e^{n\la^\star s^\star}\,,
        \end{equation*}
        where the final estimate is from Corollary 3.6 and Theorem 3.21 of \cite{NSS}. Thus, proceeding in a similar manner as in \eqref{eq:tech:1}, we have
        \begin{equation*}
            \P\Big(\xi(\bGG,\bux)\notin \Xi_0~\,\textnormal{and}~~ \big|\sol(\bGG)\big|\in \big[e^{n s_\circ(-C_2)},e^{n s_\circ(C_2)}\big]\Big)\leq \sqrt{n}e^{n\la^\star s^\star+(2-\la^\star)C_2}\cdot \E\overline{\bZ}_{\la^\star}\big[(\Xi_0)^{\textsf c}\big]=o_n(1)\,.
        \end{equation*}
        Therefore, this concludes the proof.
    \end{proof}
    \begin{proof}[Proof of Lemma \ref{lem:exp:decay}]
    For $\eps>0$, consider $\Gamma(\uC)$ for $\uC\equiv \uC(\eps,\alpha,k)$ from the conclusion of Proposition \ref{prop:nss}. Since $\Gamma(\uC)$ contains the coloring profile with exponential decay and the absence of multi-cyclic free components, we have
    \begin{equation*}
        \limsup_{n\to\infty} \P\Big(\xi[\bGG,\bux]\notin \ee_{\frac{1}{4}}~~\textnormal{or}~~p_{\fff}[\bGG,\bux]\neq 0~~\textnormal{for some multi-cyclic}~\fff\in \FFF\Big)\leq \limsup_{n\to\infty}\P\Big((\bGG,\bux)\notin \Gamma(\uC)\Big)\leq \eps\,.
    \end{equation*}
    Since $\eps>0$ is arbitrary, the \textsc{lhs} must equal $0$.
    \end{proof}
    
    Finally, the concentration of the boundary profile $B[\bGG,\bux]$ can also be established using the first moment estimates from Section 3 of \cite{NSS}. We first recall the following notation from \cite{NSS}: for $\la\in [0,1]$ and $s\in [0,\log 2)$,
    \begin{equation}\label{eq:def:Z:lambda}
    \begin{split}
    &\bZ_{\la}:=	\sum_{\ux \in \{0,1,\ff\}^V}\size(\ux;\bGG)^\lambda \one\left\{ \frac{\rr(\ux)}{nd} \vee \frac{\ff(\ux)}{n}\leq \frac{7}{2^k}\right\}\,,\\
    &\bZ_{\la, s}:=	\sum_{\ux \in \{0,1,\ff\}^V}\size(\ux;\bGG)^\lambda \one\left\{ \frac{\rr(\ux)}{nd} \vee \frac{\ff(\ux)}{n}\leq \frac{7}{2^k}, ~e^{ns} \le \textsf{size}(\ux;\GG) < e^{ns+1}\right\}\,,
    \end{split}
    \end{equation}
    where $\rr(\ux)$ denotes the number of forcing edge in $(\bGG,\ux)$ and $\ff(\ux)$ denotes the number of free variables in $\ux$. Similarly, we define the quantities $\bZ_{\la}^{\tr}$ and $\bZ_{\la,s}^{\tr}$ by the contribution of the sum in \eqref{eq:def:Z:lambda} from frozen configurations $\ux$ that does not contain cyclic free components, i.e. $\FFF(\ux,\bGG)\subset \FFF_{\tr}$. Moreover, denote by $\bZ_{\la}[B]$ (resp. $\bZ_{\la,s}[B]$) the contribution to $\bZ_{\la}$ (resp. $\bZ_{\la,s}$) from frozen configuration\rev{s} $\ux$ having boundary profile $B[\bGG,\ux]=B$. The quantities $\bZ_{\la}^{\tr}[B]$ and $\bZ_{\la,s}^{\tr}[B]$ are analogously defined.  
 
    \begin{prop}\label{prop:boundary}
    There exists the \textit{optimal boundary profile} $B^\star\equiv B^\star[\alpha,k]$ so that the following holds. For any $\eps>0$, there exists $C\equiv C(\eps,\alpha,k)>0$ such that uniformly over $|s-s^\star|\leq n^{-2/3}$, we have
    \begin{equation*}
    \E\bZ_{\la^\star,s}\bigg[\norm{B-B^\star}_1\geq \frac{C}{\sqrt{n}}\bigg]\leq \eps \cdot \E\bZ_{\la^\star, s}\,.
    \end{equation*}
    \end{prop}
    \begin{proof}
        The proof follows from Section 3 of \cite{NSS}, but we include it for completeness. By equations (3.66) and (3.71) in the proof of Proposition 3.23 of \cite{NSS}, we have that uniformly over $|s-s^\star|\leq n^{-2/3}$,
        \begin{equation*}
             \E\bZ_{\la^\star,s}\bigg[\norm{B-B^\star}_1\geq n^{-1/3} \bigg]=o_n(1)\cdot \E\bZ_{\la^\star, s}\,.
        \end{equation*}
        On the other hand, the equations (3.66) and (3.77) in \cite{NSS} show that uniformly over $\norm{B-B^\star}_1\leq n^{-1/3}$ and $|s-s^\star|\leq n^{-2/3}$, we have
        \begin{equation*}
            \E \bZ_{\la^\star,s}[B]=(1+o_n(1))C_1(\alpha,k)\cdot \E\bZ_{\la^\star,s}^{\tr}[B]\,,
        \end{equation*}
        for some $C_1(\alpha,k)>0$ ($C_1(\alpha,k)=e^{\xi^{\textnormal{uni}}(B^\star,s^\star)}$ in \cite{NSS}). Moreover, Lemma 3.16 and Proposition 3.17 of \cite{NSS} imply that (see equations (3.64) and the one below in \cite{NSS}) uniformly over $|s-s^\star|\leq n^{-2/3}$,
        \begin{equation*}
            \E\bZ_{\la^\star,s}^{\tr}\bigg[\norm{B-B^\star}_1\geq \frac{C}{\sqrt{n}}\bigg]\lesssim_k \exp\left(-\Omega_k(C{\,^2})\right)\E\bZ_{\la^\star,s}^{\tr}\,.
        \end{equation*}
        Since $\E\bZ_{\la^\star,s}^{\tr}\leq \E \bZ_{\la^\star,s}$, taking $C$ large enough completes the proof.
    \end{proof}

    \subsection{Concentration of free tree profile}
      Having established the concentration of boundary profile in Proposition \ref{prop:boundary}, we now establish the concentration of \rev{the} free component profile as in Theorem \ref{thm:concentration}. By Proposition \ref{prop:nss}, we have with high probability that the number of cyclic free components is at most $\log n$ and \rev{that} the largest free component is of size $O_k(\log n)$, thus all the interest is in the concentration of the free tree profile.

      Proposition 3.8 in \cite{NSS} shows that different free components are \textit{independent} in \rev{an} appropriate measure, which enables us to use \rev{the} local central limit theorem for triangular array \cite{Borokov17}. We first introduce the necessary notations: for a coloring profile $\xi =(B,\{p_{\fff}\}_{\fff\in \FFF})$, let $\bZ_{\la}[\xi]\equiv \bZ_{\la}[B, \{p_{\fff}\}_{\fff\in \FFF}]$ denote the contribution to $\bZ_{\la}$ from frozen configuration\rev{s} $\ux$ with coloring profile $\xi[\bGG,\ux]=\xi$. Note that in order for $\bZ_{\la}[\xi]\neq 0$, $B$ and $\{p_{\fff}\}_{\fff\in \FFF}$ must be \textit{compatible} in the following sense (see Definition 3.2 in \cite{NSS}).
      \begin{defn}\label{def:compat}
       For a free component $\fff$, let $\mathfrak{b}_{\fff}(\bb_x), x\in \{0,1\},$ count the number of boundary half-edges $e\in \dot{\partial}\fff$ such that the spin-label equals $x$. Thus, $\fb_{\fff}(\bb_0)+\fb_{\fff}(\bb_1)=|\dot{\partial}\fff|$ holds. Further denote $\fb_{\fff}(\fs)=|\hat{\partial}\fff|$. We define the vector $\ufb_{\fff}\in \N^3$ associated with $\fff\in \FFF$ as $\ufb_{\fff}:=(\,\fb_{\fff}(\bb_0),\,\fb_{\fff}(\bb_1),\,\fb_{\fff}(\fs)\,)$. Also, let $\uh\equiv \uh[\{p_{\fff}\}_{\fff\in \FFF}] =( h_{\bb_0}, \, h_{\bb_1}, \, h_{\fs})\in \R^{3}$ be $h_{\sigma}=\sum_{\fff\in \FFF}\fb_{\fff}(\sigma)p_{\fff}$ for $\sigma\in \{\bb_0,\bb_1,\fs\}$. Then, the free component profile $\{p_{\fff}\}_{\fff\in \FFF}$ is \textbf{compatible} with the boundary profile $B$, which we denote by $\{p_{\fff}\}_{\fff\in \FFF}\sim B$, if for $\tau\in \{\bb,\fs\}$,
\begin{equation}\label{eq:compatible}
	\begin{split}
	\bar{B}(\tau)&=\frac{1}{d} \sum_{\underline{\sigma}\in\{\rr,\bb\}^d } \dot{B}(\underline{\sigma})  \sum_{i=1}^d \one(\sigma_i =\tau)+ \frac{\one(\tau=\fs)}{d} h(\tau)\\
   &=\frac{1}{k}
	\sum_{\underline{\sigma}\in \{\rr,\bb,\fs\}^k} \hat{B} (\underline{\sigma})
	\sum_{j=1}^k \one(\sigma_j =\tau)
	+
	\frac{\one(\tau\in \{\bb_0,\bb_1\})}{d}  h(\tau)\, ,
	\end{split}
	\end{equation}
    and we have
    \begin{equation}\label{eq:compat:var:clause}
        \sum_{\fff\in \FFF}p_{\fff}v_{\fff} = 1-\langle \dot{B}, \textbf{1} \rangle\,,\quad\quad\sum_{\fff\in \FFF}p_{\fff}f_{\fff} = \frac{d}{k}(1-\langle \hat{B}, \textbf{1} \rangle)\,,\quad\quad \sum_{\fff\in \FFF}p_{\fff}e_{\fff} = d(1-\langle \bar{B}, \textbf{1} \rangle)\,,
    \end{equation}
    where $\textbf{1}$ denotes the all-$1$ vector. With slight abuse of notation\rev{s}, we call $\uh\equiv \uh[B]\in \R^3$ the \textbf{induced boundary profile}, \rev{which is determined from the equation \eqref{eq:compatible}}. We also remark that the last equality in \eqref{eq:compat:var:clause} is actually redundant, since it is implied by the other equalities in \eqref{eq:compatible} and \eqref{eq:compat:var:clause}.
    \end{defn}
    \begin{remark}\label{rmk:number:freetrees}
        \rev{Suppose that} $\{p_{\fff}\}_{\fff\in \FFF}\sim B$ and \rev{that} $p_{\fff}=0$ holds if $\fff$ is multi-cyclic. \rev{Then,} \eqref{eq:compat:var:clause} implies that
        \begin{equation}\label{eq:h:circ}
            h_{\circ}\equiv 1-\langle \dot{B}, \one \rangle+ \frac{d}{k}(1-\langle \hat{B}, \one \rangle) -d(1-\langle \bar{B}, \one \rangle)= \sum_{\fff\in \FFF}p_{\fff}(v_{\fff}+f_{\fff}-e_{\fff})=\sum_{\ttt\in \FFF_{\tr}}p_{\ttt}\,.
        \end{equation}
        Thus, $h_{\circ}\equiv h_{\circ}[B]$ determines the number of free trees $\sum_{\ttt\in \FFF_{\tr}}n_{\ttt}$ when there is no multi-cyclic free component.
    \end{remark}
    \rev{Next, we} define the set of \textit{labeled component} $\LLL(\fff)$ of a free component $\fff \in \FFF$ (see Definition 2.20 in \cite{NSS}). A labeled component $\fff^{\lab}$ is obtained from $\fff$ by adding additional labels on the half-edges and $\fff$ as follows. For $a\in F(\fff)$ (resp. $v \in V(\fff)$ ), arbitrarily label half-edges adjacent to $a$ (resp. $v$) by $1,...,k$ (resp. $1,...,d$). If $\fff$ is cyclic, then add an additional label by first choosing a spanning tree of $\fff$ and labeling their edges by ``tree''. We consider $\fff^{\lab}$ up to graph and label isomorphism, so $|\LLL(\fff)|$ counts the number of labeled components up to such isomorphism. If we let $T_{\fff}$ be the number of spanning trees of $\fff$ ($T_{\ttt}=1$), then, the \textbf{embedding number} of $\fff$ is defined by
		\begin{equation*}
		J_\fff := d^{1-v(\fff)} k^{-f(\fff)} \frac{|\mathscr{L}(\fff)|}{T_\fff}\;.
		\end{equation*}

    \begin{prop}[Proposition 3.7 in \cite{NSS}]
    \label{prop:moment:formula}
        Fix $\la\in [0,1]$. For a coloring profile $\xi=(B, \{p_{\fff}\}_{\fff\in \FFF})$ such that $\{p_{\fff}\}_{\fff\in \FFF}\sim B$, we have
        \begin{equation*}
              \E \bZ_\lambda [B, \{p_\fff \}_{\fff\in \mathscr{F}}] = \frac{n! m!}{nd !} \frac{(nd\bar{B})!}{(n\dot{B})! (m\hat{B})!}\prod_{\sig\in \{\rr,\bb,\fs\}^k}\hat{v}(\sig)^{m\hat{B}(\sig)} \prod_{\fff\in \mathscr{F}}\left[ \frac{1}{(n p_\fff)!} \Big(d^{e_{\fff}-f_{\fff}}k^{f_{\fff}}J_\fff w_\fff^{\la} \Big)^{n p_\fff}\right]\,,
        \end{equation*}
        where $\hat{v}(\sig)$ for $\sig \in \{\rr,\bb,\fs\}^k$ is defined in \eqref{eq:def:vhat:basic} below.
    \end{prop}
    \begin{remark}
    There is a slight difference between Proposition \ref{prop:moment:formula} and Proposition 3.7 of \cite{NSS} because of the difference \rev{in} the notion of `free tree'. Namely, the `free tree' in \cite{NSS} corresponds to an equivalence class of our notion of free trees. However, the proof of Proposition 3.7 of \cite{NSS} proceeds exactly by first establishing Proposition \ref{prop:moment:formula}, and \rev{summing up the contribution from the free trees in the same equivalence class}. 
    \end{remark}
    We now define the \textit{optimal free tree profile} and \textit{optimal coloring profile} $\xi^\star$ used in \rev{Theorem} \ref{thm:concentration}.
    \begin{defn}\label{def:optimal:coloring}
    For a free tree $\ttt \in \FFF_{\tr}$, the \textbf{optimal free tree profile} $\{p^\star_{\ttt}\}_{\ttt\in \FFF_{\tr}}$ is defined by
    \begin{equation*}
	     p_{\ttt}^\star := \frac{J_{\ttt} w_{\ttt}^{\lambda^\star}}{\bar{\mathfrak{Z}}^\star (\dot{\ZZZ}^\star)^{v_{\ttt}}(\hat{\ZZZ}^\star)^{f_{\ttt}}}\dot{q}^\star(\bb_0)^{|\dot{\partial}\ttt|}(2^{-\lambda}\hat{q}^\star(\fs))^{|\hat{\partial}\ttt|}\,,
    \end{equation*}
    where the quantities $\bar{\mathfrak{Z}}^\star, \dot{\ZZZ}^\star, \hat{\ZZZ}^\star, \dot{q}^\star(\bb_0), \hat{q}^\star(\fs)$ are defined in Definition \ref{def:opt:bdry:1stmo}. For a cyclic free component $\fff\in \FFF\setminus \FFF_{\tr}$, we define $p_{\fff}^\star=0$. Then, \rev{recalling} the optimal boundary profile $B^\star$ in Proposition \ref{prop:boundary}, we define the \textbf{optimal coloring profile} $\xi^\star$ to be $\xi^\star\equiv (B^\star, \{p_{\fff}^\star\}_{\fff\in \FFF}$\rev{)}.
    
    An important observation is that $p_{\ttt}^\star$ can be expressed by
    \begin{equation}\label{eq:theta:opt}
	     p^\star_{\ttt}= J_{\ttt}w_{\ttt}^{\lambda^\star} \exp\left(\big\langle\,\utheta^\star\,,\,(1, \ufb_{\ttt})\,\big \rangle\right)\,,
	\end{equation}
    where the vector $\utheta^\star\equiv (\theta^\star_{\circ},\theta^\star_{\bb_0},\theta^\star_{\bb_1},\theta^\star_{\fs})\in \R^4$ is defined by
    \rev{
    \begin{equation*}
        \begin{split}
            &\theta^\star_{\circ}\equiv \log\big(\frac{(\dot{\ZZZ}^\star)^{\frac{k}{kd-k-d}}(\hat{\ZZZ}^\star)^{\frac{d}{kd-k-d}}}{\bar{\mathfrak{Z}}^\star}\big)\,,~~~\theta^\star_{\bb_0}\equiv \theta^\star_{\bb_1}\equiv \log\big(\frac{\dot{q}^\star(\bb_0)}{(\dot{\ZZZ}^\star)^{\frac{1}{kd-k-d}}(\hat{\ZZZ}^\star)^{\frac{d-1}{kd-k-d}}}\big)\,,\\&\theta^\star_{\fs}\equiv\log\big(\frac{2^{-\la}\hat{q}^\star(\fs)}{(\dot{\ZZZ}^\star)^{\frac{k-1}{kd-k-d}}(\hat{\ZZZ}^\star)^{\frac{1}{kd-k-d}}}\big)\,.
        \end{split}
    \end{equation*}
    }
    \end{defn}
    \begin{remark}\label{rmk:optimal:compatible}
    It was shown in Lemma B.2 of \cite{NSS} that the optimal free tree profile $\{p_{\ttt}^\star\}$ and the optimal boundary profile $B^\star$ are compatible. That is, $h^\star\equiv (h^\star_{\bb_0},h^\star_{\bb_1}, h^\star_{\fs})$ and $h^\star_\circ$ induced from $B^\star$ by the equations \eqref{eq:compatible} and \eqref{eq:h:circ} satisfy $\sum_{\ttt\in \FFF_{\tr}}p_{\ttt}^\star = h^\star_{\circ}$ and $\sum_{\ttt\in \FFF_{\tr}}\fb_{\ttt}(\sigma)p_{\ttt}^\star =h^\star_{\sigma}$ for $\sigma\in \{\bb,\fs\}$.
    \end{remark}
    The following proposition shows the $\ell^1$-type concentration of the free tree profile conditional on the $(B, s, \{p_{\fff}\}_{\fff\in \FFF\setminus \FFF_{\tr}})\in \Psi_{\typ}$, where $\Psi_{\typ}\equiv \Psi_{\typ}(C)$ is defined by the set of $(B, s, \{p_{\fff}\}_{\fff\in \FFF\setminus \FFF_{\tr}})$ which satisfy
    \begin{equation}\label{eq:def:Psi}
        \norm{B-B^\star}_1\leq \frac{C}{\sqrt{n}}\,,~~|s-s^\star|\leq n^{-2/3}\,,~~\sum_{\fff\in \FFF\setminus \FFF_{\tr}}p_{\fff}\leq \frac{\log n}{n}\,,~~\textnormal{and}~~p_{\fff}=0 \textnormal{ if $\fff$ is multi-cyclic or $v(\fff)\geq \log n$.}
    \end{equation}
    \begin{prop}\label{prop:freetree:concentration}
    For any $\eps>0$ and a constant $C>0$, there exists another constant $C_0$, which only depends on $\eps, C,\alpha,k$, such that uniformly over $(B, s, \{p_{\fff}\}_{\fff\in \FFF\setminus \FFF_{\tr}})\in \Psi_{\typ}(C)$,
    \begin{equation}\label{eq:prop:freetree:concentration}
    \sum_{\{p_{\ttt}\}_{\ttt\in \FFF_{\tr}}:\sum_{\ttt}\big|p_{\ttt}-p_{\ttt}^\star\big|v_{\ttt}\geq \frac{C_0}{\sqrt{n}}}\E\bZ_{\la^\star,s}\Big[B,\{p_{\ttt}\}_{\ttt\in \FFF_{\tr}}, \{p_{\fff}\}_{\FFF\setminus \FFF_{\tr}}\Big]\leq \eps\cdot \sum_{\{p_{\ttt}\}_{\ttt\in \FFF_{\tr}}}\E\bZ_{\la^\star,s}\Big[B,\{p_{\ttt}\}_{\ttt\in \FFF_{\tr}}, \{p_{\fff}\}_{\FFF\setminus \FFF_{\tr}}\Big]\,.
    \end{equation}
    \end{prop}
    \begin{remark}\label{rmk:clauses}
    Note that by the definition of $\Psi_{\typ}$ in \eqref{eq:def:Psi}, every cyclic free component has one cycle and has variables $v_{\fff}<\log n$. Since every clause in a free component must have at least 2 full edges adjacent to them (otherwise, the clause is forcing), it follows that $2f_{\fff}\leq e_{\fff}$. For $\fff\in \FFF$ having at most one cycle, this implies that $f_{\fff}\leq e_{\fff}$. In particular, we may assume that $f_{\fff}<\log n$ holds if $p_{\fff}\neq 0$ and $(B, s, \{p_{\fff}\}_{\fff\in \FFF\setminus \FFF_{\tr}})\in \Psi_{\typ}$.
    \end{remark}
    We now discuss the main ideas behind the proof of Proposition \ref{prop:freetree:concentration}. By Proposition \ref{prop:moment:formula}, we can express\footnote{From \eqref{eq:compat:var:clause}, the term $\prod_{\fff\in \FFF}(d^{e_{\fff}-f_{\fff}}k^{f_{\fff}})^{np_{\fff}}$ in Proposition \ref{prop:moment:formula} cancels out\rev{.}}
    \begin{equation}\label{eq:useful:identity}
        \frac{\E\bZ_{\la^\star,s}\Big[B,\{p_{\ttt}\}_{\ttt\in \FFF_{\tr}}, \{p_{\fff}\}_{\FFF\setminus \FFF_{\tr}}\Big]}{ \sum_{\{p_{\ttt}\}_{\ttt\in \FFF_{\tr}}}\E\bZ_{\la^\star,s}\Big[B,\{p_{\ttt}\}_{\ttt\in \FFF_{\tr}}, \{p_{\fff}\}_{\FFF\setminus \FFF_{\tr}}\Big]}=\frac{\prod_{\fff\in \mathscr{F}_{\tr}}\left[ \frac{1}{(n p_\fff)!} \left(J_\fff w_\fff^\lambda \right)^{n p_\fff}\right]}{\sum_{\{p_{\ttt}\}_{\ttt\in \FFF_{\tr}}}\prod_{\fff\in \mathscr{F}_{\tr}}\left[ \frac{1}{(n p_\fff)!} \left(J_\fff w_\fff^\lambda \right)^{n p_\fff}\right]}\,,
    \end{equation}
    where the sum in the denominator in the \textsc{rhs} is restricted to sum over $\{p_{\ttt}\}_{\ttt\in \FFF_{\tr}}$ which satisfy $\{p_{\fff}\}_{\fff\in \FFF}\equiv\{\{p_{\ttt}\}_{\ttt\in \FFF_{\tr}}, \{p_{\fff}\}_{\fff\in \FFF\setminus \FFF_{\tr}}\}\sim (B,s)$, where
    \begin{equation*}
        \{p_{\fff}\}_{\fff\in \FFF}\sim (B,s)\iff \{p_{\fff}\}_{\fff\in \FFF}\sim B~~\textnormal{and}~~\sum_{\fff\in \FFF}p_{\fff} \log w_{\fff}\in [s, s+1/n)\,.
    \end{equation*}
    The rightmost condition comes from the fact that we defined $\bZ_{\la,s}$ in \eqref{eq:def:Z:lambda} as the contribution to $\bZ_{\la}$ from $\ux$ such that $\size(\ux,\bGG)=\prod_{\fff\in \FFF}w_{\fff}\in[e^{ns},e^{ns+1})$.

    Thus, a priori, the \textsc{rhs} of \eqref{eq:useful:identity} is a probability of a configuration of free trees conditional on a large deviation event. A classical tool in large deviation theory \cite{DZ10} is to introduce an exponential scaling factor to move from the large deviation regime to a moderate deviation regime. Since we are considering $\norm{B-B^\star}_1=O(n^{-1/2})$, we can use the \textit{optimal scaling factor} $\utheta^\star$ from \eqref{eq:theta:opt} (see Remark \ref{rmk:optimal:compatible}). To this end, we introduce a probability distribution over the free trees
    \begin{equation*}
        \P_{\utheta^\star}(X=\ttt)\equiv \frac{J_{\ttt}w_{\ttt}^{\lambda^\star} \exp\left(\big\langle\,\utheta^\star\,,\,(1, \ufb_{\ttt})\,\big \rangle\right)}{h^\star_{\circ}}=(h^\star_{\circ})^{-1}p_{\ttt}^\star\,, 
    \end{equation*}
    and let $X_1,...,X_{nh_{\circ}}$ be a i.i.d. sample from $\P_{\utheta^\star}(\cdot)$, where $h_{\circ}\equiv h_{\circ}[B]$ is defined in \eqref{eq:h:circ}. Recall from Remark \ref{rmk:number:freetrees} that when $p_{\fff}=0$ for multi-cyclic $\fff$, $\{p_{\fff}\}_{\fff\in \FFF}\sim B$ implies that the total number of free trees is given by $nh_{\circ}$. Thus, from the identity \eqref{eq:useful:identity}, the ratio of interest in \eqref{eq:prop:freetree:concentration} can be expressed by
    \begin{equation}\label{eq:crucial:identity}
    \frac{\sum_{\{p_{\ttt}\}_{\ttt\in \FFF_{\tr}}:\sum_{\ttt}\big|p_{\ttt}-p_{\ttt}^\star\big|v_{\ttt}\geq \frac{C_0}{\sqrt{n}}}\E\bZ_{\la^\star,s}\Big[B,\{p_{\ttt}\}_{\ttt\in \FFF_{\tr}}, \{p_{\fff}\}_{\FFF\setminus \FFF_{\tr}}\Big]}{ \sum_{\{p_{\ttt}\}_{\ttt\in \FFF_{\tr}}}\E\bZ_{\la^\star,s}\Big[B,\{p_{\ttt}\}_{\ttt\in \FFF_{\tr}}, \{p_{\fff}\}_{\FFF\setminus \FFF_{\tr}}\Big]}=\P_{\theta^\star}\bigg(\sum_{\ttt\in \FFF_{\tr}}v_{\ttt} \Big|\frac{1}{n}\sum_{i=1}^{nh_{\circ}}\one(X_i=\ttt)-p_{\ttt}^\star \Big|\geq \frac{C_0}{\sqrt{n}}\bbgiven \AAA\bigg),
    \end{equation}
    where the event $\AAA\equiv \AAA(B, s, \{p_{\fff}\}_{\fff\in \FFF\setminus \FFF_{\tr}})$ regarding $X_1,...,X_{nh_{\circ}}$ is given by
    \begin{equation}\label{eq:def:AAA}
        \AAA\equiv \Big\{\sum_{i=1}^{nh_{\circ}}\ufb_{X_i}= n\uh(B)-\sum_{\fff\in \FFF\setminus \FFF_{\tr}}np_{\fff}\ufb_{\fff}~~~~\textnormal{and}~~~~\sum_{i=1}^{nh_{\circ}}\log w_{X_i}+\sum_{\fff\in \FFF\setminus \FFF_{\tr}}np_{\fff}\log w_{\fff}\in [ns,ns+1)\Big\}.
    \end{equation}
    Observe that the event $\AAA$ is a moderate deviation event for $(B,s,\{p_{\fff}\}_{\fff\in \FFF\setminus \FFF_{\tr}})\in \Psi_{\typ}\equiv \Psi_{\typ}(C)$. Indeed, since $\E_{\theta^\star}\ufb_{X}=(h_{\circ}^\star)^{-1}h^\star$ (cf. Remark \ref{rmk:optimal:compatible}), we have uniformly over $\Psi_{\typ}$ that
    \begin{equation*}
        \Big\|nh_{\circ}\E_{\theta^\star}\ufb_{X_i}-n\uh(B)+\sum_{\fff\in \FFF\setminus \FFF_{\tr}}np_{\fff}\ufb_{\fff}\Big\|\lesssim_k n\big(|h_{\circ}(B)-h_{\circ}^\star|+\norm{\uh(B)-\uh^\star}_1\big)+\log^{2}n\lesssim \sqrt{n}\,,
    \end{equation*}
    where the first inequality holds since for $\fff$ that contains at most one cycle, $\norm{\ufb_{\fff}}_1\lesssim_{k} v_{\fff}$ holds (cf. Remark \ref{rmk:clauses}), and the second inequality holds since $B\to\uh(B)$ is $O_k(1)$-Lipschitz (see \eqref{eq:compatible} and \eqref{eq:h:circ}). Similarly, $|nh_{\circ}\E_{\theta^\star}\log w_{X}-ns-\sum_{\fff\in \FFF\setminus \FFF_{\tr}} np_{\fff}\log w_{\fff}|\lesssim_{k}\sqrt{n}$ holds. Thus, local central limit theorem with triangular arrays \cite{Borokov17} show that uniformly over $(B,s,\{p_{\fff}\}_{\fff\in \FFF\setminus \FFF_{\tr}})\in \Psi_{\typ}$,
    \begin{equation}\label{eq:prob:AAA}
    \P_{\utheta^\star}(\AAA)\asymp_k n^{-2}\,.
    \end{equation}
    Therefore, if in the \textsc{rhs} of \eqref{eq:crucial:identity}, the sum \rev{was} replaced by a supremum over free trees and $C_0$ \rev{was} allowed to grow with $n$, say $C_0=\Omega(\log n)$, then we could resort to concentration inequalities (e.g. Chernoff bounds) to show \rev{that} the conditional probability of interest is less than $\eps$. Indeed, this was the strategy taken in \cite{NSS} (see Section 6 therein). 

    However, to obtain the $\ell^1$-type control as in Proposition \ref{prop:freetree:concentration}, we need to take \rev{account} for the sum over free trees and conditioning on $\AAA$ in \eqref{eq:crucial:identity} more carefully. Our strategy is to first divide the $\FFF$ into \textit{typical} and \textit{atypical} set of free trees: for $n\geq 2$, let
    \begin{equation*}
        \FFF_{\tr}^{\typ}\equiv \FFF_{\tr}^{\typ}(n)\equiv \big\{\ttt\in \FFF_{\tr}:p_{\ttt}^\star\geq n^{-1/2}\big\}\,,~~~\FFF_{\tr}^{\atyp}\equiv \FFF_{\tr}^{\atyp}(n)\equiv \FFF_{\tr}\setminus \FFF_{\tr}^{\typ}\,.
    \end{equation*}
    For each free tree $\ttt\in \FFF_{\tr}^{\typ}\sqcup \FFF_{\tr}^{\atyp}$, we assign a \textit{cost} $\eps_{\ttt}$ that is summable and bound the probability that the empirical count of $\ttt$ among $X_i$'s deviate from $p_{\ttt}^\star$ by distance $\frac{\eps_{\ttt}}{\sqrt{n}}$ conditional on $\AAA$. For atypical trees $\ttt\in \FFF_{\tr}^{\atyp}$, we design such cost $\eps_{\ttt}$ crudely based on a Chernoff bound (cf. Lemma \ref{lem:atypical}). On the contrary, for typical trees $\ttt\in \FFF_{\tr}^{\typ}$, we assign $\eps_{\ttt}$ carefully by showing that the conditioning on $\AAA$ has a negligible effect on the probability of interest for typical trees, which we argue by a local central limit theorem for triangular arrays \cite{Borokov17} (cf. Lemma \ref{lem:typical}).
    
    We first start with the easier case of atypical trees. A crucial fact that we use throughout the proof is that $p_{\fff}^\star\in \ee_{\frac{1}{2}}$ holds from Lemma 3.13 in \cite{NSS}.
    \begin{lemma}\label{lem:atypical}
        For $\AAA\equiv \AAA(B, s, \{p_{\fff}\}_{\fff\in \FFF\setminus \FFF_{\tr}})$ in \eqref{eq:def:AAA} and $C>0$, we have uniformly over $(B, s, \{p_{\fff}\}_{\fff\in \FFF\setminus \FFF_{\tr}})\in \Psi_{\typ}(C)$ that
        \begin{equation}\label{eq:lem:atypical}
            \P_{\utheta^\star}\Big(\sum_{\ttt\in \FFF_{\tr}^{\atyp}}v_{\ttt}\Big|\frac{1}{n}\sum_{i=1}^{nh_{\circ}}\one(X_i=\ttt)-p_{\ttt}^\star \Big|\geq n^{-1/2}(\log n)^{-1/2}\bgiven \AAA\Big)=o_n(1)\,.
        \end{equation}
    \end{lemma}
    \begin{proof}
    Since $\P_{\utheta^\star}(\AAA)\gtrsim_k n^{-2}$ uniformly over $(B, s, \{p_{\fff}\}_{\fff\in \FFF\setminus \FFF_{\tr}})\in \Psi_{\typ}$ by \eqref{eq:prob:AAA}, it suffices to show that the unconditional probability in \eqref{eq:lem:atypical} is $\ll n^{-2}$. Throughout, we abbreviate $\P\equiv \P_{\theta^\star}$ for simplicity.
    
    Note that if $v(\ttt)\geq \frac{\log n}{k\log 2}$, then $\ttt\in \FFF_{\tr}^{\atyp}$ by $p_{\fff}^\star\in \ee_{\frac{1}{2}}$. We first deal with the case $v(\ttt)\geq \frac{\log n}{k\log 2}$. For $v\geq \frac{\log n}{k\log 2}$, let $\eps_v= (v^2 \log n)^{-1}$. Then, 
    \begin{equation}\label{eq:atypical:tech:1}
        \P\Big(\sum_{\ttt:v_{\ttt}=v}v \Big|\frac{1}{n}\sum_{i=1}^{nh_{\circ}}\one(X_i=\ttt)-p_{\ttt}^\star\Big|\geq \frac{\eps_v}{\sqrt{n}}\Big)\leq L_v\cdot \sup_{v(\ttt)=v}\P\Big( \Big|\frac{1}{n}\sum_{i=1}^{nh_{\circ}}\one(X_i=\ttt)-p_{\ttt}^\star\Big|\geq \frac{\eps_v}{v L_v \sqrt{n}}\Big)\,,
    \end{equation}
    where $L_v:=|\{\ttt\in \FFF_{\tr}: v(\ttt)=v\}|$. An important observation is that $L_v\leq (Ck)^v$ for some universal constant $C>0$\footnote{There are at most $4^{v+f}$ number of isomorphism classes of trees with $v+f$ nodes (see Section 9.5 in \cite{FlajoletSedgewick}). The factor $k^v$ comes from assigning the number of spin-labels $\{0,1\}$ to the clauses in $\fff$ which have boundary-half edges.}. Thus, it follows that for $v(\ttt)=v \geq \frac{\log n}{k\log 2}$, 
    \begin{equation*}
        |h_\circ \E\one(X_i=\ttt)-p_{\ttt}^\star|\lesssim_k \frac{p_{\ttt}^\star}{\sqrt{n}}\leq \frac{2^{-\frac{kv}{2}}}{\sqrt{n}}\ll \frac{\eps_v}{v L_v \sqrt{n}}\,,
    \end{equation*}
    where the second inequality is due to $\{p_{\fff}^\star\}\in \ee_{\frac{1}{2}}$. Thus, we can use Chernoff's bound in the \textsc{rhs} of \eqref{eq:atypical:tech:1} to have
    \begin{equation*}
         \P\Big(\sum_{\ttt:v_{\ttt}=v}v \Big|\frac{1}{n}\sum_{i=1}^{nh_{\circ}}\one(X_i=\ttt)-p_{\ttt}^\star\Big|\geq \frac{\eps_v}{\sqrt{n}}\Big)\leq (Ck)^v\exp\Big(-\Omega\Big(\frac{\eps_v^2}{v^2 (Ck)^{2v}p_{\ttt}^\star}\Big)\Big)\,.
    \end{equation*}
    Again, using the fact $p_{\ttt}^\star \leq 2^{-\frac{kv}{2}}$, the \textsc{rhs} above is at most $\exp\big(-\Omega\big( 2^{\frac{kv}{3}}/(\log n)^2\big)\big)$. Therefore,
    \begin{equation}\label{eq:atypical:first:case}
    \sum_{v\geq \frac{\log n}{k \log 2}} \P\Big(\sum_{\ttt:v_{\ttt}=v}v \Big|\frac{1}{n}\sum_{i=1}^{nh_{\circ}}\one(X_i=\ttt)-p_{\ttt}^\star\Big|\geq \frac{\eps_v}{\sqrt{n}}\Big)=\exp\Big(-\Omega\Big(\frac{n^{2/3}}{\log^2 n}\Big)\Big)\ll n^{-2}\,.
    \end{equation}
    Since $\sum_{v}\eps_v\asymp (\log n)^{-1}\ll (\log n)^{-1/2}$, \eqref{eq:atypical:first:case} takes care of the case $v(\ttt) \geq \frac{\log n}{k\log 2}$.

    Let us now consider the case where $v(\ttt)\leq \frac{\log n}{k\log 2}$ and $p_{\ttt}^\star\leq n^{-1/2}$. Note that since $L_v \leq (Ck)^v$, the number of such trees is at most $n^{O(\frac{\log k}{k})}$. Thus, we have for some universal constant $C^\prime>0$ \rev{that}
    \begin{equation*}
    \begin{split}
        &\P\bigg(\sum_{v(\ttt)\leq \frac{\log n}{k\log 2},\,p_{\ttt}^\star\leq n^{-1/2}}v_{\ttt}\Big|\frac{1}{n}\sum_{i=1}^{nh_{\circ}}\one(X_i=\ttt)-p_{\ttt}^\star \Big|\geq \frac{n^{-1/2}(\log n)^{-1/2}}{2}\bigg)\\
        &\leq n^{\frac{C^\prime \log k}{k}} \sup_{p_{\ttt}^\star\leq n^{-1/2}}\P\bigg(\Big|\frac{1}{n}\sum_{i=1}^{nh_{\circ}}\one(X_i=\ttt)-p_{\ttt}^\star \Big|\geq n^{-\frac{1}{2}-\frac{C^\prime \log k}{k}}\bigg)\,,
    \end{split}
    \end{equation*}
    Note that if $p_{\ttt}^\star \leq n^{-1/2}$, then $|h_{\circ}\P(X_i=\ttt)-p_{\ttt}^\star|\lesssim_k n^{-1/2}p_{\ttt}^\star\leq n^{-1} \ll n^{-1/2-C^\prime \log k/k}$ for large enough $k$. Thus, using Chernoff's bound in the \textsc{rhs} above, it follows that
    \begin{equation}\label{eq:atypical:second:case}
        \P\bigg(\sum_{v(\ttt)\leq \frac{\log n}{k\log 2},\,p_{\ttt}^\star\leq n^{-1/2}}v_{\ttt}\Big|\frac{1}{n}\sum_{i=1}^{nh_{\circ}}\one(X_i=\ttt)-p_{\ttt}^\star \Big|\geq \frac{n^{-1/2}(\log n)^{-1/2}}{2}\bigg)\leq n^{\frac{C^\prime \log k}{k}}\exp\Big(-\Omega\big(n^{\frac{1}{2}-\frac{2C^\prime \log k}{k}}\big)\Big)\,,
    \end{equation}
    \rev{which is negligible compared to $n^{-2}$.} The estimates \eqref{eq:atypical:first:case} and \eqref{eq:atypical:second:case} conclude the proof.
    \end{proof}
    Next, we consider the more delicate case of typical trees.
    \begin{lemma}\label{lem:typical}
    For $C>0$, there exist constants $\wt{C}\equiv \wt{C}(C,\alpha,k)$ and $C_k$, such that the following holds. For any $\ttt\in \FFF_{\tr}^{\typ}$, and $\eps_{\ttt}>0$, we have
    \begin{equation}\label{eq:lem:typical}
         \P_{\utheta^\star}\Big(v_{\ttt}\Big|\frac{1}{n}\sum_{i=1}^{nh_{\circ}}\one(X_i=\ttt)-p_{\ttt}^\star \Big|\geq \frac{\eps_{\ttt}}{\sqrt{n}}\bgiven \AAA\Big)\leq  \wt{C}\cdot \P_{\utheta^\star}\Big(v_{\ttt}\Big|\frac{1}{n}\sum_{i=1}^{nh_{\circ}}\one(X_i=\ttt)-p_{\ttt}^\star \Big|\geq \frac{\eps_{\ttt}}{\sqrt{n}}\Big)+\exp(-C_k \sqrt{n})\,.
    \end{equation}
    \end{lemma}
    \begin{proof}
    We first show that the fraction $(nh_{\circ})^{-1}\sum_{i=1}^{nh_{\circ}}\one(X_i=\ttt)$ is bounded away from $1$ w.h.p.. To this end, for $\ttt\in \FFF_{\tr}^{\typ}$, denote $\bI_{\ttt}\equiv \{1\leq i\leq nh_{\circ}:X_i=\ttt\}$. Let the constant $\eps_0\equiv \eps_0(\alpha,k)>0$ be chosen so that $(1+\eps_0)\sup_{\ttt\in \FFF_{\tr}}p_\ttt^\star<(1-\eps_0)h_{\circ}^\star$ \rev{holds}. Such $\eps_0>0$ exists since $h_{\circ}^\star =\sum_{\ttt\in \FFF_{\tr}} p_{\ttt}^\star>0$ (cf. Remark \ref{rmk:optimal:compatible}). Then, by a Chernoff bound,
        \begin{equation}\label{eq:I:not:large}
\P_{\utheta^\star}\Big(\sum_{i=1}^{nh_{\circ}}\one(X_i=\ttt)\geq n(1-\eps_0)h_{\circ}^\star \Big)\leq \exp\Big(-\Omega_k\big(np_{\ttt}^\star\eps_0^2\big)\Big)\leq \exp(-C_k\sqrt{n})\,.
        \end{equation}
    Thus, $|\bI_{\ttt}|\leq n(1-\eps_0)h_\circ^\star$ holds with probability $1-\exp(-C_k\sqrt{n})$. Note that we can express
    \begin{equation*}
        \P_{\utheta^\star}\bigg(\AAA\cap \Big\{\,\Big|\frac{1}{n}\sum_{i=1}^{nh_{\circ}}\one(X_i=\ttt)-p_{\ttt}^\star \Big|\geq \frac{\eps_{\ttt}}{v_{\ttt}\sqrt{n}}\,\Big\}\bigg)
        =\sum_{I:\big|\frac{|I|}{n}-p_{\ttt}^\star\big|\geq \frac{\eps_{\ttt}}{v_{\ttt}\sqrt{n}}}\P_{\utheta^\star}\Big(\AAA\cap \Big\{\bI_{\ttt}=I\Big\}\Big)\,.
    \end{equation*}
    Having \eqref{eq:I:not:large} in mind, we bound the \textsc{rhs} above by
    \begin{equation*}
        \sup_{|I|\leq n(1-\eps_0)h_{\circ}^\star}\P_{\utheta^\star}\Big(\AAA\bgiven \bI_{\ttt}=I\Big)\cdot \P_{\utheta^\star}\Big(\Big|\frac{1}{n}\sum_{i=1}^{nh_{\circ}}\one(X_i=\ttt)-p_{\ttt}^\star \Big|\geq \frac{\eps_{\ttt}}{v_{\ttt}\sqrt{n}}\Big)+\P_{\utheta^\star}\Big(\sum_{i=1}^{nh_{\circ}}\one(X_i=\ttt)\geq n(1-\eps_0)h_{\circ}^\star \Big)\,.
    \end{equation*}
    Thus, because $\P_{\utheta^\star}(\AAA)\gtrsim_k n^{-2}\gg \exp(-C_k\sqrt{n})$, it follows that
    \begin{equation}\label{eq:typ:tech}
         \P_{\utheta^\star}\Big(\Big|\frac{1}{n}\sum_{i=1}^{nh_{\circ}}\one(X_i=\ttt)-p_{\ttt}^\star \Big|\geq \frac{\eps_{\ttt}}{v_{\ttt}\sqrt{n}}\bgiven \AAA\Big)\leq  \mathcal{R} \cdot \P_{\utheta^\star}\Big(\Big|\frac{1}{n}\sum_{i=1}^{nh_{\circ}}\one(X_i=\ttt)-p_{\ttt}^\star \Big|\geq \frac{\eps_{\ttt}}{v_{\ttt}\sqrt{n}}\Big)+\exp(-C_k^\prime \sqrt{n})\,,
    \end{equation}
    where $\mathcal{R}:= \sup_{|I|\leq n(1-\eps_0)h_{\circ}^\star}\frac{\P_{\utheta^\star}(\AAA\given \bI_{\ttt}=I)}{\P_{\utheta^\star}(\AAA)}$. We now argue that $\mathcal{R}\lesssim_k 1$ by a local central limit theorem: note that the distribution of $X_1,..., X_{nh_{\circ}}$ given $\bI_{\ttt}=I$ are i.i.d from the distribution
    \rev{
    \begin{equation*}
        \P^{-\ttt}(X_i=\ttt^\prime):= \frac{p_{\ttt^\prime}^\star}{h_{\circ}^\star-p_{\ttt}^\star}\one(\ttt^\prime \neq \ttt)\,,\quad\ttt^\prime \in \FFF_{\tr}\,.
    \end{equation*}
    }
    Thus, it follows that
    \begin{equation*}
    \begin{split}
        &\P_{\utheta^\star}\Big(\AAA\bgiven \bI_{\ttt}=I\Big)=\P^{-\ttt}\big(\AAA_{|I|}\big),\quad\textnormal{where}\\
        &\AAA_{\ell}:=\Big\{\sum_{i>\ell}\ufb_{X_i}= n\uh(B)-\sum_{\fff\in \FFF\setminus \FFF_{\tr}}np_{\fff}\ufb_{\fff}-\ell\cdot \ufb_{\ttt},~~~\sum_{i>\ell}\log w_{X_i}+\sum_{\fff\in \FFF\setminus \FFF_{\tr}}np_{\fff}\log w_{\fff}+\ell\cdot \log w_{\ttt}\in [ns,ns+1)\Big\}.
    \end{split}
    \end{equation*}
    By local central limit theorem for triangular arrays, we have for $\AAA_{\ell}\equiv \AAA_{\ell}(B,s, \{p_{\fff}\}_{\fff\in \FFF\setminus \FFF_{\tr}})$ that
    \begin{equation*}
    \limsup_{n\to\infty} \sup_{\substack{(B,s, \{p_{\fff}\}_{\fff\in \FFF\setminus \FFF_{\tr}})\in \Psi_{\typ}(C)\\ \ttt\in \FFF_{\tr}^{\typ}, \; \ell\leq n(1-\eps_0)h_{\circ}^\star}}\; n^2\cdot \P^{-\ttt}\big(\AAA_{\ell}\big)<\infty\,.
    \end{equation*}
    Recalling that $\P_{\utheta^\star}(\AAA)\gtrsim_k n^{-2}$ holds (cf. \eqref{eq:prob:AAA}), we therefore have that $\mathcal{R}\leq \wt{C}$ for some $\wt{C}\equiv \wt{C}(C,\alpha,k)$. Plugging this estimate into \eqref{eq:typ:tech} concludes the proof.
    \end{proof}
    Having Lemmas \ref{lem:atypical} and \ref{lem:typical} in hand, we now prove Proposition \ref{prop:freetree:concentration}.
    \begin{proof}[Proof of Proposition \ref{prop:freetree:concentration}]
    Fix $C>0$ and $\eps>0$. By Lemma \ref{lem:atypical}, it suffices to restrict our attention to the typical trees. That is, recalling the identity \eqref{eq:crucial:identity}, it suffices to show that there exists $C_0\equiv C_0(\eps,C,\alpha,k)$ such that for any $(B, s, \{p_{\fff}\}_{\fff\in \FFF\setminus \FFF_{\tr}})\in \Psi_{\typ}(C)$, 
    \begin{equation}\label{eq:typical:goal}
          \P_{\utheta^\star}\Big(\sum_{\ttt\in \FFF_{\tr}^{\typ}}v_{\ttt}\Big|\frac{1}{n}\sum_{i=1}^{nh_{\circ}}\one(X_i=\ttt)-p_{\ttt}^\star \Big|\geq \frac{C_0}{\sqrt{n}}\bgiven \AAA\Big)\leq \eps\,.
    \end{equation}
    To prove \eqref{eq:typical:goal}, we assign a cost $\eps_{\ttt}$ to each typical tree $\ttt\in \FFF_{\tr}^{\typ}$. We choose $\eps_{\ttt}\equiv \eps_{v_\ttt}=C_0^\prime (v^2 (Ck)^v)^{-1}$, where $C>0$ is the universal constant from the bound $L_v:=|\{\ttt\in \FFF_{\tr}: v_{\ttt}=v\}|\leq (Ck)^v$, and $C_0^\prime = C_0^\prime(C,\eps,\alpha,k)$ is determined later. Then, we let $C_0=C_0(C,\eps,\alpha,k)$ to be large enough so that \begin{equation*}
        \sum_{\ttt\in \FFF_{\tr}^{\atyp}}\eps_{\ttt}\leq \sum_{v\geq 1}\eps_v L_v = C_0^\prime \sum_{v\geq 1}v^{-2}\leq C_0\,.
    \end{equation*}
    Thus, Lemma \ref{lem:typical} shows that the \textsc{lhs} of \eqref{eq:typical:goal} is bounded by
    \begin{equation*}
    \wt{C}\sum_{\ttt\in \FFF_{\tr}^{\typ}} \P_{\utheta^\star}\Big(\Big|\frac{1}{n}\sum_{i=1}^{nh_{\circ}}\one(X_i=\ttt)-p_{\ttt}^\star \Big|\geq \frac{\eps_{\ttt}}{v_{\ttt}\sqrt{n}}\Big)+|\FFF_{\tr}^{\typ}|\exp(-C_k\sqrt{n})\,,
    \end{equation*}
    where $\wt{C}\equiv \wt{C}(C,\alpha,k)$ is from Lemma \ref{lem:typical}. An important observation is that $\FFF_{\tr}^{\typ}\subset \{\ttt\in \FFF_{\tr}: v_{\ttt}\leq \frac{\log n}{k \log 2}\}$ since $\{p_{\fff}^\star\}_{\fff\in \FFF} \in \ee_{\frac{1}{2}}$ holds by Lemma 3.13 in \cite{NSS}. Thus, $|\FFF_{\tr}^{\typ}|\leq (Ck)^{\frac{\log n}{k \log 2}}= n^{O(1)}$. Therefore, altogether,
    \begin{equation}\label{eq:typical:goal:inter}
      \P_{\utheta^\star}\Big(\sum_{\ttt\in \FFF_{\tr}^{\typ}}v_{\ttt}\Big|\frac{1}{n}\sum_{i=1}^{nh_{\circ}}\one(X_i=\ttt)-p_{\ttt}^\star \Big|\geq \frac{C_0}{\sqrt{n}}\bgiven \AAA\Big)\leq \wt{C}\sum_{\ttt\in \FFF_{\tr}^{\typ}} \P_{\utheta^\star}\Big(\Big|\frac{1}{n}\sum_{i=1}^{nh_{\circ}}\one(X_i=\ttt)-p_{\ttt}^\star \Big|\geq \frac{\eps_{\ttt}}{v_{\ttt}\sqrt{n}}\Big)+o_n(1)\,.
    \end{equation}
    The final step is to bound the \textsc{rhs} above by a Chernoff bound. Choose $C_0^\prime=C_0^\prime(C,\eps,\alpha,k)$ large enough so that for every $v\geq 1$,
    \begin{equation*}
        \sqrt{n}\big|h_{\circ}\P_{\utheta^\star}(X_i=\ttt)-p_{\ttt}^\star\big|\leq \wt{C}^\prime p_{\ttt^\star}\leq \wt{C}^\prime 2^{-\frac{kv}{2}}\leq \frac{C_0^\prime}{2 v^2(Ck)^v}\,,
    \end{equation*}
    where the first inequality holds for some $\wt{C}^\prime \equiv \wt{C}^\prime(C,\alpha,k)$ since $B\to h_{\circ}(B)$ is Lipschitz and $\norm{B-B^\star}_1\leq \frac{C}{\sqrt{n}}$ by definition of $\Psi_{\typ}(C)$ in \eqref{eq:def:Psi}. Thus, we have by a Chernoff bound that
    \begin{equation}\label{eq:typical:goal:inter:2}
    \begin{split}
    \sum_{\ttt\in \FFF_{\tr}^{\typ}} \P_{\utheta^\star}\Big(\Big|\frac{1}{n}\sum_{i=1}^{nh_{\circ}}\one(X_i=\ttt)-p_{\ttt}^\star \Big|\geq \frac{\eps_{\ttt}}{v_{\ttt}\sqrt{n}}\Big)
    &\leq \sum_{\ttt\in \FFF_{\tr}^{\typ}} \exp\Big(-\frac{\eps_{\ttt}^2}{C (v_{\ttt} p_{\ttt}^\star)^2}\Big)\\
    &\leq \sum_{v\geq 1}(Ck)^v\exp\Big(-\frac{(C_0^\prime)^2 2^{\frac{kv}{2}}}{C v^6 (Ck)^{2v}}\Big)\,,
    \end{split}
    \end{equation}
    where $C>0$ denotes a universal constant. If we denote the \textsc{rhs} above by $f(C_0^\prime)$, then $f$ does not depend on any other parameters and clearly \rev{satisfies} $\lim_{C_0^\prime\to\infty}f(C_0^\prime)=0$. Therefore, taking $C_0^\prime$ to be large enough so that $f(C_0^\prime)\leq (2\wt{C})^{-1}\eps$ concludes the proof by \eqref{eq:typical:goal:inter} and \eqref{eq:typical:goal:inter:2}. 
    \end{proof}
    By a corollary of Proposition \ref{prop:boundary} and Proposition \ref{prop:freetree:concentration}, we have the following.
    \begin{cor}\label{cor:concentration}
    Recall the set of coloring profile $\Xi_0$ from Proposition \ref{prop:nss}. Then, for any $\eps>0$, there exists $C_0\equiv C_0(\eps,\alpha,k)$ such that uniformly over $|s-s^\star|\leq n^{-2/3}$,
    \begin{equation*}
        \E \bZ_{\la^\star,s}\Big[\norm{\xi-\xi^\star}_{\tsq} \geq \frac{C_0}{\sqrt{n}}\,~\textnormal{and}~~\xi \in \Xi_0\Big]\leq \eps \cdot\E \bZ_{\la^\star,s}\,.
    \end{equation*}
    \end{cor}
    \begin{proof}
        For $\xi \in \Xi_0$, the quantity $\norm{\xi-\xi^\star}_{\tsq}$ defined in \eqref{eq:def:tsq:norm} can be bounded by
        \begin{equation*}
            \norm{\xi-\xi^\star}_{\tsq}\leq \norm{B-B^\star}_1+2\sum_{\ttt\in \FFF_{\tr}}|p_{\ttt}-p_{\ttt}^\star|v_{\ttt}+\frac{\log^2 n}{n}\,,
        \end{equation*}
        since by Remark \ref{rmk:clauses}, $f_{\fff}\leq v_{\fff}$ holds for non multi-cyclic free component\rev{s} $\fff$, and $v(\fff)\leq \frac{4\log n}{k\log 2}$ holds if $p_{\fff}\neq 0$ from $\xi\in \ee_{\frac{1}{4}}$. Thus, by taking $C_0\equiv C_0(\eps,\alpha,k)$ large enough, Proposition \ref{prop:boundary} shows that for some $C\equiv C(\eps,\alpha,k)>0$,
        \begin{equation}\label{eq:cor:concentration}
        \begin{split}
             &\E \bZ_{\la^\star,s}\Big[\norm{\xi-\xi^\star}_{\tsq} \geq \frac{C_0}{\sqrt{n}}\,~\textnormal{and}~~\xi \in \Xi_0\Big]\\
             &\leq \frac{\eps}{2} \cdot \E\bZ_{\la^\star,s}+\E \bZ_{\la^\star,s}\bigg[\norm{B-B^\star}_1\leq \frac{C}{\sqrt{n}}\,,~\sum_{\ttt\in \FFF_{\tr}}|p_{\ttt}-p_{\ttt}^\star|v_{\ttt}\geq \frac{C_0}{2\sqrt{n}}\,,~\textnormal{and}~~(B,\{p_{\fff}\}_{\fff\in \FFF})\in \Xi_0\bigg]\,.
        \end{split}
        \end{equation}
        Here, we used the fact that $p_{\fff}^\star=0$ for cyclic free components $\fff$ (cf. Definition \ref{def:optimal:coloring}). Note that $\norm{B-B^\star}_1\leq \frac{C}{\sqrt{n}}$, $|s-s^\star|\leq n^{-2/3}$,  and $(B,\{p_{\fff}\}_{\fff\in \FFF})\in \Xi_0$ implies that $(B,s, \{p_{\fff}\}_{\fff\in \FFF\setminus \FFF_{\tr}})\in \Psi_{\typ}(C)$. Therefore, Proposition \ref{prop:freetree:concentration} shows that if take $C_0\equiv C_0(\eps,\alpha,k)$ large enough, the rightmost term in \eqref{eq:cor:concentration} is at most $\frac{\eps}{2}\cdot \E \bZ_{\la^\star,s}$.
    \end{proof}
    Finally, it is straightforward to establish Theorem \ref{thm:concentration} from Proposition \ref{prop:nss} and Corollary \ref{cor:concentration}.
    \begin{proof}[Proof of Theorem \ref{thm:concentration}]
        By Proposition \ref{prop:nss}, there exist\rev{s} $\uC\equiv \uC(\eps,\alpha,k)$ such that $\P\big((\bGG,\bux)\notin \Gamma(\uC)\big)\leq \eps/2$. For such $\uC=(C_1,C_2)$, we have that
        \begin{equation}\label{eq:inter}
        \begin{split}
            &\P\bigg((\bGG,\bux)\in \Gamma(\uC)~~\textnormal{and}~~\norm{\xi[\bGG,\bux]-\xi^\star}_{\tsq}\geq \frac{C_0}{\sqrt{n}}\bigg)=\E\bigg[\E\Big[\sum_{\substack{\ux:(\bGG,\ux)\in \Gamma(\uC)\\\norm{\xi[\bGG,\ux]-\xi^\star}_{\tsq}\geq \frac{C_0}{\sqrt{n}}}}\; \frac{\size(\ux;\bGG)}{|\sol(\bGG)|}\bgiven \bGG\Big]\bigg]\\
            &\leq \sqrt{n}e^{n\la^\star s^\star+(2-\la^\star)C_2}\sum_{s\in [s_{\circ}(C_1),s_{\circ}(C_2)]}\E \bZ_{\la^\star,s}\Big[\norm{\xi-\xi^\star}_{\tsq} \geq \frac{C_0}{\sqrt{n}}\,~\textnormal{and}~~\xi \in \Xi_0\Big]+o_n(1)\,,
        \end{split}
        \end{equation}
        where the extra $o_n(1)$ comes from the truncation of the number of free variables and forcing edges in the definition of $\bZ_{\la^\star,s}$ in \eqref{eq:def:Z:lambda} (see Lemma 2.17 in \cite{ssz22}). By Corollary \ref{cor:concentration}, the sum in the \textsc{rhs} above can be made small enough compared to $\sum_{s\in [s_{\circ}(C_1),s_{\circ}(C_2)]}\E\bZ_{\la^\star,s}$. Moreover, Theorem 3.22 in \cite{NSS} shows that uniformly over $|s-s^\star|\leq n^{-2/3}$, $\E\bZ_{\la^\star,s}\asymp_k n^{-1/2}e^{n\la^\star s^\star}$ holds (see also equation (3.54) therein). Note that $s$ lies in the lattice $n^{-1}\Z$, so the number of $s\in [s_{\circ}(C_1),s_{\circ}(C_2)]$ is at most $\lceil C_2-C_1+1 \rceil$. Therefore, taking $C_0\equiv C_0(\eps,\alpha,k)>0$ large enough compared to $|C_1|\vee|C_2|$ in \eqref{eq:inter}, we have by Corollary \ref{cor:concentration} that $\P\big((\bGG,\bux)\in \Gamma(\uC)~~\textnormal{and}~~\norm{\xi[\bGG,\bux]-\xi^\star}_{\tsq}\geq \frac{C_0}{\sqrt{n}}\big)\leq \frac{\eps}{2}$, which concludes the proof.
    \end{proof}

         \section{Coupling $t$-neighborhoods to a broadcast process on a tree}
    \label{sec:coupling}
    In this section, we prove Lemma \ref{lem:bdry:cycle} and Theorem \ref{thm:coupling}. We mostly focus on the proof of Theorem \ref{thm:coupling}, and the proof of Lemma \ref{lem:bdry:cycle}, which is based on a first moment estimate, is provided at the end of this section. Throughout, we let $(\bGG, \bux)$ be a random frozen configuration drawn with probability proportional to its size (cf. $(a)$ of Observation \ref{obs:sampling}).

    As mentioned in Section \ref{subsec:proof:coupling}, the key idea in proving Theorem \ref{thm:coupling} is to construct a coupling between the measures $\E_{\xi} \nu_t[\bGG,\bux]$ and $\nu^\star_t$, and another coupling between $\E_{\xi}\big[\nu_t[\bGG,\bux]\otimes \nu_t[\bGG,\bux]\big]$ and $\E_{\xi}\nu_t[\bGG,\bux] \otimes \E_{\xi}\nu_t[\bGG,\bux]$.
    \begin{prop}\label{prop:coupling}
    There exists an explicit $\nu^\star_t\equiv \nu^\star_t[\alpha,k]\in \PPP(\Omega_t)$ such that the following holds: for any $t\geq 1$ and $C>0$, there exists a constant $K\equiv K(C,k,t)$ such that 
        \begin{align}
            &\sup_{\xi \in \Xi_C}d_{\tv}(\,\E_{\xi}\nu_t[\bGG,\bux]\,,\, \nu_t^\star\,)\leq \frac{K}{\sqrt{n}}\,,
            \label{eq:single:prop:coupling}\\
             &\sup_{\xi \in \Xi_C}d_{\tv}\Big (\,\E_{\xi}\big[\nu_t[\bGG,\bux]\otimes \nu_t[\bGG,\bux]\big]\,,\, \E_{\xi}\nu_t[\bGG,\bux]\otimes \E_{\xi}\nu_t[\bGG,\bux]\,\Big)\leq \frac{K}{n}\,.
             \label{eq:double:prop:coupling}
        \end{align}
    \end{prop}
    Proposition \ref{prop:coupling} clearly implies Theorem \ref{thm:coupling}, and we include the proof for completeness.
    \begin{proof}[Proof of Theorem \ref{thm:coupling}]
        By Chebyshev's inequality, we have for any $w \in [-1,1]^{\Omega_t\cup\{\cyc\} }$ that
        \begin{equation*}
        \begin{split}
             \P_{\xi}\bigg(\Big|\big\langle \nu_t[\bGG, \bux] -\nu_t^\star \,,\, w\big\rangle\Big|\geq \frac{K_0}{\sqrt{n}}\bigg)
             &\leq \frac{n}{(K_0)^2}\left(\Big(\big\langle \E_{\xi} \nu_t[\bGG, \bux] -\nu_t^\star \,,\, w\big\rangle\Big)^2+\Var_{\xi} \Big(\big\langle \nu_t[\bGG,\bux]\,,\, w\big\rangle\Big)\right)\\
             &\leq \frac{n}{(K_0)^2}\left(\Big\|\E_{\xi} \bnu_t-\nu_t^\star\Big\|_1^2+\Big\|\E_{\xi}[\bnu_t\otimes \bnu_t]-\E_{\xi}\bnu_t\otimes \E_{\xi}\bnu_t\Big\|_1 \right)\,,
        \end{split}
        \end{equation*}
        where we abbreviated $\bnu_t\equiv \nu_t[\bGG,\bux]$ in the last line. Recall that $d_{\tv}(\mu,\nu)=\frac{1}{2}\norm{\mu-\nu}_1$ holds for any probability measures $\mu, \nu$. Thus, taking $K_0\equiv K_0(C,\eps,k,t)$ large enough so that $(K_0)^2\geq 8\eps^{-1}(K\vee 1)^2$ for $K\equiv K(C,k,t)$ in Proposition \ref{prop:coupling} concludes the proof.
    \end{proof}
    Thus, the rest of this section is devoted to the proof of Proposition \ref{prop:coupling} except for a brief moment where we prove Lemma \ref{lem:bdry:cycle}.
    \subsection{A broadcast process with edge configurations}\label{subsec:broadcast}
    In this subsection, we define the \textit{optimal $t$-coloring profile} $\nu^\star_t\equiv \nu^\star_t[\alpha,k]$ based on the optimal coloring profile $\xi^\star\equiv \xi^\star[\alpha,k]$ in Definition \ref{def:optimal:coloring}. The following notations will be useful throughout.

     We say $\utau \in \CCC^d$ (resp. $(\utau, \uL^\prime) \in (\CCC\times \{0,1\})^k$) is \textit{valid} if it can be realized as $\sig_{\delta v}=\utau$ (resp. $(\sig_{\delta a}, \uL_{\delta a})=(\utau,\uL^\prime)$) for a valid component coloring $(V,F,E,\uL,\sig)$ and $v\in V$ (resp. $a\in F$). We say $\utau\in \CCC^k$ is valid if there exists $\uL^\prime \in \{0,1\}^k$ such that $(\utau,\uL^\prime)$ is valid. We denote the collection of valid $\utau \in \CCC^d$ such that $\utau \not\subset \{\rr,\bb,\fs\}^d$ by $\dot{\CCC}_{\ff}$, and denote the collection of valid $\utau \in \CCC^k$ such that $\utau \not\subset \{\rr,\bb,\fs\}^k$ by $\hat{\CCC}_{\ff}$

    Note that for $\utau\in \dot{\CCC}_{\ff}$ and $\sig_{\delta v}=\utau$, then the free component $\fff(\utau)\in \FFF$ that contains $v$ is determined \textit{solely} with $\utau$. Similarly, for $\utau \in \hat{\CCC}_{\ff}$, the free component $\fff(\utau)$ induces is well-defined. Conversely, given a free component $\fff\in \FFF, v\in V(\fff)$, and $a\in F(\fff)$, the component colorings $\sig_{\delta v}\in \CCC^{d}$ and $\sig_{\delta a} \in \CCC^k$ are determined up to a permutation of the coordinates. For $\utau\in \dot{\CCC}_{\ff}$, define the \textbf{multiplicity} of $\utau$ by  
    \begin{equation}\label{eq:dot:multi}
        \dot{m}(\utau):=\big|\{v\in V(\fff(\utau)): \sig_{\delta v}\in  \textnormal{per}(\utau)\}\big|\,,
    \end{equation}
    where $\textnormal{per}(\utau)$ denotes the set of permutations $\utau$. Similarly, the multiplicity of $\utau \in \hat{\CCC}_{\ff}$ is defined by
    \begin{equation}\label{eq:hat:multi}
        \hat{m}(\utau):=\big|\{a\in F(\fff(\utau)): \sig_{\delta a}\in  \textnormal{per}(\utau)\}\big|\,.
    \end{equation}
    \begin{remark}\label{rmk:coloring:determines:spins}
    Let $\sig\in \CCC^{E}$ be a valid component coloring corresponding to $(\GG,\ux)$. Note that for $\utau\in \dot{\CCC}_{\ff}$, the quantity $\dot{m}(\utau)\cdot p_{\fff(\utau)}[\GG,\ux]$ equals the fraction of variables such that $\sig_{\delta v}\in \textnormal{per}(\utau)$. A similar remark can be made for $\utau\in \hat{\CCC}_{\ff}$. On the other hand, the empirical profile of spins adjacent to frozen variables and separating clauses is encoded by $B[\GG,\ux]$. Thus, the coloring profile $\xi=\xi[\GG,\ux]$ alone determines the spin profiles $\{\sig_{\delta v}\}_{v\in V}$ and $\{\sig_{\delta a}\}_{a\in F}$ up to a permutation of the coordinates, which serves as a useful fact throughout this section. 
    \end{remark}
    We now introduce the \textit{broadcast model} on an infinite $(d,k)$-regular tree with edge configurations.
    \begin{defn}\label{def:broadcast}
    Consider symmetric probability distributions $\dot{\HH}\in \PPP(\CCC^{d})$, $\hat{\HH}\in \PPP(\CCC^{k})$ with the same marginal $\bar{\HH}\in \PPP(\CCC)$. That is, $\dot{\HH}(\sigma_1,\ldots,\sigma_d)=\dot{\HH}(\sigma_{\pi(1)},\ldots, \sigma_{\pi(d)})$ holds for $\pi \in S_d$ and $(\sigma_1,\ldots,\sigma_d)\in \CCC^d$. An analogous condition holds for $\hHH$. Moreover, for any $\tau \in \CCC$,
    \begin{equation}\label{eq:same:marginal}
          \bar{\HH}(\tau):=\sum_{\utau \in \CCC^{d}}\dot{\HH}(\utau)\one\{\tau_1=\tau\} = \sum_{\utau \in \CCC^k}\hat{\HH}(\utau)\one\{\tau_1 = \tau \}\,.
    \end{equation}
    Given such \textbf{channel} $(\dot{\HH}, \hat{\HH})$, and an infinite $(d,k)$-regular tree $\TTT_{d,k}$ with root $\rho$, the \textbf{broadcast process} (with edge configurations) is a probability distribution of $(\bsig_t,\buL_t)=\big((\bsigma_e)_{e\in E(\TTT_{d,k,t})},(\bL_e)_{e\in E_{\sf in }(\TTT_{d,k,t})}\big)$ defined as follows. The spin neighborhood around the root $\bsig_{\delta \rho}\in \CCC^{d}$ is drawn from the distribution $\dot{\HH}$. Then, it is propagated along the variables and clauses with the following rule: if an edge $e\in E(\TTT_{d,k})$ has children edges $\delta a(e)\setminus e$, i.e. when $d(a(e),\rho)>d(v(e),\rho)$, then for $\utau=(\tau_1,\ldots,\tau_k)\in \CCC^k$ and $\tau\in \CCC$, 
        \begin{equation}\label{eq:broadcast:clause}
            \P\big(\bsig_{\delta a(e)}= \utau \,\big| \bsigma_e = \tau\big)= \frac{1}{k}\frac{\hat{\HH}(\utau)\sum_{i=1}^{k}\one(\tau_i=\tau)}{\bar{\HH}(\tau)}=:\hat{\HH}(\utau \given \tau)\,.
        \end{equation}
    If an edge $e\in E(\TTT_{d,k})$ has children edges $\delta v(e)\setminus e$, i.e. when $d(v(e),\rho)>d(a(e),\rho)$, then for $\utau=(\tau_1,\ldots,\tau_d)\in \CCC^d$ and $\tau \in \CCC$,
    \begin{equation*}
        \P\big(\bsig_{\delta v(e)}= \utau \,\big| \bsigma_e = \tau\big)= \frac{1}{d}\frac{\dot{\HH}(\utau)\sum_{i=1}^{d}\one(\tau_i=\tau)}{\bar{\HH}(\tau)}=:\dot{\HH}(\utau\given \tau)\,.
    \end{equation*}
    Finally, conditional on $(\bsigma_e)_{e\in E(\TTT_{d,k})}$, draw $\buL_{\delta a}\in \{0,1\}^k$ for each clause $a\in F(\TTT_{d,k})$ independently and uniformly at random among $\uL\in \{0,1\}^k$ such that $(\bsig_{\delta a},\uL)$ is valid.  
    \end{defn}
    In the last step above, we \rev{emphasize} that \rev{if} $\sig_{\delta a} \in \hat{\CCC}_{\ff}$, then there exists a unique $\uL_{\delta a}\in \{0,1\}^k$ such that $(\sig_{\delta a}, \uL_{\delta a})$ is valid since the non-boundary component coloring carries the literal information.
    \begin{defn}\label{def:optimal:t}
    Let $\dot{\HH}^\star\equiv \dot{\HH}^\star[\alpha,k]\in \PPP(\CCC^d)$ and $\hat{\HH}^\star\equiv \hHH^\star[\alpha,k]$ be defined as follows.
    \begin{eqnarray*}
        \dHH^\star(\utau):=
        \begin{cases}
            \dot{B}^\star(\utau) & \utau\in \{\rr,\bb,\fs\}^{d},\\
            (|\textnormal{per}(\utau)|)^{-1}\cdot \dot{m}(\utau)\cdot p^\star_{\fff(\utau)} &\utau\in \dot{\CCC}_{\ff},
        \end{cases}
        &
        \hHH^\star(\utau):=
        \begin{cases}
            \hat{B}^\star(\utau) & \utau\in \{\rr,\bb,\fs \}^{d},\\
            \frac{k}{d}\cdot  (|\textnormal{per}(\utau)|)^{-1}\cdot\hat{m}(\utau)\cdot p^\star_{\fff(\utau)} &\utau\in \hat{\CCC}_{\ff}\,.
        \end{cases}
    \end{eqnarray*}
    The compatibility $\{p_{\fff}^\star\}_{\fff \in \FFF}\sim B^\star$ (cf. Remark \ref{rmk:optimal:compatible}) guarantees that $\dHH^\star$ and $\hHH^\star$ have total mass $1$ and have the same marginals. Consider the sample $(\bsig^{\star}_t,\buL^\star_t)$ drawn from the broadcast process with channel $(\dHH^\star, \hHH^\star)$. Then, the \textbf{optimal $t$-neighborhood coloring profile} $\nu^\star_t\equiv \nu^\star_t[\alpha,k]$ is the distribution of $(\bsig^{\star}_t,\buL^\star_t)\in \Omega_t$, considered up to automorphisms as in Definition \ref{def:t:nbd:empirical}.
    \end{defn}
    
    \subsection{Coupling based on a configuration model}
    \label{subsec:coupling:config}
    We now prove Proposition \ref{prop:coupling}. We focus on the estimate $d_{\tv}(\E_{\xi} \nu_t[\bGG,\bux], \nu^\star_t)\lekt n^{-1/2}$ (cf. \eqref{eq:single:prop:coupling}) as the second estimate \eqref{eq:double:prop:coupling} is obtained in a simpler manner. Throughout, we fix $\xi\in \Xi_C$.

    Recall that the law $\P_{\xi}(\cdot ):=\P\big((\bGG,\bux)=\cdot \,\big|\,\xi[\bGG,\bux]=\xi\big)$ is a uniform distribution of $(\GG,\ux)$ that satisfy $\xi[\GG,\ux]=\xi$ (cf. Remark \ref{rmk:P:xi}). With abuse of notation, we write $\xi[\GG,\sig]=\xi$ for $\sig$ corresponding to a frozen configuration $(\GG,\ux)$ with $\xi[\GG,\ux]=\xi$, and consider $\P_{\xi}(\cdot)$ also as a distribution of component coloring $(\bGG,\bsig)$ conditioned on $\xi[\bGG,\bsig]=\xi$, which is uniform.

     Then, note that $\E_{\xi}\nu_t[\bGG,\bux]\in \PPP(\Omega_t\sqcup \{\cyc\})$ is the law of $\big(\bsig_t(\bv,\bGG),\buL_t(\bv,\bGG)\big)$ considered up to automorphisms where $(\bGG,\bsig)\equiv (V,F,\bE,\buL,\bsig)\sim \P_{\xi}$ and $\bv\sim \Unif(V)$. Here, we take the convention that if $N_t(\bv,\bG)$ contains a cycle, then $\big(\bsig_t(\bv,\bGG),\buL_t(\bv,\bGG)\big)\equiv\cyc$. Moreover, the following observation, which follows from Remark \ref{rmk:coloring:determines:spins}, shows that $(\bGG,\bsig)\in \P_{\xi}$ is a sample from a \textit{configuration model}. It serves as our main intuition in the construction of the coupling.
    \begin{obs}\label{obs:config}
    Consider a coloring profile $\xi=(B, \{p_{\fff}\}_{\fff\in \FFF})$, where $\{p_{\fff}\}_{\fff\in \FFF}\sim B$ \footnote{Here, we assume that $n\dot{B}, m\hat{H},$ and $\{np_{\fff}\}_{\fff\in \FFF}$ are integer valued.}. Then, $(\bGG,\bsig)\sim \P_{\xi}$ can be drawn from the following \textbf{configuration model}.
   \begin{enumerate}[label=(\alph*)]
       \item For $\utau\equiv (\tau_1,\ldots,\tau_d) \in \{\rr,\bb,\fs\}^{d}$, consider $n\dot{B}(\tau_1,\ldots,\tau_d)$ number of variables and assign to each of the variables the spin $\tau_i$ to its $i$'th half-edge, $1\leq i \leq d$. For $\utau \in \dot{\CCC}_{\ff}$, consider $n\cdot \dot{m}(\utau)\cdot p_{\fff(\utau)}$ number of variables and assign to each of the variables the spin $\tau_{\boldsymbol{\pi}(i)}$ to its $i$'th half-edge, $1\leq i\leq d$, where $\boldsymbol{\pi}\in S_d$ is a u.a.r. permutation. The total number of variables considered is $n$ by compatibility (cf. Definition \ref{def:compat}). Then, permute the location of the considered variables u.a.r., i.e. assign a random order to the $n$ variables.
       \item Similarly, for $\utau \in \{\rr,\bb,\fs\}^k$, consider $m\hat{B}(\utau)$ number of clauses with its neighboring spin $\utau$, and for $\utau \in \hat{\CCC}_{\ff}$, consider $n\cdot \hat{m}(\utau)\cdot p_{\fff(\utau)}$ with its neighboring spins a u.a.r. permutation of $\utau$. Then, permute the location of the considered clauses u.a.r..
       \item Subsequently, match the half-edges adjacent to variables and the half-edges adjacent to clauses among those that have the same spin in a uniformly random manner. 
       \item Finally, we draw the literals: for each clause $a\in F$, independently draw $\buL_{\delta a}\in \{0,1\}^k$ uniformly at random among those which $(\bsig_{\delta a}, \buL_{\delta a})$ is valid. 
   \end{enumerate}
    \end{obs}
    \begin{remark}\label{rmk:symmetry:B}
    Note that a priori, the Steps $(a)$ and $(b)$ in Observation \ref{obs:config} does not permute the spins adjacent to frozen variables and separating clauses, i.e. $\utau \in \{\rr,\bb,\fs\}^d$ and $\utau \in \{\rr,\bb,\fs\}^k$. This is because $\dot{B}(\utau)$ and $\hat{B}(\utau)$ stores the information on the orderings of $\utau$. However, we do not distinguish such orderings for $(\sig_t,\uL_t)\in \Omega_t$: for permutations $\pi_1\in S_k$ and $\pi_2\in S_d$, let $\xi_{\pi_1,\pi_2}\equiv (\dot{B}_{\pi_1},\hat{B}_{\pi_2}, \bar{B}, \{p_{\fff}\}_{\fff\in \FFF})$, where $\dot{B}_{\pi_1}(\utau)\equiv \dot{B}(\tau_{\pi_1(1)},\ldots ,\tau_{\pi_1(d)})$ and $\hat{B}_{\pi_2}$ is similarly defined. Then, we have that $\E_{\xi}\nu_t[\bGG,\bux]=\E_{\xi_{\pi_1,\pi_2}}\nu_t[\bGG,\bux]$ and $\E_{\xi}\nu_t[\bGG,\bux]\otimes \nu_t[\bGG,\bux]=\E_{\xi_{\pi_1,\pi_2}}\nu_t[\bGG,\bux]\otimes \nu_t[\bGG,\bux]$. Thus, for the purpose of proving Proposition \ref{prop:coupling}, we may assume that the spins $\utau\in \{\rr,\bb,\fs\}^d$ and $\utau \in \{\rr,\bb,\fs\}^k$ are permuted u.a.r. like the spins $\utau\in \dot{\CCC}_{\fff}$ and $\utau\in \hat{\CCC}_{\fff}$ in the Steps $(a)$ and $(b)$ in Observation \ref{obs:config}. 
    \end{remark}
     \begin{remark}\label{rmk:exploration}
    Note that Observation \ref{obs:config} and Remark \ref{rmk:symmetry:B} show that $\E_{\xi}\nu_t[\bGG,\bux]=\Law\big(\bsig_t(\bv,\bGG),\buL_t(\bv,\bGG)\big)$, where $(\bGG,\bsig)\sim \P_{\xi}$ and $\bv\sim \Unif(V)$, is an \textit{exploration process}, which can be described in a \textit{breadth-first} manner as follows. 
    \begin{enumerate}[label=(\alph*)]
        \item  Consider $n$ variables with $d$ half-edges hanging and $m$ clauses with $k$ half-edges hanging. Each half-edge is assigned a color $\sigma\in \CCC$ by the Steps $(a)$, $(b)$ in Observation \ref{obs:config} and Remark \ref{rmk:symmetry:B}. Then, pick a variable $\bv$ uniformly at random among $n$ variables. 
        \item Let $(e_1,...,e_d)=\delta \bv$ be the half-edges adjacent to $\bv$. For $e_1$ match it with a half-edge hanging on a clause u.a.r. among those which have the same color $\sigma_{e_1}$. Similarly, match $e_2$ with a half-edge u.a.r. among those which have the same color \textbf{and} have not been matched, i.e. $e_2$ cannot be matched with the half-edge that $e_1$ has been matched with. Repeat the same procedure for $e_3,...,e_d$, i.e. match them sequentially u.a.r. among those \rev{that} have the same color and have not been matched previously. 
        \item By the previous step, we obtain a $2-$neighborhood around $\bv$ (not necessarily a tree) with the boundary half-edges hanging on clauses. For each boundary half-edges, we repeat the same procedure. We repeat this process until depth $2t$ to obtain the $2t-\frac{3}{2}$ neighborhood of $\bv$ denoted by $N_t(\bv)$, and a coloring configuration on $N_t(\bv)$ denoted by $\bsig_t(\bv)\in \CCC^{E(\TTT_{d,k,t})}\sqcup\{\cyc\}$. Finally, for each clause $a\in N_t(\bv)$, independently draw $\buL_{\delta a}\in \{0,1\}^k$ uniformly at random among those which $(\bsig_{\delta a}, \buL_{\delta a})$ is valid. The final output is $(\bsig_t(\bv),\buL_t(\bv))\in \Omega_t\sqcup \{\cyc\}$.
    \end{enumerate}
    \end{remark}
    The description of the exploration process above is a \textit{sampling without replacement} version (due to the pairs of half-edges already matched) of the broadcast process with the channel $(\dHH^{\sy}, \hHH^{\sy})$ given as follows: for $\xi=(B, \{p_{\fff}\}_{\fff\in \FFF})$, the \textbf{symmetrized coloring profile} of $\xi$ is defined by $(\dHH^{\sy}, \hHH^{\sy})\equiv (\dHH^{\sy}[\xi], \hHH^{\sy}[\xi])\in \PPP(\CCC^d)\times \PPP(\CCC^k)$, where 
    \begin{equation}\label{eq:symm:coloring}
    \begin{split}
        &\dHH^{\sy}(\utau):=
        \begin{cases}
            (|\textnormal{per}(\utau)|)^{-1}\cdot\sum_{\utau^\prime \in \textnormal{per}(\utau)}\dot{B}(\utau^\prime) &~~\utau\in \{\rr,\bb,\fs\}^{d},\\
            (|\textnormal{per}(\utau)|)^{-1}\cdot \dot{m}(\utau)\cdot p_{\fff(\utau)} &~~\utau\in \dot{\CCC}_{\ff},
        \end{cases}
    \\
        &\hHH^{\sy}(\utau):=
        \begin{cases}
            (|\textnormal{per}(\utau)|)^{-1}\cdot\sum_{\utau^\prime \in \textnormal{per}(\utau)}\hat{B}(\utau^\prime) & \utau\in \{\rr,\bb,\fs \}^{d},\\
            \frac{k}{d}\cdot  (|\textnormal{per}(\utau)|)^{-1}\cdot\hat{m}(\utau)\cdot p_{\fff(\utau)} &\utau\in \hat{\CCC}_{\ff}\,.
        \end{cases}
    \end{split}
    \end{equation}
 The only difference between the exploration process $\E_{\xi}\nu_t[\bGG,\bux]$ and the broadcast model with channel $(\dHH^{\sy},\hHH^{\sy})$ is that during the process when the boundary half-edge $e$ has spin $\sigma_e\in \{\rr,\bb,\fs\}$, then the distribution of the spins of the children edges has a slight tilt compared to $\dHH^{\sy}$ or $\hHH^{\sy}$ due to the pairs of half-edges already matched. \footnote{If $\sigma_e \notin \{\rr,\bb,\fs\}$, then the spins of the children edges are deterministic (see Remark \ref{rmk:comp:col}).} The following lemma shows that this tilt is $O_{k,t}(n^{-1})$ and that the error arising from the difference of channels $(\dHH^{\sy},\hHH^{\sy})$ and $(\dHH^{\star},\hHH^{\star})$ is $O_{k,t}(n^{-1/2})$:
\begin{lemma}\label{lem:tilt}
    Let $\dHH\in \PPP(\CCC^d), \hHH\in \PPP(\CCC^k)$ be symmetric probability measures with the same marginal $\bHH\in \PPP(\CCC)$ that satisfy \eqref{eq:same:marginal}. For $\tau \in \{\rr,\bb,\fs\}$, non-negative integer $\ell_1\in \N$, and vectors $\dul_2\equiv (\dot{\ell}_2(\sig))_{\sig\in \CCC^d} \in \N^{\CCC^{d}}, \hul_2\equiv (\hat{\ell}_2(\sig))_{\sig\in \CCC^k} \in \N^{\CCC^{k}}$, define the conditional probability measures $\dot{H}_{\ell_1,\dul_2}(\cdot \given \tau)\equiv \dHH_{\ell_1,\dul_{2},n}(\cdot \given \tau)\in \PPP(\CCC^d)$ and $\hHH_{\ell_1,\hul_2}(\cdot \given \tau)\equiv \hHH_{\ell_1,\hul_2,n}(\cdot \given \tau)\in \PPP(\CCC^k)$ as follows. For $\sig=(\sigma_1,\ldots,\sigma_d)\in \CCC^d$ and $\sig^\prime=(\sigma^\prime_1,\ldots,\sigma^\prime_k)\in \CCC^k$, let
    \begin{equation*}
    \dHH_{\ell_1,\dul_2}\big(\sig \,\big|\, \tau\big)=\frac{n\dHH(\sig)\sum_{i=1}^{d}\one(\sigma_i=\tau)-\dot{\ell}_2(\sig)}{nd\bHH(\tau)-\ell_1}\,,~~~~~\hHH_{\ell_1,\hul_2}\big(\sig^\prime \,\big|\, \tau\big)=\frac{m\hHH(\sig^\prime)\sum_{i=1}^{k}\one(\sigma^\prime_i=\tau)-\hat{\ell}_2(\sig^\prime)}{nd\bHH(\tau)-\ell_1}\,.
    \end{equation*}
    For any constant $C>0$, there is a constant $K(C_0,t,k)>0$ such that for $(\dHH^{\sy}, \hHH^{\sy})\equiv (\dHH^{\sy}[\xi], \hHH^{\sy}[\xi])\in \PPP(\CCC^d)\times \PPP(\CCC^k)$, we have uniformly over $\xi\in \Xi_C$, $\tau\in \{\rr,\bb,\fs\}$, and $\ell_1,\ell_1^\prime, \big\|\dul_2\big\|_1, \big\|\dul_2^\prime \big\|_1, \big\|\hul_2\big\|_1, \big\|\hul_2^\prime \big\|_1\leq 2(kd)^{2t}$ that
    \begin{equation}\label{eq:lem:tilt}
    \begin{split}
    d_{\tv}\Big(\dHH_{\ell_1,\dul_2}^{\sy}\big(\cdot \,\big|\, \tau\big)\; ,\; \dHH_{\ell_1^\prime,\dul_2^\prime}^{\sy}\big(\cdot \,\big|\, \tau\big)\Big)\vee d_{\tv}\Big(\hHH_{\ell_1,\hul_2}^{\sy}\big(\cdot \,\big|\, \tau\big)\; ,\; \hHH_{\ell_1^\prime,\hul_2^\prime}^{\sy}\big(\cdot \,\big|\, \tau\big)\Big)&\leq \frac{K}{n}\,,\\
    d_{\tv}\big(\dHH^{\sy}\,,\,\dHH^\star\big)\vee d_{\tv}\Big(\dHH_{\ell_1,\dul_2}^{\sy}\big(\cdot \,\big|\, \tau\big)\; ,\; \dHH_{0,0}^{\star}\big(\cdot \,\big|\, \tau\big)\Big)\vee d_{\tv}\Big(\hHH_{\ell_1,\hul_2}^{\sy}\big(\cdot \,\big|\, \tau\big)\; ,\; \hHH_{0,0}^{\star}\big(\cdot \,\big|\, \tau\big)\Big)&\leq \frac{K}{\sqrt{n}}\,.
    \end{split}
    \end{equation}
\end{lemma}
  \begin{proof}
      By definition of $(\dHH^{\sy}, \hHH^{\sy})\equiv (\dHH^{\sy}[\xi], \hHH^{\sy}[\xi])$ in \eqref{eq:symm:coloring} and Definition \ref{def:optimal:t} of $\dHH^\star, \hHH^\star$, we have by triangular inequality that
      \begin{equation}\label{eq:dot:sy:close}
      \begin{split}
          \big\|\dHH^{\sy}-\dHH^\star\big\|_1
          &\leq \big\|\dot{B}-\dot{B}^\star\big\|_1+\sum_{\utau\in \dot{\CCC}_{\ff}}\big(|\textnormal{per}(\utau)|\big)^{-1}\big|p_{\fff(\utau)}-p^\star_{\fff(\utau)}\big|\dot{m}(\utau)\\
          &=\big\|\dot{B}-\dot{B}^\star\big\|_1+\sum_{\fff\in \FFF}\big|p_{\fff}-p^\star_{\fff}\big|v_{\fff}\leq \norm{\xi-\xi^\star}_{\tsq}\leq \frac{C}{\sqrt{n}}\,,
      \end{split}
      \end{equation}
      where the equality holds since $\sum_{\utau\in \dot{\CCC}_{\ff}}\big(|\textnormal{per}(\utau)|\big)^{-1}\dot{m}(\utau)\one(f(\tau)=\fff)=v_{\fff}$ for $\fff\in \FFF$ by definition of $\dot{m}(\utau)$ in \eqref{eq:dot:multi}, and the last inequality is by Definition \ref{def:Xi} of $\xi\in \Xi_C$. Analogously, we have that
      \begin{equation}\label{eq:hat:sy:close}
           \big\|\hHH^{\sy}-\hHH^\star\big\|_1\leq \big\|\hat{B}-\hat{B}^\star\big\|_1+\sum_{\fff\in \FFF}\big|p_{\fff}-p^\star_{\fff}\big|f_{\fff}\leq \norm{\xi-\xi^\star}_{\tsq}\leq \frac{C}{\sqrt{n}}\,.
      \end{equation}
      Note that the equations above also imply that $d_{\tv}(\bHH^{\sy},\bHH^\star)\leq \frac{C}{\sqrt{n}}$ holds, where $\bHH^{\sy}$ is the marginal of $(\dHH^{\sy},\bHH^{\sy})\equiv (\dHH^{\sy}[\xi],\bHH^{\sy}[\xi])$ and $\bar{H}^\star$ is the marginal of $(\dHH^\star, \hHH^\star)$. Moreover, the compatibility $\{p_{\fff}^\star\}_{\fff\in \FFF}\sim B^\star$ shows that the $\bar{H}^\star$ at the boundary spins $\tau\in \{\rr,\bb,\fs\}$ is given by $\bHH^\star(\tau)=\bar{B}^\star(\tau)$. From the definition of $\bar{B}^\star$ in Definition \ref{def:opt:bdry:1stmo}, we have $\min_{\tau\in \{\rr,\bb,\fs\}} \bar{B}^\star(\tau)\gtrsim_{k} 1$. Hence, it follows that uniformly over $\xi\in \Xi_C$, 
      \begin{equation}\label{eq:bar:min}
          \min_{\tau\in \{\rr,\bb,\fs\}} \bHH^\star(\tau)\wedge \bHH^{\sy}(\utau)\gtrsim_k 1\,.
      \end{equation}
      The estimates \eqref{eq:dot:sy:close}, \eqref{eq:hat:sy:close}, and \eqref{eq:bar:min} easily imply the estimate \eqref{eq:lem:tilt}, and we omit the details.
  \end{proof}
  Having Lemma \ref{lem:tilt} in hand, we now prove Proposition \ref{prop:coupling}.
  \begin{proof}[Proof of Proposition \ref{prop:coupling}]
  Throughout, we fix $C>0$ and treat it as a constant that depends on $k,t$. We also fix $\xi\in \Xi_C$. Recall that by Definition \ref{def:Xi} of $\Xi_C$, we have $\E_{\xi}[N_{\cyc}(2t;\GG)]\leq C$. Thus, for $\bv\sim \Unif(V)$, the probability of having a cycle in $N_{t}(\bv,\bGG)$ is $O_{k,t}(n^{-1})$ by Markov's inequality, which will be useful in the coupling below.
  
  First, we consider the single copy estimate $d_{\tv}(\E_{\xi}\nu_t[\bGG,\bux], \nu^\star_t)\lekt n^{-1/2}$. We proceed by coupling the exploration process $(\bsig_t(\bv),\buL_t(\bv))\sim \E_{\xi}\nu_t[\bGG,\bux]$ in Remark \ref{rmk:exploration} and the broadcast process $\{(\bsigma^\star_e, \bL^\star_e)\}_{e\in E(\TTT_{d,k,t})}$ with channel $(\dHH^\star,\hHH^\star)$ in Definition \ref{def:optimal:t}. Throughout, we abbreviate $(\dHH^{\sy},\bHH^{\sy})\equiv (\dHH^{\sy}[\xi],\bHH^{\sy}[\xi])$.

  We begin by revealing the spins in the neighbor of $\bv$ and $\rho$, i.e. $\bsig_{\delta \bv}$ and $\bsig^\star_{\delta \rho}$. Recall that the distribution of $\bsig_{\delta \bv}$ is $\dHH^{\sy}$ and the distribution of $\bsig^\star_{\delta \rho}$ is $\dHH^\star$. By Lemma \ref{lem:tilt}, we have $d_{\tv}(\dHH^{\sy},\dHH^\star)\lekt n^{-1/2}$, thus we can couple $\bsig_{\delta \bv}$ and $\bsig^\star_{\delta \rho}$ so that $\bsig_{\delta \bv}=\bsig^\star_{\delta \rho}$ with probability at least $1-O_{k,t}(n^{-1/2})$.

  Next, we sequentially reveal the spins and the literals associated with the `children half-edges' of a boundary half-edge. That is, for each boundary half-edge $e$ adjacent to a variable, we reveal the half-edge $e^\prime$ that is matched with $e$, and reveal the spins and literals associated with children half-edges $\delta a(e^\prime)\setminus e^\prime$. If boundary half-edge $e$ is adjacent to a clause, we reveal the half-edge $e^\prime$ matched with $e$ and only reveal the spins associated with children half-edges $\delta v(e^\prime)\setminus e^\prime$. This procedure is carried out by utilizing a breadth-first search for both the neighbors of $\bv$ and $\rho$.
  
  At time $\ell\geq 1$, denote the revealed neighborhood of $\bv$ by $\NN_{\ell}(\bv)$ and the revealed neighborhood of $\rho$ by $\NN_{\ell}(\rho)$. We let $\partial \NN_{\ell}(\bv)$ and $\partial \NN_{\ell}(\rho)$ be the set of boundary half-edges of $\NN_{\ell}(\bv)$ and $\NN_{\ell}(\rho)$. Also, denote the revealed spins (resp. literals) in $\NN_{\ell}(\bv)$ by $\bsig_{\NN_{\ell}(\bv)}$ (resp. $\buL_{\NN_{\ell}(\bv)}$) and the revealed spins (resp. literals) in $\NN_{\ell}(\rho)$ by $\bsig^\star_{\NN_{\ell}(\rho)}$ (resp. $\buL^\star_{\NN_{\ell}(\rho)}$). Then, define the event $\EEE_{\ell}$ of success by
  \begin{equation*}
      \EEE_{\ell}:= \Big\{\textnormal{$\NN_{\ell}(\bv)$   is   a   tree  
 and  
 $(\bsig_{\NN_{\ell}(\bv)},\buL_{\NN_{\ell}(\bv)})=(\bsig^\star_{\NN_{\ell}(\rho)},\buL^\star_{\NN_{\ell}(\rho)})$}  \Big\}\,.
  \end{equation*}

  Now, suppose at time $\ell+1$, we take a boundary half-edge $e\in \partial \NN_{\ell}(\bv)$ adjacent to a variable $v(e)\in \NN_{\ell}(\bv)$, and reveal the connection of $e$, and the spins and literals of children half-edges of $e$. Note that the probability of creating a cycle by revealing the connection of $e$ is $O_{k,t}(n^{-1})$ since \textit{a priori}, the probability of having a cycle in $N_t(\bv,\bGG)$ is $O_{k,t}(n^{-1})$ by definition of $\Xi_C$. Moreover, if the spin at $e$, $\bsigma_e$, is free, i.e. $\bsigma_e \notin \{\rr,\bb,\fs\}$, then the spins and literals of children half-edges $\delta a(e)\setminus e$ is completely determined by $\bsigma_e$ (cf. Remark \ref{rmk:comp:col}).
  
  On the other hand, if $\bsigma_e\in \{\rr,\bb,\fs\}$, then conditioned on $\NN_{\ell}(\bv)$ and $\bsig_{\NN_{\ell}(\bv)}$, $\bsig_{\delta a(e)}$ is drawn from $\hHH^{\sy}_{\ell_1,\hul_2}(\cdot \given \bsigma_e)$, defined in Lemma \ref{lem:tilt}. Here, $\ell_1\in \N, \hul_2\equiv (\hat{\ell}_2(\sig))_{\sig\in \CCC^k} \in \N^{\CCC^k}$ is determined by $\bsig_{\NN_{\ell}(\bv)}$. More precisely, $\ell_1$ is the number of edges $e^\prime$ in $\NN_{\ell}(\bv)$ that have spins $\bsigma_{e^\prime}=\bsigma_e$, and $\hat{\ell}_2(\sig)$ is the number of clauses $a$ in $\NN_{\ell}(\bv)$ that have spin neighborhood $\bsig_{\delta a}=\sig$ (up to a permutation) times the number of $\bsig_e$ in $\sig$. In particular, note that $\ell_1, \big\|\hul_2\big\|_1\leq (kd)^{2t}$ holds.
  
  Thus, Lemma \ref{lem:tilt} shows that conditioned on $\NN_{\ell}(v), \bsig_{\NN_{\ell}(\bv)}$, and $\EEE_{\ell}$, we can couple the spins of the children half-edges $\bsig_{\delta a(e)\setminus e}$ and $\bsig^\star_{\delta a(\phi(e))\setminus \phi(e)}$, where $\phi(e)$ is the boundary half-edge of $\NN_{\ell}(\rho)$ corresponding to $e$, so that $\bsig_{\delta a(e)\setminus e}=\bsig^\star_{\delta a(\phi(e))\setminus \phi(e)}$ with probability at least $1-O_{k,t}(n^{-1/2})$. Finally, conditioned on $\bsig_{\delta a(e)\setminus e}=\bsig^\star_{\delta a(\phi(e))\setminus \phi(e)}$, the literals $\buL_{\delta a(e)}$ and $\buL^\star_{a(\phi(e))}$ is distributed the same, so we can use the same randomness to ensure $\buL_{\delta a(e)}=\buL^\star_{a(\phi(e))}$. Therefore, we have that
  \begin{equation}\label{eq:success:coupling}
      \P\big(\EEE_{\ell+1}^{\textsf{c}}\cap \EEE_{\ell}\big)\lekt n^{-1/2}\,.
  \end{equation}
  The same analysis applies when the boundary half-edge $e\in \delta\NN_{\ell}(\bv)$ is adjacent to a clause $a(e)\in \NN_{\ell}(\bv)$ to ensure that \eqref{eq:success:coupling}. Since it takes at most $\ell\leq (kd)^{2t}=O_{k,t}(1)$ times to explore the $2t-\frac{3}{2}$ neighborhood around $\bv$, summing up \eqref{eq:success:coupling} shows that $d_{\tv}(\E_{\xi}\nu_t[\bGG,\bux], \nu^\star_t)\lekt n^{-1/2}$.

  Second, we consider the pair copy estimate $d_{\tv}\left(\E_{\xi}[\bnu_t\otimes \bnu_t], \E_{\xi}\bnu_t\otimes \E_{\xi}\bnu_t\right)\lekt n^{-1}$ for $\bnu_t\equiv \nu_t[\bGG,\bux]$. Observe that $\E_{\xi}[\bnu_t\otimes \bnu_t]$ is the law of $\big(\bsig_t(\bv^1,\bGG),\buL_t(\bv^1,\bGG),\bsig_t(\bv^2,\bGG),\buL_t(\bv^2,\bGG)\big)$, where $(\bGG,\bsig)\sim \P_{\xi}$ and $\bv^1,\bv^2 \stackrel{i.i.d}{\sim} \Unif(V)$. Note that if we let $(\bGG^1,\bsig^1)=(\bGG,\bsig)$ and $(\bGG^2,\bsig^2)\sim \P_{\xi}$ be independent of $(\bGG^1,\bsig^1)$, then $ \E_{\xi}\bnu_t\otimes \E_{\xi}\bnu_t$ is the law of $\big(\bsig_t(\bv^1,\bGG^1),\buL_t(\bv^1,\bGG^1),\bsig_t(\bv^2,\bGG^2),\buL_t(\bv^2,\bGG^2)\big)$. Thus, it suffices to construct a coupling of $\big(\bsig_t(\bv^2,\bGG^1),\buL_t(\bv^2,\bGG^1)\big)$ and $\big(\bsig_t(\bv^2,\bGG^2),\buL_t(\bv^2,\bGG^2)\big)$ conditional on $N_{t}(\bv^1,\bGG^1), \bsig_t(\bv^1,\bGG^1),$ and $\buL_t(\bv^1,\bGG^1)$.

  Observe that conditional on $N_{t}(\bv^1,\bGG^1), \bsig_t(\bv^1,\bGG^1), \buL_t(\bv^1,\bGG^1)$, and the event $N_t(\bv^1, \bGG^1)\cap N_t(\bv^2, \bGG^1)=\emptyset$, which happens with probability $1-O_{k,t}(n^{-1})$, the law of $\big(\bsig_t(\bv^2,\bGG^1),\buL_t(\bv^2,\bGG^1)\big)$ can be described as an alternate exploration process. The only difference between the exploration process described in Remark \ref{rmk:exploration} is that the variables, clauses, full-edges, and boundary half-edges of $N_{t}(\bv^1,\bGG^1)$ cannot be used nor matched during the process. Thus, we can couple these 2 exploration processes by a breadth-first search as before. The probability of error comes from 2 sources as before. The first is from forming a cycle in either of the processes, which has probability $O_{k,t}(n^{-1})$. The second is the difference of $\ell_1, \dul_2,\hul_2$ regarding the distribution of the spins associated with children half-edges. That is, because the alternate process cannot use variables, clauses, and edges of $N_{t}(\bv^1,\bGG^1)$, the associated $\ell_1^\prime, \dul_2^\prime ,\hul_2^\prime$ can differ from $\ell_1,\dul_2,\hul_2$ associated with the original exploration process. Note that however, the number of variables, clauses, and edges in $N_{t}(\bv^1,\bGG^1)$ is at most $(kd)^{2t}$, so we still have that $\ell_1^\prime, \big\|\dul_2^\prime\big\|_1,\big\|\hul_2^\prime\big\|_1\leq 2(kd)^{2t}$. Hence, by Lemma \ref{lem:tilt}, the error probability is at most $O_{k,t}(n^{-1})$. We, therefore, conclude that $d_{\tv}\left(\E_{\xi}[\bnu_t\otimes \bnu_t], \E_{\xi}\bnu_t\otimes \E_{\xi}\bnu_t\right)\lekt n^{-1}$.
  \end{proof}

  We conclude this section with the proof of Lemma \ref{lem:bdry:cycle}
  \begin{proof}[Proof of Lemma \ref{lem:bdry:cycle}]
      By Markov's inequality, it suffices to show that $\E_{\xi} N^{\sf b}_{\cyc}(2t;\bGG,\bux)\leq n^{1/4}$, uniformly over $\xi=(B,\{p_{\fff}\}_{\fff\in \FFF})$ such that $\{p_{\fff}\}_{\fff\in \FFF}\in \ee_{\frac{1}{4}}$, $p_{\fff}=0$ if $\fff$ is multi-\rev{cyclic}, and $\norm{B-B^\star}_{1}\leq n^{-1/3}$. Note that from the definition of $\bar{B}^\star$ in Definition \ref{def:opt:bdry:1stmo}, we have $\min_{\tau\in \{\rr,\bb,\fs\}} \bar{B}^\star(\tau)\gtrsim_{k} 1$, thus if $\norm{B-B^\star}_1\leq n^{-1/3}$, then we have $\min _{\tau\in \{\rr,\bb,\fs\}}\bar{B}(\tau)\gtrsim_k 1$. Note that $\min _{\tau\in \{\rr,\bb,\fs\}}\bar{B}(\tau)\gtrsim_k 1$ is equivalent to the fact that there are $\Omega_k(n)$ number of edges that have spins $\tau$ for each $\tau\in \{\rr,\bb,\fs\}$. This fact plays a key role in the analysis below.

      Let $(\bGG, \bsig)$ be the component coloring corresponding to $(\bGG,\bux)\sim \P_{\xi}$. Then, recall from Observation \ref{obs:config} that $(\bGG, \bsig)$ can be drawn from a configuration model. In the description of the configuration model in Observation \ref{obs:config}, we further condition on Steps $(a)$, $(b)$, and the connections of free components therein. That is, we condition on the location of the variables (resp. clauses) with assigned spin neighborhoods $\utau \in \CCC^{d}$ (resp. $\utau \in \CCC^k$), and also on the matchings of the half-edges with associated spin $\tau \notin \{\rr,\bb,\fs\}$. Then, the randomness comes from the u.a.r. matching of the half-edges that have the same associated spins $\tau \in \{\rr,\bb,\fs\}$, and also from assigning literals. Since we do not bother with the literals of $\CC$, we only analyze the former randomness with regards to the u.a.r. matching of the $\{\rr,\bb,\fs\}$-colored spins.

      Let $\CC=(e_1,\ldots, e_{2\kappa})$ be a self-avoiding cycle \rev{that has $2t$ edges colored $\{\rr,\bb,\fs\}$}. Let $(\SS_l)_{1\leq l\leq K}$ be the \textit{boundary segments} of $\CC$, where $\SS_{l}= (e_{i_l},e_{i_l+1},\ldots, e_{j_l})$ satisfies the condition below. Intuitively, the boundary segment encodes a path that connects two free components by $\{\rr,\bb,\fs\}$-colored edges.
      \begin{enumerate}[label=(\arabic*)]
          \item Suppose that the endpoints of the segment $(e_{i_l},\ldots ,e_{j_l})$ are formed by variables $v(e_{i_l})$ and $v(e_{j_l})$. That is, $a(e_{i_{l}})=a(e_{i_{l}+1})\,,\, v(e_{i_l+1})=v(e_{i_l+2})\,,\, \ldots\,,\, a(e_{j_l-1})=a(e_{j_l})$. Then, $v(e_{i_l})$ and $v(e_{j_l})$ are free, but all the other variables $v(e_{i_l+1}),\ldots, v(e_{j_l-1})$ are frozen. Also, the clauses $a(e_{i_l}),\,\ldots,\, a(e_{j_l})$ are separating.
          \item If the endpoints of the segment $(e_{i_l},\ldots ,e_{j_l})$ are clauses, then $a(e_{i_l})$ and $a(e_{j_l})$ is non-separating while all the other clauses in the segment \rev{are} separating. Also, the variables in the segment is frozen.
          \item If the endpoints of the segment $(e_{i_l},\ldots ,e_{j_l})$ are formed by a variable and a clause, they are respectively free and non-separating. All the other variables and clauses are respectively frozen and separating.
      \end{enumerate}
      We assume that $\SS_l$ and $\SS_{l+1}$, where $\SS_{K+1}\equiv \SS_1$, are adjacent in the sense that the right endpoint in $\SS_l$, which is either $v(e_{j_l})$ or $a(e_{j_l})$, and the left endpoint in $\SS_{l+1}$, which is either $v(e_{i_{l+1}})$ or $a(e_{i_{l+1}})$, lies in the same free component. Also, we denote the length of $\SS_l$ by $L_l$ for $1\leq l\leq K$. Then $\sum_{l=1}^K L_l=2t$ holds, \rev{so} we have $K\leq 2t$.

      Observe that the cycle $\CC$ is `almost' determined by the boundary segments $(\SS_l)_{l\leq K}$. Namely, since all cyclic free components have at most one cycle, there are at most $2$ paths from the right endpoint of $\SS_l$ to the left endpoint of $\SS_{l+1}$. Therefore, if we fix the configuration of boundary segments $(\SS_l)_{l \leq K}$, there are at most $2^K \leq 2^{2t}$ corresponding self-avoiding cycles \rev{that have $2t$ $\{\rr,\bb,\fs\}$-colored edges.}
      
      Now suppose that $e_i$ is contained in a boundary segment. Then, by its definition, either $v(e_i)$ is frozen or $a(e_i)$ is separating (both statements hold when $v(e_i), a(e_i)$ are not endpoints). Without loss of generality, suppose $v(e_i)$ is frozen. Then, the revealed spin neighborhood of $v(e_i)$ must be $\utau=(\tau_{l})_{l\leq d} \in \{\rr,\bb\}^d$. Since we have $\min _{\tau\in \{\rr,\bb,\fs\}}\bar{B}(\tau)\gtrsim_k 1$, the number of edges with spin $\tau_l$ is $\Omega_k(n)$ for every $l\leq d$. Thus, the probability of having an edge $e_i=(v(e_i),a(e_i))$ under u.a.r. matching of the half-edges with the same spins is at most $O_{k}(n^{-1})$. Therefore, the probability of containing a specific configuration of a boundary segment $\SS$ is $O_k(n^{-L(\SS)})$, where $L(\SS)$ denotes the length of $\SS$.

      Moreover, note that the number of choosing variables and clauses involved in $(\SS_{\ell})_{1\leq \ell \leq K}$ is at most $O_{k,t}(n^{2t} (\log n)^{K})$, which can be argued as follows. The number of choosing variables and clauses involved in $\SS_1$ is $O_{k}(n^{L_1+1})$. Then, since the left endpoint of $\SS_2$ must lie in the same free component as the right endpoint of $\SS_1$, and the largest free component has $O_{k}(\log n)$ number of variables and clauses (cf. $\xi \in \ee_{\frac{1}{4}}$), there are at most $O_{k}(n^{L_2}\log n)$ number of choosing the variables and the clauses involved in $\SS_2$. By repeating this and noting that the right endpoint of $\SS_K$ must lie in the same free component as in the left endpoint of $\SS_1$, the total number is $O_{k,t}(n^{2t} (\log n)^{K})$.

      To conclude, we have shown that the probability of containing boundary segments $(\SS_l)_{l \leq K}$ is at most $O_k(n^{-2t})$, and \rev{the} number of choosing the $(\SS_l)_{l \leq K}$ is at most $O_{k,t}(n^{2t} (\log n)^{K})$. Also, $K\leq 2t$ and there are $2^{2K}$ number of self-avoiding cycles \rev{that have $2t$ $\{\rr,\bb,\fs\}$-colored edges.} Therefore, it follows that $\E_{\xi}[N^{\sf b}_{\cyc}[2t;\bGG,\bux]\lekt (\log n)^{2t}\ll n^{1/4}$ hold uniformly over $\xi=(B,\{p_{\fff}\}_{\fff\in \FFF})$ such that $\{p_{\fff}\}_{\fff\in \FFF}\in \ee_{\frac{1}{4}}$, $p_{\fff}=0$ if $\fff$ is multi-\rev{cyclic}, and $\norm{B-B^\star}_{1}\leq n^{-1/3}$, which concludes the proof.
  \end{proof}

    \section*{Acknowledgements}
    We thank Andrea Montanari for helpful discussions. Youngtak Sohn is supported by Simons-NSF Collaboration on Deep Learning NSF DMS-2031883 and Elchanan Mossels's Vannevar Bush Faculty Fellowship award ONR-N00014-20-1-2826. Allan Sly is supported by a Simons Investigator grant and a MacArthur Fellowship.

	\bibliography{naesatref}

\begin{thebibliography}{10}

\bibitem{achlioptas2008algorithmic}
{\sc Achlioptas, D., and Coja-Oghlan, A.}
\newblock Algorithmic barriers from phase transitions.
\newblock In {\em 2008 49th Annual IEEE Symposium on Foundations of Computer
  Science\/} (2008), IEEE, pp.~793--802.

\bibitem{am06}
{\sc Achlioptas, D., and Moore, C.}
\newblock Random {$k$}-{SAT}: two moments suffice to cross a sharp threshold.
\newblock {\em SIAM J. Comput. 36}, 3 (2006), 740--762.

\bibitem{anp05}
{\sc Achlioptas, D., Naor, A., and Peres, Y.}
\newblock Rigorous location of phase transitions in hard optimization problems.
\newblock {\em Nature 435}, 7043 (2005), 759--764.

\bibitem{aldous2001zeta}
{\sc Aldous, D.~J.}
\newblock The $\zeta$ (2) limit in the random assignment problem.
\newblock {\em Random Structures \& Algorithms 18}, 4 (2001), 381--418.

\bibitem{benjamini2001recurrence}
{\sc Benjamini, I., and Schramm, O.}
\newblock Recurrence of distributional limits of finite planar graphs.
\newblock {\em Electronic Journal of Probability [electronic only] 6\/} (2001).

\bibitem{Bordenave15}
{\sc Bordenave, C., and Caputo, P.}
\newblock Large deviations of empirical neighborhood distribution in sparse
  random graphs.
\newblock {\em Probability Theory and Related Fields 163}, 1 (2015), 149--222.

\bibitem{Borokov17}
{\sc Borovkov, A.~A.}
\newblock Generalization and refinement of the integro-local stone theorem for
  sums of random vectors.
\newblock {\em Theory of Probability \& Its Applications 61}, 4 (2017),
  590--612.

\bibitem{coja14stoc}
{\sc Coja-Oghlan, A.}
\newblock The asymptotic k-{SAT} threshold.
\newblock In {\em Proceedings of the Forty-Sixth Annual ACM Symposium on Theory
  of Computing\/} (New York, NY, USA, 2014), STOC '14, Association for
  Computing Machinery, p.~804–813.

\bibitem{coja2018local}
{\sc Coja-Oghlan, A., Efthymiou, C., and Jaafari, N.}
\newblock Local convergence of random graph colorings.
\newblock {\em Combinatorica 38}, 2 (2018), 341--380.

\bibitem{coja2020replica}
{\sc Coja-Oghlan, A., Kapetanopoulos, T., and M{\"u}ller, N.}
\newblock The replica symmetric phase of random constraint satisfaction
  problems.
\newblock {\em Combinatorics, Probability and Computing 29}, 3 (2020),
  346--422.

\bibitem{cp13stoc}
{\sc Coja-Oghlan, A., and Panagiotou, K.}
\newblock Going after the k-{SAT} threshold.
\newblock In {\em Proceedings of the Forty-Fifth Annual ACM Symposium on Theory
  of Computing\/} (New York, NY, USA, 2013), STOC '13, Association for
  Computing Machinery, p.~705–714.

\bibitem{DZ10}
{\sc Dembo, A., and Zeitouni, O.}
\newblock {\em Large deviations techniques and applications}, vol.~38 of {\em
  Stochastic Modelling and Applied Probability}.
\newblock Springer-Verlag, Berlin, 2010.

\bibitem{dss16}
{\sc Ding, J., Sly, A., and Sun, N.}
\newblock Satisfiability threshold for random regular {NAE-SAT}.
\newblock {\em Communications in Mathematical Physics 341}, 2 (2016), 435--489.

\bibitem{dss22}
{\sc Ding, J., Sly, A., and Sun, N.}
\newblock {Proof of the satisfiability conjecture for large $k$}.
\newblock {\em Annals of Mathematics 196}, 1 (2022), 1 -- 388.

\bibitem{FlajoletSedgewick}
{\sc Flajolet, P., and Sedgewick, R.}
\newblock {\em Analytic combinatorics}.
\newblock Cambridge University press, 2009.

\bibitem{gerschenfeld2007reconstruction}
{\sc Gerschenfeld, A., and Montanari, A.}
\newblock Reconstruction for models on random graphs.
\newblock In {\em 48th Annual IEEE Symposium on Foundations of Computer Science
  (FOCS'07)\/} (2007), IEEE, pp.~194--204.

\bibitem{janson1995random}
{\sc Janson, S.}
\newblock Random regular graphs: asymptotic distributions and contiguity.
\newblock {\em Combinatorics, Probability and Computing 4}, 4 (1995), 369--405.

\bibitem{kmrsz07}
{\sc Kr\c{z}aka\l{}a, F., Montanari, A., Ricci-Tersenghi, F., Semerjian, G.,
  and Zdeborov{\'a}, L.}
\newblock Gibbs states and the set of solutions of random constraint
  satisfaction problems.
\newblock {\em Proceedings of the National Academy of Sciences 104}, 25 (2007),
  10318--10323.

\bibitem{mm09}
{\sc M\'{e}zard, M., and Montanari, A.}
\newblock {\em Information, physics, and computation}.
\newblock Oxford Graduate Texts. Oxford University Press, 01 2009.

\bibitem{mpz02}
{\sc M{\'e}zard, M., Parisi, G., and Zecchina, R.}
\newblock Analytic and algorithmic solution of random satisfiability problems.
\newblock {\em Science 297}, 5582 (2002), 812--815.

\bibitem{molloy2018freezing}
{\sc Molloy, M.}
\newblock The freezing threshold for k-colourings of a random graph.
\newblock {\em Journal of the ACM (JACM) 65}, 2 (2018), 1--62.

\bibitem{mr13}
{\sc Molloy, M., and Restrepo, R.}
\newblock Frozen variables in random boolean constraint satisfaction problems.
\newblock In {\em Proceedings of the Twenty-fourth Annual ACM-SIAM Symposium on
  Discrete Algorithms\/} (Philadelphia, PA, USA, 2013), SODA '13, Society for
  Industrial and Applied Mathematics, pp.~1306--1318.

\bibitem{montanari2012weak}
{\sc Montanari, A., Mossel, E., and Sly, A.}
\newblock The weak limit of ising models on locally tree-like graphs.
\newblock {\em Probability Theory and Related Fields 152\/} (2012), 31--51.

\bibitem{montanari2011reconstruction}
{\sc Montanari, A., Restrepo, R., and Tetali, P.}
\newblock Reconstruction and clustering in random constraint satisfaction
  problems.
\newblock {\em SIAM Journal on Discrete Mathematics 25}, 2 (2011), 771--808.

\bibitem{montanari2008clusters}
{\sc Montanari, A., Ricci-Tersenghi, F., and Semerjian, G.}
\newblock Clusters of solutions and replica symmetry breaking in random
  k-satisfiability.
\newblock {\em Journal of Statistical Mechanics: Theory and Experiment 2008},
  04 (2008), P04004.

\bibitem{mossel2015reconstruction}
{\sc Mossel, E., Neeman, J., and Sly, A.}
\newblock Reconstruction and estimation in the planted partition model.
\newblock {\em Probability Theory and Related Fields 162\/} (2015), 431--461.

\bibitem{MSS22}
{\sc Mossel, E., Sly, A., and Sohn, Y.}
\newblock Exact phase transitions for stochastic block models and
  reconstruction on trees.
\newblock In {\em Proceedings of the 55th Annual ACM Symposium on Theory of
  Computing\/} (New York, NY, USA, 2023), STOC 2023, Association for Computing
  Machinery, p.~96–102.

\bibitem{NSS}
{\sc Nam, D., Sly, A., and Sohn, Y.}
\newblock One-step replica symmetry breaking of random regular {NAE-SAT} {I}.
\newblock {\em arXiv preprint, arXiv:2011.14270\/} (2020).

\bibitem{nss2}
{\sc Nam, D., Sly, A., and Sohn, Y.}
\newblock One-step replica symmetry breaking of random regular {NAE-SAT II}.
\newblock {\em Communications in Mathematical Physics 405}, 3 (2024), 61.

\bibitem{panchenko13a}
{\sc Panchenko, D.}
\newblock {\em The {S}herrington-{K}irkpatrick model}.
\newblock Springer Monographs in Mathematics. Springer, New York, 2013.

\bibitem{Parisi02}
{\sc Parisi, G.}
\newblock On local equilibrium equations for clustering states.
\newblock {\em arXiv:cs/0212047\/} (2002).

\bibitem{robinson1994almost}
{\sc Robinson, R.~W., and Wormald, N.~C.}
\newblock Almost all regular graphs are hamiltonian.
\newblock {\em Random Structures \& Algorithms 5}, 2 (1994), 363--374.

\bibitem{ssz22}
{\sc Sly, A., Sun, N., and Zhang, Y.}
\newblock The number of solutions for random regular {NAE-SAT}.
\newblock {\em Probability Theory and Related Fields 182}, 1 (2022), 1--109.

\bibitem{zdeborova2007phase}
{\sc Zdeborov{\'a}, L., and Krzaka{\l}a, F.}
\newblock Phase transitions in the coloring of random graphs.
\newblock {\em Physical Review E 76}, 3 (2007), 031131.

\end{thebibliography}
\newpage 
\appendix

 \section{The belief-propagation fixed point}\label{sec:BP}
	In this appendix, we review the combinatorial models for frozen configurations described in Section 2 of \cite{ssz22}. We also review the \textit{belief propagation fixed point} established in Section 5 of \cite{ssz22}, and the related notion of \textit{optimal boundary profile} in Section 3 of \cite{NSS}.
	\subsection{Message configurations and the Bethe formula}\label{subsec:model:sszrev}
    We first introduce the \textit{message configuration}, which enables us to calculate the size of a free tree, i.e. $w_{\ttt}$, by local quantities. The message configuration is given by $\utau = (\tau_e)_{e\in E} \in \mathscr{M}^E$ ($\mathscr{M}$ is defined below). Here, $\tau_e=(\dot{\tau}_e,\hat{\tau}_e)$, where $\dot{\tau}$ (resp. $\hat{\tau}$) denotes the message from $v(e)$ to $a(e)$ (resp. $a(e)$ to $v(e)$).
	
	A message will carry information of the structure of the free tree it belongs to. To this end, we first define the notion of \textit{joining} $l$ trees at a vertex (either variable or clause) to produce a new tree.
	Let $t_1,\ldots , t_l$ be  a collection of rooted bipartite factor trees satisfying the following conditions:
	\begin{itemize}
		\item Their roots $\rho_1,\ldots,\rho_l$ are all of the same type (i.e., either all-variables or all-clauses) and are all degree one.
		
		\item 
		If an edge in $t_i$ is adjacent to a degree one vertex, which is not the root, then the edge is called a \textbf{boundary-edge}. The rest of the edges are called \textbf{internal-edges}. For the case where $t_i$ consists of a single edge and a single vertex, we regard the single edge to be a boundary edge.
		
		\item $t_1,\ldots,t_l$ are \textbf{boundary-labeled trees}, meaning that their variables, clauses, and internal edges are unlabeled (except we distinguish the root), but the boundary edges are assigned with values from $\{0,1,\fs \}$, where $\fs$ stands for `separating'. 
	\end{itemize}
	
    Then, the joined tree $t \equiv  \textsf{j}(t_1,\ldots, t_l) $ is obtained by identifying all the roots as a single vertex $o$, and adding an edge which joins $o$ to a new root $o'$ of an opposite type of $o$ (e.g., if $o$ was a variable, then $o'$ is a clause). Note that $t= \textsf{j}(t_1,\ldots,t_l)$ is also a boundary-labeled tree, whose labels at the boundary edges are induced by those of $t_1,\ldots,t_l$.
	
	For the simplest trees that consist of a single vertex and a single edge, we use $0$ (resp. $1$) to stand for the ones whose edge is labeled $0$ (resp. $1$): for the case of $\dot{\tau}$, the root is the clause, and for the case of $\hat{\tau}$, the root is the variable. Also,  if its root is a variable and its edge is labeled $\fs$, we write the tree as $\fs$. 
	
	We can also define the Boolean addition to a boundary-labeled tree $t$ as follows. For the trees $0,1$, the Boolean-additions $0\oplus \tL$, $1\oplus\tL$ are defined as above ($t\oplus \tL$), and we define $\fs \oplus \tL = \fs$ for $\tL\in\{0,1\}$. For the rest of the trees, $t \oplus 0 := t$, and $t\oplus 1$ is the boundary-labeled tree with the same graphical structure as $t$ and the labels of the boundary Boolean-added by $1$ (Here, we define $\fs \oplus 1 = \fs$ for the $\fs$-labels).
	
	\begin{defn}[Message configuration]\label{def:msg config}
		Let $\dot{\MMM}_0:= \{0,1,\star \}$ and $\hat{\MMM}_0:= \emptyset$. Suppose that $\dot{\MMM}_t, \hat{\MMM}_t$ are defined, and we inductively define $\dot{\MMM}_{t+1}, \hat{\MMM}_{t+1}$ as follows: For $\hat{\utau} \in (\hat{\MMM}_t)^{d-1}$, $\dot{\utau} \in (\dot{\MMM}_t)^{k-1}$, we write $\{\hat{\tau}_i \}:= \{\hat{\tau}_1,\ldots,\hat{\tau}_{d-1} \}$ and similarly for $\{\dot{\tau}_i \}$. We define
		\begin{eqnarray}
		\hat{T}\left(\dot{\utau} \right):= 
		\begin{cases}
		0 & \{\dot{\tau}_i \} = \{  1 \};\\
		1 & \{ \dot{\tau }_i \} = \{0  \};\\
		\fs & \{\dot{\tau}_i \} \supseteq \{ 0,1 \};\\
		\star & \star \in \{ \dot{\tau}_i \}, \{0,1 \} \nsubseteq \{\dot{\tau}_i \};\\
		\mathsf{j}\left(\dot{\utau} \right) & \textnormal{otherwise}, 
		\end{cases}
		&  
		\dot{T}(\hat{\utau}) :=
		\begin{cases}
		0 & 0 \in \{\hat{\tau}_i \} \subseteq \hat{\MMM}_t \setminus\{ 1\};\\
		1 &  1\in \{\hat{\tau}_i\} \subseteq \hat{\MMM}_t \setminus \{0 \};\\
		\tz & \{0,1 \} \subseteq \{\hat{\tau}_i \};\\
		\star & \star \in \{\hat{\tau}_i \} \subseteq \hat{\MMM}_t \setminus \{0,1\};\\
		\mathsf{j}\left(\hat{\utau} \right) & \{\hat{\tau}_i \} \subseteq \hat{\MMM}_t\setminus \{0,1,\star \}.
		\end{cases}
		\end{eqnarray}
		Further, we set $\dot{\MMM}_{t+1} := \dot{\MMM}_t \cup \dot{T}( \hat{\MMM}_t^{d-1}  ) \setminus \{\tz \}$, and $\hat{\MMM}_{t+1}:= \hat{\MMM}_t \cup \hat{T}(\dot{\MMM}_t^{k-1} )$, and define $\dot{\MMM}$ (resp. $\hat{\MMM}$) to be the union of all $\dot{\MMM}_t$ (resp. $\hat{\MMM}_t$) and $\MMM:= \dot{\MMM} \times \hat{\MMM}$. Then, a \textbf{message configuration} on $\GG=(V,F,E,\uL)$ is a configuration $\utau \in \MMM^E$ that satisfies the local equations given by
		\begin{equation}\label{eq:def:localeq:msg}
		\tau_e = (\dot{\tau}_e, \hat{\tau}_e) = \left(\dot{T}\big(\hat{\utau}_{\delta v(e)\setminus e} \big), \tL_e \oplus  \hat{T} \big((\uL \oplus \dot{\utau})_{\delta a(e)\setminus e} \big) \right),
		\end{equation}
		for all $e\in E$.
	\end{defn} 
	
	In the definition, $\star$ is the symbol introduced to cover cycles, and $\tz$ is an error message. See Figure 1 in Section 2 of \cite{ssz22} for an example of $\star$ message. 
	
	When a frozen configuration $\ux$ on $\GG$ with no free cycles is given, we can construct a message configuration $\utau$ via the following procedure:
	\begin{enumerate}
		\item For a forcing edge $e$, set $\hat{\tau}_e=x_{v(e)}$. Also, for an edge $e\in E$, if there exists $e^\prime \in \delta v(e) \setminus e$ such that $\hat{\tau}_{e^\prime} \in \{0,1\}$, then set $\dot{\tau}_e=x_{v(e)}$.
		\item For an edge $e\in E$, if there exists $e_1,e_2\in \delta a(e)\setminus e$ such that $\{\tL_{e_1}\oplus\dot{\tau}_{e_1}, \tL_{e_2}\oplus\dot{\tau}_{e_2}\}=\{0,1\}$, then set $\hat{\tau}_e = \fs$.
		
		\item After these steps, apply the local equations \eqref{eq:def:localeq:msg} recursively to define $\dot{\tau}_e$ and $\hat{\tau}_e$ wherever possible.
		
		\item For the places where it is no longer possible to define their messages until the previous step, set them to be $\star$.
	\end{enumerate}
	
	In fact, the following lemma shows the relation between the frozen and message configurations. We refer to \cite{ssz22}, Lemma 2.7 for its proof.
	
	\begin{lemma}\label{lem:model:bij:frozen and msg}
		The mapping explained above defines a bijection
		\begin{equation}\label{eq:frz msg 1to1}
		\begin{Bmatrix}
		\textnormal{Frozen configurations } \ux \in \{0,1,\ff \}^V\\
		\textnormal{without free cycles}
		\end{Bmatrix}
		\quad
		\longleftrightarrow
		\quad
		\begin{Bmatrix}
		\textnormal{Message configurations}\\
		\utau \in \MMM^E
		\end{Bmatrix}.
		\end{equation}
	\end{lemma}

	Next, we introduce a dynamic programming method based on \textit{belief propagation} to calculate the size of a free tree by local quantities from a message configuration. 
	
	\begin{defn}\label{def:msg:m}
		Let $\mathcal{P}\{0,1\} $ denote the space of probability measures on $\{0,1\}$. We define the mappings $\dot{\mm}:\dot{\MMM} \rightarrow \PP\{0,1\}$ and $\hat{\mm} : \hat{\MMM} \rightarrow \PP \{0,1\}$ as follows. For $\dot{\tau}\in\{0,1\}$ and $\hat{\tau}\in\{0,1\}$, let $\dot{\mm}[\dot{\tau}] =\delta_{\dot{\tau}}$, $\hat{\mm}[\hat{\tau}] = \delta_{\hat{\tau}}$. For $\dot{\tau}\in\dot{\MMM} \setminus \{0,1,\star\}$ and $\hat{\tau}\in\hat{\MMM} \setminus \{0,1,\star \}$, $\dot{\mm}[\dot{\tau}]$ and $\hat{\mm}[\hat{\tau}]$ are recursively defined:
		\begin{itemize}
			\item Let $\dot{\tau} = \dot{T}(\hat{\tau}_1,\ldots,\hat{\tau}_{d-1})$, with $\star \notin \{\hat{\tau}_i \}$. Define
			\begin{equation}\label{bethe:bpmsg:dot}
			\dot{z}[\dot{\tau}] := \sum_{\bx\in\{0,1\} }
			\prod_{i=1}^{d-1} \hat{\mm}[\hat{\tau}_i](\bx), \quad 
			\dot{\mm}[\dot{\tau}](\bx) :=
			\frac{1}{\dot{z}[\dot{\tau}]} \prod_{i=1}^{d-1} \hat{\mm}[\hat{\tau}_i](\bx).
			\end{equation}
			Note that $\dot{z}[\dot{\tau}]$ and $\dot{\mm}[\dot{\tau}](\bx)$ are well-defined, since $(\hat{\tau}_1,\ldots, \hat{\tau}_{d-1})$ can be recoved from $\dot{\tau}$ up to permutation.

			\item Let $\hat{\tau} = \hat{T} ( \dot{\tau}_1,\ldots,\dot{\tau}_{k-1})$, with $\star \notin \{\dot{\tau}_i \}$. Define
			\begin{equation}\label{bethe:bpmsg:hat}
			\hat{z}[\hat{\tau}] := 2-\sum_{\bx\in\{0,1\}} \prod_{i=1}^{k-1} \dot{\mm}[\dot{\tau}_i](\bx), \quad 
			\hat{\mm}[\hat{\tau}](\bx) :=
			\frac{1}{\hat{z}[\hat{\tau}]}
			\left\{1- \prod_{i=1}^{k-1} \dot{\mm}[\dot{\tau}_i](\bx) \right\} .
			\end{equation}
			Similarly as above, $\hat{z}[\hat{\tau}]$ and $	\hat{\mm}[\hat{\tau}]$ are well-defined.
		\end{itemize}
		Moreover, observe that inductively, $\dot{\mm}[\dot{\tau}], \hat{\mm}[\hat{\tau}] $ are not Dirac measures unless $\dot{\tau}, \hat{\tau}\in \{0,1\}$. 
	\end{defn} 
	It turns out that $\dot{\mm}[\star], \hat{\mm}[\star]$ can be arbitrary measures for our purpose, and hence we assume that they are uniform measures on $\{0,1\}$.
	
	The equations \eqref{bethe:bpmsg:dot} and \eqref{bethe:bpmsg:hat} are known as \textit{belief propagation} equations. We refer the detailed explanation to Section 2 of \cite{ssz22} where the same notions are introduced, or to Chapter 14 of \cite{mm09} for more fundamental background. From these quantities, we define the following local weights.
	\begin{equation}\label{eq:def:phi}
	\begin{split}
	&\bar{\varphi} (\dot{\tau}, \hat{\tau}) := \bigg\{ \sum_{\bx\in \{0,1\}} \dot{\mm}[\dot{\tau}](\bx) \hat{\mm}[\hat{\tau}](\bx) \bigg\}^{-1}; \quad \hat{\varphi}^{\textnormal{lit}} (\dot{\tau}_1,\ldots, \dot{\tau}_k):= 1-\sum_{\bx \in \{0,1\}} \prod_{i=1}^k \dot{\mm}[\dot{\tau}_i](\bx);\\
	&\dot{\varphi} (\hat{\tau}_1,\ldots,\hat{\tau}_d):=
	\sum_{\bx\in\{0,1\} }\prod_{i=1}^d \hat{\mm}[\hat{\tau}_i](\bx).
	\end{split}
	\end{equation}
	\begin{lemma}[\cite{ssz22}, Lemma 2.9 and Corollary 2.10; \cite{mm09}, Ch. 14]\label{lem:size:msg and trees}
		Let $\ux$ be a frozen configuration on $\GG=(V,F,E,\uL)$ without any free cycles, and $\utau$ be the corresponding message configuration. Then, we have that
		\begin{equation}\label{eq:free:tree:weight:lit}
		w_{\ttt}= \prod_{v\in V(\ttt)} \left\{ \dot{\varphi}(\hat{\utau}_{\delta v}) \prod_{e\in \delta v} \bar{\varphi}(\tau_e) \right\} \prod_{a\in F(\ttt)} \hat{\varphi}^{\textnormal{lit}}\big( (\dot{\utau} \oplus \uL)_{\delta a} \big).
		\end{equation}
		Furthermore, we have
		\begin{equation*}
		\textsf{size}(\ux,\GG)= \prod_{v\in V} \dot{\varphi} (\hat{\utau}_{\delta v}) \prod_{a\in F} \hat{\varphi}^{\textnormal{lit}}\big((\dot{\utau}\oplus \uL)_{\delta a} \big) \prod_{e \in E} \bar{\varphi} (\tau_e).
		\end{equation*} 
	\end{lemma}

	\subsection{Colorings}\label{subsubsec:model:col}
	In this subsection, we introduce the \textit{coloring configuration}, which is a simplification of the message configuration. Recall the definition of $\MMM=\dot{\MMM}\times \hat{\MMM}, $ and let $\{ \fF \} \subset \MMM$ be defined by $\{\fF\}:= \{\tau \in \MMM: \, \dot{\tau} \notin \{ 0,1,\star\}, \hat{\tau}\notin \{ 0,1,\star\} \}$. Define $\Omega :=  \{\rr_0, \rr_1, \bb_0, \bb_1\} \cup \{\fF \}$ and let $\textsf{S}: \MMM \to \Omega$ be the projections given by
	
	\begin{equation}\label{eq:simplify:coloring}
	\textsf{S}(\tau) := 
	\begin{cases}
	\rr_0 & \hat{\tau}=0;\\
	\rr_1 & \hat{\tau}=1;\\
	\bb_0 & \hat{\tau } \neq 0, \, \dot{\tau}=0;\\
	\bb_1 & \hat{\tau} \neq 1, \, \dot{\tau}=1;\\
	\tau & \textnormal{otherwise, i.e., } \tau \in \{ \fF \},
	\end{cases}
	\end{equation}
	
    The coloring model if the projection of message configurations under the projection $\textsf{S}$. For simplicity, we abbreviate $\{\rr \}= \{\rr_0, \rr_1 \}$ and $\{\bb \} = \{\bb_0, \bb_1 \}$, and  define the Boolean addition  as $\bb_\bx \oplus \tL := \bb_{\bx \oplus \tL}$, and similarly for $\rr_\bx$. Also, for $\sigma \in \{ \rr,\bb,\fs\}$, we set $\dot{\sigma} :=  \sigma=:\hat{\sigma}$.
	\begin{defn}[Colorings]\label{def:model:col}
		For $\sig \in \Omega ^d$, let
		\begin{equation*}
		\dot{I}(\sig) : =
		\begin{cases}
		1 & \rr_0 \in \{\sigma_i\} \subseteq \{\rr_0, \bb_0 \};\\
		1& \rr_1 \in \{\sigma_i\} \subseteq \{\rr_1,\bb_1 \};\\
		1 & \{\sigma_i \} \subseteq \{ {\fF} \}, \textnormal{ and } \dot{\sigma}_i = \dot{T}\big( (\hat{\sigma}_j)_{j\neq i} \big), \ \forall i;\\
		0 & \textnormal{otherwise}.
		\end{cases}
		\end{equation*}
		Also, define $	\hat{I}^{\textnormal{lit}}: \Omega^k \to \mathbb{R}$ to be
		\begin{equation*}
		\begin{split}
		\hat{I}^{\textnormal{lit}}(\sig)&:=
		\begin{cases}
		1 & \exists i:\, \sigma_i = \rr_0 \textnormal{ and } \{\sigma_j \}_{j\neq i} = \{\bb_1 \};\\
		1 & \exists i:\, \sigma_i = \rr_1 \textnormal{ and } \{\sigma_j \}_{j\neq i} = \{\bb_0 \};\\
		1 & \{\bb \} \subseteq \{\sigma_i \} \subseteq \{\bb\} \cup \{\sigma \in \{\fF \}: \,\hat{\sigma}=\fs \};\\
		1 & \{\sigma_i \} \subseteq \{ \bb_0, {\fF} \}, \, |\{i: \sigma_i\in  \{{\fF} \}\} | \ge 2 , \textnormal{ and } \hat{\sigma}_i = \hat{T}((\dot{\sigma}_j)_{j\neq i}; 0), \ \forall i \textnormal{ s.t. } \sigma_i \neq \bb_0;\\
		1 & \{\sigma_i \} \subseteq \{ \bb_1, {\fF} \}, \, |\{i: \sigma_i\in  \{{\fF} \}\}| \ge 2 , \textnormal{ and } \hat{\sigma}_i = \hat{T}((\dot{\sigma}_j)_{j\neq i}; 0), \ \forall i \textnormal{ s.t. } \sigma_i \neq \bb_1;\\
		0 & \textnormal{otherwise}.
		\end{cases}
		\end{split}
		\end{equation*}
		On a \textsc{nae-sat} instance $\GG = (V,F,E,\uL)$, $\sig\in \Omega^E$ is a (valid) \textbf{coloring} if $\dot{I}(\sig_{\delta v})=\hat{I}^{\textnormal{lit}}((\sig\oplus\uL)_{\delta a}) =1 $ for all $v\in V, a\in F$. 
	\end{defn}
	
	Given \textsc{nae-sat} instance $\GG$, it was shown in Lemma 2.12 of \cite{ssz22} that there is a bijection
	\begin{equation}\label{eq:msg col 1to1}
	\begin{Bmatrix}
	\textnormal{message configurations}\\
	\utau \in \MMM^E
	\end{Bmatrix}
	\ \longleftrightarrow \
	\begin{Bmatrix}
	\textnormal{colorings} \\
	\sig \in \Omega^E
	\end{Bmatrix}
	\end{equation}
	The weight elements for coloring, denoted by $\dot{\Phi}, \hat{\Phi}^{\textnormal{lit}}, \bar{\Phi}$, are defined as follows. For $\sig \in \Omega^d,$ let
	\begin{equation*}
	\begin{split}
	\dot{\Phi}(\sig) := 
	\begin{cases}
	\dot{\varphi}(\hat{\sig}) & \dot{I}(\sig) =1 \textnormal{ and } \{\sigma_i \} \subseteq \{\fF \};\\
	1 & \dot{I}(\sig) =1 \textnormal{ and } \{\sigma_i \}\subseteq \{\bb, \rr \};\\
	0 & \textnormal{otherwise, i.e., } \dot{I}(\sig)=0.
	\end{cases}
	\end{split}
	\end{equation*}
	For $\sig \in \Omega^k$, let
	\begin{equation*}
	\hat{\Phi}^{\textnormal{lit}}(\sig) :=
	\begin{cases}
	\hat{\varphi}^\lit((\dot{\tau}(\sigma_i))_i)
	& \hat{I}^{\textnormal{lit}} (\sig) = 1 \textnormal{ and } \{\sigma_i \} \cap \{\rr \} = \emptyset;\\
	1 & \hat{I}^{\textnormal{lit}}(\sig) = 1 \textnormal{ and } \{\sigma_i \} \cap \{\rr \} \neq \emptyset;\\
	0 & \textnormal{otherwise, i.e., } \hat{I}^{\textnormal{lit}}(\sig)=0.
	\end{cases}
	\end{equation*}
	(If $\sigma \notin \{\rr \}, $ then $\dot{\tau}(\sigma_i)$ is well-defined.)
	
	Lastly, let
	\begin{equation*}
	\bar{\Phi}(\sigma) := 
	\begin{cases}
	\bar{\varphi} (\sigma) & \sigma \in \{\fF \};\\
	1 & \sigma \in \{\rr, \bb \}.
	\end{cases}
	\end{equation*}
	Note that if $\hat{\sigma}=\fs$, then $\bar{\varphi}(\dot{\sigma},\hat{\sigma})=2$ for any $\dot{\sigma}$.
	The rest of the details explaining the compatibility of $\varphi$ and $\Phi$ can be found in \cite{ssz22}, Section 2.4.
	Then, the formula for the cluster size we have  seen in Lemma \ref{lem:size:msg and trees} works the same for the coloring configuration.
	
	\begin{lemma}[\cite{ssz22}, Lemma 2.13]\label{lem:model:size:col}
		Let 	$\ux \in \{0,1,\ff \}^V$ be a frozen configuration on $\GG=(V,F,E,\uL)$, and let $\sig \in \Omega^E$ be the corresponding coloring. Define
		\begin{equation*}
		w_\GG^{\textnormal{lit}}(\sig):= \prod_{v\in V}\dot{\Phi}(\sig_{\delta v}) \prod_{a\in F} \hat{\Phi}^{\textnormal{lit}} ((\sig \oplus \uL)_{\delta a}) \prod_{e\in E} \bar{\Phi}(\sigma_e).
		\end{equation*}
		Then, we have $\textsf{size}(\ux;\GG) = w_\GG^{\textnormal{lit}}(\sig)$.\
	\end{lemma}
	Among the valid frozen configurations, we can ignore the contribution from the configurations with too many free or red colors, as observed in the following lemma.
		\begin{lemma}[\cite{dss16} Proposition 2.2 ,\cite{ssz22} Lemma 3.3]
		For a frozen configuration $\ux \in \{0,1,\ff \}^{V}$, let $\rr(\ux)$ count the number of forcing edges and $\ff(\ux)$ count the number of free variables.
		There exists an absolute constant $c>0$ such that for $k\geq k_0$, $\alpha \in [\alpha_{\textsf{lbd}}, \alpha_{\textsf{ubd}}]$, and $\lambda \in(0,1]$, 
		\begin{equation*}
		\sum_{\ux \in \{0,1,\ff\}^V} \E \left[	\textsf{size}(\ux;\GG)^\lambda\right] \mathds{1}\left\{ \frac{\rr(\ux)}{nd} \vee \frac{\ff(\ux)}{n}> \frac{7}{2^k} \right\}  \le e^{-cn},
		\end{equation*}
		where $\textsf{size}(\ux;\GG)$ is the number of \textsc{nae-sat} solutions $\bux \in \{0,1\}^{V}$ which extends $\ux\in \{0,1,\ff\}^{V}$.
	\end{lemma}
	\begin{defn}[Truncated colorings]
			Let $1\leq L< \infty$,  $\ux $ be a frozen configuration on $\GG$ without free cycles and $\sig\in \Omega^E$ be the coloring corresponding to $\ux$. We say $\sig$ is a (valid) $L$-\textbf{truncated coloring} if $|V(\ttt)| \le L$ for all $\ttt \in \mathscr{F}(\ux, \GGG)$. For an equivalent definition, let $|\sigma|:=v(\dot{\sigma})+v(\hat{\sigma})-1$ for $\sigma \in \{\fF\}$, where $v(\dot{\sigma})$ (resp. $v(\hat{\sigma})$) denotes the number of variables in $\dot{\sigma}$ (resp. $\hat{\sigma}$). Define $\Omega_L := \{\rr,\bb \}\cup\{\fF \}_L$, where $\{\fF \}_L$ be the collection of $\sigma\in \{\fF \}$ such that $|\sigma| \le L$. Then,  $\sig$ is a (valid) $L$-truncated coloring if $\sig \in \Omega_L^E$.
			
			To clarify the names, we often call the original coloring $\sig\in \Omega^E$ the \textbf{untruncated coloring}.
		\end{defn}
	
	\subsection{Averaging over the literals}\label{subsubsec:model:avglit:1stmo}
	We now consider $\mathbb{E}^{\textnormal{lit}}[w_\GG^{\textnormal{lit}}(\sig) ]$ for a given coloring $\sig \in \Omega^E,$ where $\mathbb{E}^{\textnormal{lit}}$ denotes the expectation over the literals $\uL \sim \textnormal{Unif} [\{0,1\}^E]$. From Lemma \ref{lem:model:size:col}, we can write
	\begin{equation}\label{def:w}
	w_{\GG}(\sig)^\la:=\E^{\textnormal{lit}} [ w_\GG^{\textnormal{lit}}(\sig)^\lambda] = \prod_{v\in V} \dot{\Phi}(\sig_{\delta v})^\lambda \prod_{a\in F} \E^{\textnormal{lit}} \hat{\Phi}^{\textnormal{lit}}((\sig\oplus \uL)_{\delta a})^\lambda \prod_{e\in E} \bar{\Phi}(\sigma_e)^\lambda.
	\end{equation}
	Define $\hat{\Phi}(\sig_{\delta a})^\lambda := \E^{\textnormal{lit}}[ \hat{\Phi}^{\textnormal{lit}}((\sig\oplus\uL)_{\delta a})^\lambda]. $ To give a more explicit expression of this formula, we recall a property of $\hat{\Phi}^{\lit}$ from \cite{ssz22}, Lemma 2.17:
	\begin{lemma}[\cite{ssz22}, Lemma 2.17]\label{lem:decompose:Phi:hat}
		$\hat{\Phi}^{\lit}$ can be factorized as $\hat{\Phi}^{\lit}(\sig\oplus\uL) = \hat{I}^{\lit}(\sigma \oplus \uL) \hat{\Phi}^{\textnormal{m}}(\sig)$ for 
		\begin{equation}\label{eq:def:Phi:hat:max}
		\hat{\Phi}^{\textnormal{m}}(\sig) := \max\big\{\hat{\Phi}^{\lit}(\sig\oplus\uL): \uL \in \{0,1\}^k \big\}=
		\begin{cases}
		1 & \sig \in \{\rr,\bb\}^{k},\\
		\frac{\hat{z}[\hat{\sigma}_j]}{\bar{\varphi}(\sigma_j)} &\sig \in \Omega^{k}\textnormal{ with } \sigma_j \in \{\ff\}. 
		\end{cases}
		\end{equation}
	\end{lemma}
	As a consequence, we can write $\hat{\Phi}(\sig)^\lambda = \hat{\Phi}^{\textnormal{m}}(\sig)^\lambda \hat{v}(\sig)$, where
	\begin{equation}\label{eq:def:vhat:basic}
	\hat{v}(\sig) := \E^{\lit} [ \hat{I}^{\lit}(\sig\oplus\uL)]. 
	\end{equation}

	\subsection{Optimal boundary profile}
	We define \textsc{bp} functional, which was introduced in Section 5 of \cite{ssz22}. For probability measures $\qdot,\qhat \in \PPP(\Omega_L)$, where $L<\infty$, let
	\begin{equation}\label{eq:pre:BP:1stmo}
	\begin{split}
	    &[\dot{\mathbf{B}}_{1,\lambda}(\qhat)](\sigma)\cong \bar{\Phi}(\sigma)^\lambda \sum_{\sig \in \Omega_L^{d}}\one\{\sigma_1=\sigma\}\dot{\Phi}(\sig)^{\lambda}\prod_{i=2}^{d}\qhat(\sigma_i)\\
	    &[\hat{\mathbf{B}}_{1,\lambda}(\qdot)](\sigma)\cong \bar{\Phi}(\sigma)^\lambda \sum_{\sig \in \Omega_L^{k}}\one\{\sigma_1=\sigma\}\hat{\Phi}(\sig)^{\lambda}\prod_{i=2}^{d}\qdot(\sigma_i),
	\end{split}
	\end{equation}
	where $\sigma \in \Omega_L$ and $\cong$ denotes equality up to normlization, so that the output is a probability measure. We denote by $\dot{\mathscr{Z}}\equiv\dot{\mathscr{Z}}_{\hat{q}} ,\hat{\mathscr{Z}}\equiv \hat{\mathscr{Z}}_{\dot{q}}$ the normalizing constants for \eqref{eq:pre:BP:1stmo}. Now, restrict the domain to the probability measures with \textit{one-sided} dependence, i.e. satisfying $\qdot(\sigma)=\dot{f}(\dot{\sigma})$ and $\qhat(\sigma)=\hat{f}(\hat{\sigma})$ for some $\dot{f} :\dot{\Omega}_L\to\R_{\geq 0}$ and $\hat{f} :\hat{\Omega}_L\to\R_{\geq 0}$. It can be checked that $\dot{\mathbf{B}}_{1,\lambda}, \hat{\mathbf{B}}_{1,\lambda}$ preserve the one-sided property, inducing
	\begin{equation*}
	    \dot{\textnormal{BP}}_{\lambda,L}:\PPP(\hat{\Omega}_{L}) \rightarrow \PPP(\dot{\Omega}_{L}),\quad\hat{\textnormal{BP}}_{\lambda,L}:\PPP(\dot{\Omega}_L) \rightarrow \PPP(\hat{\Omega }_L).
	\end{equation*}
    More precisely, for $\hat{q}\in \PPP(\hat{\Omega}_L)$ and $\dot{q} \in \PPP(\dot{\Omega}_L)$, define the probability measures $\dot{\textnormal{BP}}_{\la,L}(\hat{q})\in \PPP(\dot{\Omega}_L)$ and $\hat{\textnormal{BP}}_{\la,L}(\dot{q})\in \PPP(\hat{\Omega}_L)$ as follows. For $\dot{\sigma}\in \dot{\Omega}_L$ and $\hat{\sigma}\in \hat{\Omega}_L$, let
    \begin{equation}\label{eq:def:BP}
    \begin{split}
    &[\dot{\textnormal{BP}}_{\la,L}(\hat{q})](\dot{\sigma})=\big(\dot{\ZZZ}_{\hat{q}}\big)^{-1} \cdot \bar{\Phi}(\dot{\sigma}, \hat{\sigma}^\prime)^\lambda \sum_{\sig \in \Omega_L^{d}}\one\{\sigma_1=(\dot{\sigma},\hat{\sigma}^\prime)\}\dot{\Phi}(\sig)^{\lambda}\prod_{i=2}^{d}\hat{q}(\hat{\sigma}_i)\,,\\
    &[\hat{\textnormal{BP}}_{\la,L}(\dot{q})](\hat{\sigma})=\big(\hat{\ZZZ}_{\dot{q}}\big)^{-1} \cdot \bar{\Phi}(\dot{\sigma}^\prime, \hat{\sigma})^\lambda \sum_{\sig \in \Omega_L^{k}}\one\{\sigma_1=(\dot{\sigma}^\prime, \hat{\sigma})\}\dot{\Phi}(\sig)^{\lambda}\prod_{i=2}^{k}\dot{q}(\dot{\sigma}_i)\,,
    \end{split}
    \end{equation}
    where $\hat{\sigma}^\prime \in \hat{\Omega}_L$ and $\dot{\sigma}^\prime \in \dot{\Omega}_L$ are arbitrary with the only exception that when $\dot{\sigma}\in \{\rr,\bb\}$ (resp. $\hat{\sigma}\in \{\rr,\bb\}$), then we take $\hat{\sigma}^\prime = \dot{\sigma}$ (resp. $\dot{\sigma}^\prime = \hat{\sigma}$) so that the \textsc{rhs} above is non-zero. From the definition of $\dot{\Phi},\hat{\Phi}$, and $\bar{\Phi}$, it can be checked that the choices of $\hat{\sigma}^\prime \in \hat{\Omega}_L$ and $\dot{\sigma}^\prime \in \dot{\Omega}_L$ do not affect the values of the \textsc{rhs} above. The normalizing constants $\dot{\ZZZ}_{\hat{q}}$ and $\hat{\ZZZ}_{\dot{q}}$ are given by
    \begin{equation}\label{eq:BP:normalization}
 \begin{split}
 &\dot{\ZZZ}_{\hat{q}}\equiv\sum_{\dot{\sigma} \in \dot{\Omega}_L} \bar{\Phi}(\dot{\sigma}, \hat{\sigma}^\prime)^\lambda \sum_{\sig \in \Omega_L^{d}}\one\{\sigma_1=(\dot{\sigma},\hat{\sigma}^\prime)\}\dot{\Phi}(\sig)^{\lambda}\prod_{i=2}^{d}\hat{q}(\hat{\sigma}_i)\,,\\
 &\hat{\ZZZ}_{\dot{q}}\equiv \sum_{\hat{\sigma} \in \hat{\Omega}_L}\bar{\Phi}(\dot{\sigma}^\prime, \hat{\sigma})^\lambda \sum_{\sig \in \Omega_L^{k}}\one\{\sigma_1=(\dot{\sigma}^\prime, \hat{\sigma})\}\dot{\Phi}(\sig)^{\lambda}\prod_{i=2}^{k}\dot{q}(\dot{\sigma}_i)\,.
 \end{split}
 \end{equation}
Here, $\hat{\sigma}^\prime \in \hat{\Omega}_L$ and $\dot{\sigma}^\prime \in \dot{\Omega}_L$ are again arbitrary. We then define the \textit{Belief Propagation functional} by $\textnormal{BP}_{\lambda,L}:= \dot{\textnormal{BP}}_{\lambda,L}\circ \hat{\textnormal{BP}}_{\lambda,L}$. The untruncated BP map, which we denote by $\textnormal{BP}_{\lambda}:\PPP(\dot{\Omega}) \to \PPP(\dot{\Omega})$, is analogously defined, where we replace $\dot{\Omega}_L$(resp. $\hat{\Omega}_L$) with $\dot{\Omega}$(resp. $\hat{\Omega}$). Let $\mathbf{\Gamma}_C$ be the set of $\dot{q} \in \PPP(\dot{\Omega})$ such that 
	\begin{equation}\label{eq:def:bp:contract:set:1stmo}
	    \dot{q}(\dot{\sigma})=\dot{q}(\dot{\sigma}\oplus 1)\quad\text{for}\quad\dot{\sigma} \in \dot{\Omega },\quad\text{and}\quad \frac{\dot{q}(\rr)+2^k\dot{q}(\ff)}{C}\leq \dot{q}(\bb) \leq \frac{\dot{q}(\rr)}{1-C2^{-k}}.
	\end{equation}
	\begin{prop}[Proposition 5.5 item a,b of \cite{ssz22}]
	\label{prop:BP:contraction}
	For $\lambda \in [0,1]$, the following holds:
	\begin{enumerate}
	    \item There exists a large enough universal constant $C$ such that the map $\textnormal{BP}\equiv\textnormal{BP}_{\lambda,L}$ has a unique fixed point $\dot{q}^\star_{\lambda,L}\in \mathbf{\Gamma}_C$. Moreover, if $\dotq \in \mathbf{\Gamma}_C$, $\textnormal{BP}\dotq \in \mathbf{\Gamma}_C$ holds with
	    \begin{equation}\label{eq:BPcontraction:1stmo}
	        ||\textnormal{BP}\dotq-\dotq^\star_{\lambda,L}||_1\lesssim k^2 2^{-k}||\dotq-\dotq^\star_{\lambda,L}||_1.
	    \end{equation}
	    The same holds for the untruncated BP, i.e. $\textnormal{BP}_{\la}$, with fixed point $\dot{q}^\star_{\lambda}\in \Gamma_C$. $\dot{q}^\star_{\la,L}$ for large enough $L$ and $\dot{q}^\star_{\la}$ have full support in their domains.
	    \item In the limit $L \to \infty$, $||\dot{q}^\star_{\lambda,L}-\dot{q}^\star_{\lambda}||_1 \to 0$.
	\end{enumerate}
	\end{prop}
	For $\dot{q} \in \PPP(\dot{\Omega})$, denote $\hat{q}\equiv \hat{\textnormal{BP}}\dot{q}$, and define $H_{\dot{q}}=(\dot{H}_{\dot{q}},\hat{H}_{\dot{q}}, \bar{H}_{\dot{q}})\in \bDelta$ by
	\begin{equation}\label{eq:H:q:1stmo}
	    \dot{H}_{\dot{q}}(\sig)=\frac{\dot{\Phi}(\sig)^{\lambda}}{\dot{\mathfrak{Z}}}\prod_{i=1}^{d}\hat{q}(\hat{\sigma}_i),\quad \hat{H}_{\dot{q}}(\sig)=\frac{\hat{\Phi}(\sig)^{\lambda}}{\hat{\mathfrak{Z}}}\prod_{i=1}^{k}\dot{q}(\dot{\sigma}_i),\quad \bar{H}_{\dot{q}}(\sigma)=\frac{\bar{\Phi}(\sigma)^{-\lambda}}{\bar{\mathfrak{Z}}}\dot{q}(\dot{\sigma})\hat{q}(\hat{\sigma}),
	\end{equation}
	where $\dot{\mathfrak{Z}}\equiv \dot{\mathfrak{Z}}_{\dot{q}},\hat{\mathfrak{Z}}\equiv\hat{\mathfrak{Z}}_{\dot{q}}$ and $\bar{\mathfrak{Z}}\equiv \bar{\mathfrak{Z}}_{\dot{q}}$ are normalizing constants.
	\begin{defn}[Definition 5.6 of \cite{ssz22}]\label{def:opt:coloring:profile}
	    The \textit{optimal coloring profile} for the untruncated model is the tuple $H^\star_{\lambda}=(\dot{H}^\star_{\lambda},\hat{H}^\star_{\lambda},\bar{H}^\star_{\lambda})$, defined respectively by $ H^\star_{\lambda,L}:= H_{\dot{q}^\star_{\lambda,L}}$ and $H^\star_{\lambda}:=H_{\dot{q}^\star_{\lambda}}$.
	\end{defn}
	\begin{defn}[optimal boundary profile, free tree profile and weight]\label{def:opt:bdry:1stmo}
	The \textit{optimal boundary profile}, the \textit{optimal free tree profile}, and the \textit{optimal weight} are defined by the following.
	\begin{itemize}
	    \item  The optimal boundary profile for the truncated model is the tuple $B^\star_{\lambda} \equiv(\dot{B}^\star_{\lambda},\hat{B}^\star_{\lambda},\bar{B}^\star_{\lambda})$, defined by restricting the optimal coloring profile to $(\dot{\partial}^\bullet)^d, (\hat{\partial}^\bullet)^k, \hat{\partial}^\bullet$:
	 \begin{equation}\label{eq:def:optimal:bdry}
	 \begin{split}
	     &\dot{B}^\star_{\lambda}(\sig):=  \dot{H}^\star_{\lambda}(\sig)\quad\textnormal{for}\quad \sig \in (\dot{\partial}^\bullet)^d\\
	     &\hat{B}^\star_{\lambda}(\sig) := \sum_{\utau \in \Omega^k, \utau_{\fs}=\sig} \hat{H}^\star_{\lambda}(\utau)\quad\textnormal{for}\quad \sig \in (\hat{\partial}^\bullet)^k
	     \\
	     &\bar{B}^\star_{\lambda}(\sigma):= \sum_{\tau \in\Omega, \tau_{\fs}= \sigma} \bar{H}^\star_{\lambda}(\tau)\quad\textnormal{for}\quad \sigma \in \hat{\partial}^\bullet,
	 \end{split}
	 \end{equation}
	 where $\tau_{\fs}$ is defined by the simplified coloring of $\tau\in \Omega$, where $\tau_{\fs}:=\tau$, if $\hat{\tau}\neq \fs$, and $\tau_{\fs}:=\fs$, if $\hat{\tau}=\fs$. $\utau_{\fs}$ is the coordinate-wise simplified coloring of $\utau$.
	 
	 \item The (normalized) optimal free tree profile $(p_{\ttt, \lambda}^\star)_{\ttt\in \FFF_{\tr}}$ for the truncated model is defined as follows. Recall the normalizing constants $\dot{\mathscr{Z}}^\star \equiv \dot{\mathscr{Z}}_{\hat{q}^\star_{\lambda}}, \hat{\mathscr{Z}}\equiv \hat{\mathscr{Z}}_{\dot{q}^\star_{\lambda}}$ for the BP map in \eqref{eq:pre:BP:1stmo}, where $\hat{q}^\star_{\lambda}\equiv \hat{\textnormal{BP}}\dot{q}^\star_{\lambda}$, and $\bar{\mathfrak{Z}}^\star \equiv\bar{\mathfrak{Z}}_{\dot{q}^\star_{\lambda}}$ in \eqref{eq:H:q:1stmo}. Writing $\dot{q}^\star=\dot{q}^\star_{\lambda}$ and $\hat{q}^\star=\hat{q}^\star_{\lambda}$, define
	 \begin{equation}\label{eq:optimal:tree:1stmo}
	     p_{\ttt, \lambda}^\star := \frac{J_{\ttt} w_{\ttt}^\lambda}{\bar{\mathfrak{Z}}^\star (\dot{\ZZZ}^\star)^{v(\ttt)}(\hat{\ZZZ}^\star)^{f(\ttt)}}\dot{q}^\star(\bb_0)^{\eta_{\ttt}(\bb_0)+\eta_{\ttt}(\bb_1)}(2^{-\lambda}\hat{q}^\star(\fs))^{\eta_{\ttt}(\fs)},
	 \end{equation}
	 for $\ttt\in \FFF$ with $v(\ttt)\leq L$.
    \end{itemize}
	\end{defn}
\section{Equivalent descriptions of the weak local limit}
\label{sec:appendix:B}
In this section, we prove the following proposition. Recall the definition of $\mathcal{P}_{\star}^{t}$ in Section \ref{s:lwl} and the definition of $\nu^\star_{t}$ in Definition \ref{def:optimal:t}. Also, recall the definition of $p_{\uz_t}(\sig_t,\uL_t)$ in Remark \ref{rmk:p}.
\begin{prop}\label{prop:equiv}
For any $\uz_t\in \{0,1\}^{V(\TTT_{d,k,t})}$ and $\uL_t\in \{0,1\}^{E_{\sf in}(\TTT_{d,k,t})}$, we have that
\begin{equation*}
    \PP_{\star}^{t}(\uz_t,\uL_t)=\sum_{\sig_{t}\in \CCC_t}\nu_{t}^\star(\sig_t,\uL_t)p_{\uz_t}(\sig_t,\uL_t)\,.
\end{equation*}
\end{prop}
To prove Proposition \ref{prop:equiv}, we first reduce the component coloring model to the coloring model. To this end, we define the analog of $\nu^\star_t$ for the coloring model. Throughout, we use the supercript $^{\circ}$ to distinguish a coloring $\sigma^{\circ}\in \Omega$ from a component coloring $\sigma \in \CCC$. Moreover, we denote the optimal coloring profile for the untruncated model with $\la=\la^\star$ by $H^\star\equiv H^\star_{\la^\star}$. 
\begin{defn}\label{def:optiaml:t:ssz}
Consider the broadcast process with channel $H^\star\equiv (\dot{H}^\star, \hat{H}^\star, \bar{H}^\star)$ analogous to Definition \ref{def:broadcast}. That is, given a $\TTT_{d,k}$ with root $\rho$, the spins around the root $\bsig^{\circ}_{\delta \rho} \in \Omega^d$ is drawn from $\dot{H}^\star$. Then, it is propagated along the variables and clauses as follows. If an edge $e\in E(\TTT_{d,k})$ has children edges $\delta a(e)\setminus e$, then for $\utau^{\circ}=(\tau_1^{\circ},\ldots, \tau_k^{\circ})\in \Omega^k$ and $\tau^{\circ}\in \Omega$,  
        \begin{equation*}
            \P\big(\bsig^{\circ}_{\delta a(e)}= \utau^{\circ} \,\big| \bsigma_e^{\circ} = \tau^{\circ}\big)= \frac{1}{k}\frac{\hat{H}^\star(\utau^{\circ})\sum_{i=1}^{k}\one(\tau_i^\circ =\tau^{\circ})}{\bar{H}^\star(\tau^{\circ})}\,.
        \end{equation*}
    If an edge $e\in E(\TTT_{d,k})$ has children edges $\delta v(e)\setminus e$, then for $\utau^{\circ}=(\tau^{\circ}_1,\ldots, \tau^{\circ}_d)\in \Omega^d$ and $\tau^{\circ}\in \Omega$,
    \begin{equation*}
        \P\big(\bsig^{\circ}_{\delta v(e)}= \utau^{\circ} \,\big| \bsigma_e^{\circ} = \tau^{\circ}\big)= \frac{1}{d}\frac{\dot{H}^\star(\utau^{\circ})\sum_{i=1}^{d}\one(\tau_i^\circ =\tau^{\circ})}{\bar{H}^\star(\tau^{\circ})}\,.
    \end{equation*}
    Then, conditional on $(\bsigma_e^{\circ})_{e\in E(\TTT_{d,k})}$, draw $\buL_{\delta a}\in \{0,1\}^k$ for each clause $a\in F(\TTT_{d,k})$ independently and uniformly at random among $\uL$ which satisfy $\hat{I}^{\lit}(\sig^{\circ}_{\delta a}\oplus \uL)=1$. Define $\nu^{\star,\circ}_t\equiv \nu^{\star,\circ}_t[\alpha,k]$ by the law of $(\bsig_t^{\circ},\buL_t)\equiv \big((\bsigma_e^{\circ})_{e\in E(\TTT_{d,k,t})},(\bL_e)_{e\in E_{\sf in}(\TTT_{d,k,t})}\big)$ considered up to automorphisms analogous to Definition \ref{def:t:nbd:empirical}.
\end{defn}

Observe that for a $t$-coloring $(\sig_t,\uL_t)\in \Omega_t$ such that $\sig_t=(\sigma_e)_{e\in E(\TTT_{d,k,t})}$ does not contain a cyclic free component, i.e. if $\sigma_e=(\fff,e^\prime)$, then $\fff\in \FFF_{\tr}$, we can reconstruct the corresponding coloring configuration $\sig^{\circ}_t\equiv( \sigma_e^{\circ})_{e\in E(\TTT_{d,k,t})}$ by the following procedure.
\begin{itemize}
    \item For $e\in E(\TTT_{d,k,t})$, if $\sigma_e\in \{\rr,\bb\}$, then set $\sigma^{\circ}_e=\sigma_e$.
    \item Suppose that $\sigma_e=(\ttt, e^\prime)$ for some $\ttt\in \FFF_{\tr}$ and $e^\prime \in E(\ttt)$. Then, since the free tree $\ttt$ contains the information of its boundary spins and literal information, we can apply the map $\dot{T}(\cdot)$ and $\hat{T}(\cdot)$ (cf. Definition \ref{def:msg config}) from the boundary of $\ttt$ to recover the coloring $\sigma_{e^\prime}^{\circ}(\ttt)=\big(\dot{\sigma}_{e^\prime}^{\circ}(\ttt), \hat{\sigma}_{e^\prime}^{\circ}(\ttt)\big)\in \Omega$ at the edge $e^\prime$. Then, define the map $\psi:\{(\ttt,e^\prime ):\ttt\in \FFF_{\tr},e^\prime\in E(\ttt)\}\to \Omega$ by $\psi(\ttt,e^\prime)=\sigma_{e^\prime}^{\circ}(\ttt)$, and set $\sigma^{\circ}_e=\psi(\sigma_e)$.
    \item For all $e\in E(\TTT_{d,k,t})$ such that $\sigma_e=\fs$, set $\hat{\sigma}^{\circ}_e=\fs$. Note that after this step, $\hat{\sigma}_e^{\circ}$ have been determined for $e\in E(\TTT_{d,k,t})$.
    \item Finally, for $e\in E(\TTT_{d,k,t})$ such that $\sigma_e=\fs$, set $\dot{\sigma}^{\circ}_e:=\dot{T}(\hat{\sig}^{\circ}_{\delta v(e)\setminus e})$ (cf. Definition \ref{def:msg config}), which is well-defined since if $\sig_t$ is valid per Definition \ref{def:t:nbd:empirical}, then $\sigma_e=\fs$ implies that there exists no edge $e^\prime \in \delta v(e)$ such that $\sigma_e^\prime \in \{\rr,\bb\}$.
\end{itemize} 
We denote $\Psi(\sig_t,\uL_t):=(\sig^{\circ}_t,\uL_t)$ by the resulting coloring configuration. Then, the lemma below follows easily from the relationship between $H^\star$ and $p_{\ttt}^\star$ established in Lemma B.1 in \cite{NSS}.
\begin{lemma}\label{lem:equiv}
If $(\bsig_t,\buL_t)\sim \nu_t^\star$, then $\Psi(\bsig_t,\buL_t)\sim \nu_t^{\star, \circ}$.
\end{lemma}
Another important ingredient of the proof of Proposition \ref{prop:equiv} is the relationship with the fixed point $\dot{\mu}_{\la}$ (cf. Proposition \ref{prop:ssz:physics:bp:fixed}) and $\dot{q}^\star_{\la}$ (cf. Proposition \ref{prop:BP:contraction}) established in \cite{ssz22}. Recall that we defined $\{\fF\}\equiv \{\tau \in \MMM: \, \dot{\tau} \notin \{ 0,1,\star\}, \hat{\tau}\notin \{ 0,1,\star\} \}$ and $\Omega \equiv \{\rr_0, \rr_1, \bb_0, \bb_1\} \cup \{\fF \}$.
\begin{lemma}[Lemma B.2 and B.3 in \cite{ssz22}]\label{lem:relation:bp}
Recall the BP fixed point $\dot{q}^\star_{\la}$ and $\hat{q}^\star_{\la}\equiv \textnormal{BP}\dot{q}^\star_{\la}$ in Proposition \ref{prop:BP:contraction}. For $\la\in [0,1]$, we have that
\begin{equation*}
\begin{split}
&\sum_{\dot{\tau}\notin \dot{\MMM}\setminus \{0,1,\star \}}\big(\dot{\mm}[\dot{\tau}](1)\big)^{\la}\dot{q}^\star_{\la}(\dot{\tau})=\sum_{\dot{\tau}\notin \dot{\MMM}\setminus \{0,1,\star \}}\big(\dot{\mm}[\dot{\tau}](0)\big)^{\la}\dot{q}^\star_{\la}(\dot{\tau})=\dot{q}^\star_{\la}(\rr_1)-\dot{q}^\star_{\la}(\bb_1)\,,\\
&\sum_{\hat{\tau}\notin \hat{\MMM}\setminus \{0,1,\star\}}\big(\hat{\mm}[\hat{\tau}](1)\big)^{\la}\hat{q}^\star_{\la}(\hat{\tau})=\sum_{\hat{\tau}\notin \hat{\MMM}\setminus \{0,1,\star\}}\big(\hat{\mm}[\hat{\tau}](0)\big)^{\la}\hat{q}^\star_{\la}(\hat{\tau})=\hat{q}^\star_{\la}(\bb_1)\,.
\end{split}
\end{equation*}
Moreover, the fixed point $\dot{\mu}_{\la}$ and $\hat{\mu}_{\la}\equiv \hat{\RR}_{\la}\dot{\mu}_{\la}$ in Proposition \ref{prop:ssz:physics:bp:fixed} is a discrete measure on $[0,1]$ whose support is at most countable, which is given as follows. For $y\in [0,1]$, we have
\begin{equation*}
\begin{split}
\dot{\mu}_{\la}(y)
&=\sum_{\dot{\tau}\in \dot{\MMM}\setminus\{0,1,\star\}}\frac{\dot{q}^\star_{\la}(\dot{\tau})}{1-\dot{q}^\star_{\la}(\rr)}\one\{\dot{\mm}[\dot{\tau}](1)=y\}+\sum_{\bx\in \{0,1\}}\frac{\dot{q}^\star_{\la}(\bb_{\bx})}{1-\dot{q}^\star_{\la}(\rr)}\one\{y=\bx\}\,,\\
\hat{\mu}_{\la}(y)
&=\sum_{\hat{\tau}\in \hat{\MMM}\setminus\{0,1,\star\}}\frac{\hat{q}^\star_{\la}(\hat{\tau})}{1-\hat{q}^\star_{\la}(\bb)}\one\{\hat{\mm}[\hat{\tau}](1)=y\}+\sum_{\bx\in \{0,1\}}\frac{\hat{q}^\star_{\la}(\rr_{\bx})}{1-\hat{q}^\star_{\la}(\bb)}\one\{y=\bx\}\,.
\end{split}
\end{equation*}
\end{lemma}
With Lemma \ref{lem:relation:bp} in hand, we establish the relationship between $\nu^{\star,\circ}_t$ and $\nu_{\la^\star}$, defined in Section \ref{s:lwl}. Denote by $\nu_t^{\star,\circ}(\,\cdot\,;\,\uL_t)$ the law of $\bsig_t^{\circ}$ conditional on $\buL_t =\uL_t$, where $(\bsig_t^{\circ},\buL_t)\sim \nu_t^{\star,\circ}$.
\begin{lemma}\label{lem:coloring:to:message}
    For any $\uL_t \in \{0,1\}^{E_{\sf in}(\TTT_{d,k,t})}$, there exists a coupling between $\bsig_t^{\circ}\equiv (\bsigma_e^{\circ})_{e\in E(\TTT_{d,k,t})} \sim \nu_t^{\star,\circ}(\,\cdot\,;\,\uL_t)$ and $\underline{\mm}_t\equiv (\mm_e)_{e\in E(\TTT_{d,k,t})} \sim \nu_{\la^\star}(\,\cdot\,;\,\uL_t)$ such that it satisfies the following conditions almost surely.
    \begin{enumerate}
        \item The frozen configurations corresponding to $\bsig_t^{\circ}$ and $\underline{\mm}_t$ are the same. More precisely, for any $e\in E(\TTT_{d,k,t})$ and $\bx\in \{0,1\}$, $\hat{\mm}_e=\delta_{\bx}$ if and only if $\bsigma_e^{\circ}=\rr_{\bx}$.
        \item For any $e\in \partial\TTT_{d,k,t}$, whenever $\bsigma_e^{\circ}\notin \{\rr, \bb\}$, $\dot{\mm}_e=\dot{\mm}[\dot{\bsigma}_e^{\circ}]$ and $\hat{\mm}_e=\hat{\mm}[\hat{\bsigma}_e^{\circ}]$ hold, where $\dot{\mm}[\cdot]$ and $\hat{\mm}[\cdot]$ are defined in \eqref{bethe:bpmsg:dot} and \eqref{bethe:bpmsg:hat} respectively. Otherwise, for $\bx\in \{0,1\}$ and $e\in \partial\TTT_{d,k,t}$, $\dot{\mm}_e=\delta_{\bx}$ holds if $\bsigma_e^{\circ}=\bb_{\bx}$, and $\hat{\mm}_e=\delta_{\bx}$ holds if $\bsigma_e^{\circ}=\rr_{\bx}$.
    \end{enumerate}
\end{lemma}
\begin{proof}
    For a message $\mm \equiv (\dot{\mm},\hat{\mm})\in \PPP(\{0,1\})^{2}$, define it's simplified message by $\mm^{\tns}\equiv \mm^{\tns}[\mm]$ by
    \begin{equation*}
    \mm^{\tns} =
    \begin{cases}
        (*,\delta_{\bx})~~~~&\textnormal{if $\hat{\mm}=\delta_{\bx}$ for some $\bx\in \{0,1\}$,}\\
        (\delta_{\bx},*) ~~~~&\textnormal{if $\hat{\mm}\notin \{\delta_0, \delta_1\}$ and $\dot{\mm}=\delta_{\bx}$ for some $\bx\in \{0,1\}$,}\\
        \mm ~~~~&\textnormal{otherwise.}
    \end{cases}
    \end{equation*}
    Thus, $\mm^{\tns}$ is obtained by erasing one side of $\mm$ when $\hat{\mm} \in \{\delta_0, \delta_1\}$ or $\dot{\mm} \in \{\delta_0, \delta_1\}$. Note that such procedure is similar to the projection in \eqref{eq:simplify:coloring} that leads to the coloring. To this end, for a coloring $\sigma^{\circ} \in \Omega$, we define its simplified message $\mm^{\circ}\equiv \mm^{\circ}[\sigma^{\circ}]$ by
     \begin{equation*}
    \mm^{\circ} =
    \begin{cases}
        (*,\delta_{\bx})~~~~&\textnormal{if $\sigma^{\circ}=\rr_{\bx}$ for some $\bx\in \{0,1\}$,}\\
        (\delta_{\bx},*) ~~~~&\textnormal{if $\sigma^{\circ}= \bb_{\bx}$ for some $\bx\in \{0,1\}$,}\\
        (\dot{\mm}[\dot{\sigma}],\hat{\mm}[\hat{\sigma}])~~~~&\textnormal{if $\sigma^{\circ} \in \{\fF\}$.}
    \end{cases}
    \end{equation*}
    For $\underline{\mm}_t \sim \nu_{\la^\star}(\,\cdot\,;\,\uL_t)$, denote its simplified message by $\underline{\mm}^{\tns}_t\equiv \mm^{\tns}[\underline{\mm}_t]$, where $\mm^{\tns}[\cdot]$ is applied coordinatewise. Similarly, for $\bsig^{\circ}_t\sim \nu^{\star, \circ}_t(\,\cdot\,;\,\uL_t)$, denote its simplirifed message by $\underline{\mm}^{\circ}_t\equiv \mm^{\circ}[\bsig^{\circ}_t]$, where $\mm^{\circ}[\cdot]$ is applied coordinatewise. Then, we claim that
    \begin{equation}\label{eq:goal:coloring:to:message}
    \underline{\mm}^{\tns}_t \stackrel{d}{=}\underline{\mm}^{\circ}_t
    \end{equation}
    Note that \eqref{eq:goal:coloring:to:message} implies our goal by definition of $\mm^{\tns}$ and $\mm^{\circ}$, so we establish \eqref{eq:goal:coloring:to:message} for the rest of the proof. Throughout, we assume that $\uL_t\equiv 0$ for simplicity. By the symmetry of $\naesat$ model, it is straightforward to generalize the argument to any $\uL_t\in \{0,1\}^{E_{\sf in}(\TTT_{d,k,t})}$.

    Observe that $\nu_{\la^\star}$ defined in Section \ref{s:lwl} is a Gibbs measure on the tree $\TTT_{d,k,t}$, so it can be described as a broadcast model analogous to $\nu_t^{\star,\circ}$ in Definition \ref{def:optiaml:t:ssz}. The channel $(\dot{K},\hat{K})$ of the broadcast model is given as follows. $\dot{K}$ and $\hat{K}$ are probability measures on $\left(\PPP(\{0,1\})^2\right)^{d}$ and $\left(\PPP(\{0,1\})^2\right)^{k}$ given by
    \begin{equation*}
    \begin{split}
    &\dot{K}(\mm_1,\ldots, \mm_d)= (\dot{Z})^{-1}\dot{\varphi}(\hat{\mm}_1,\ldots,\hat{\mm}_d)^{\la^\star}\prod_{i=1}^{d}\hat{\mu}_{\la^\star}\left(\hat{\mm}_i(1)\right)\,,\\
    &\hat{K}(\mm_1,\ldots, \mm_k)= (\hat{Z})^{-1}\hat{\varphi}^{\lit}(\dot{\mm}_1 ,\ldots,\dot{\mm}_d)^{\la^\star}\prod_{i=1}^{k}\dot{\mu}_{\la^\star}\left(\dot{\mm}_i(1)\right)\,, 
    \end{split}
    \end{equation*}
    where $\dot{Z}$ and $\hat{Z}$ are normalizing constants. Since $\underline{\mm}^{\tns}_t$ is obtained by a projection of $\underline{\mm}_t \sim \nu_{\la^\star}(\,\cdot\,;\,\uL_t)$, $\underline{\mm}^{\tns}_t$ is a broadcast model with channel $(\dot{K}^{\tns},\hat{K}^{\tns})$. Here, $\dot{K}^{\tns}$ is the probability measure given by the projection of $\dot{K}$ according to $\mm^{\tns}[\cdot]$: for $\mm_1^{\tns},\ldots,\mm_d^{\tns}\in \PPP(\{0,1\})^2\sqcup_{\bx\in \{0,1\}} \{(*,\delta_{\bx}),(\delta_{\bx},*)\}$,
    \begin{equation*}
    \dot{K}^{\tns}(\mm^{\tns}_1,\ldots, \mm^{\tns}_d)=\sum_{\mm_1,\ldots, \mm_d}\dot{K}(\mm_1,\ldots,\mm_d)\one\left\{\mm^{\tns}[\mm_i]=\mm^{\tns}_i~~\textnormal{for all}~~1\leq i \leq d \right\}\,.
    \end{equation*}
    $\hat{K}^{\tns}$ is defined similarly. Observe that we can simplify $\dot{K}^{\tns}$ using Lemma \ref{lem:relation:bp} as follows. Consider the case where the simplified messages $(\mm^{\tns}_1,\ldots, \mm^{\tns}_d)$ are adjacent to a frozen variable, i.e. $(\mm^{\tns}_1,\ldots, \mm^{\tns}_d)\in \sqcup_{\bx\in \{0,1\}}\left( \{(*,\delta_{\bx}),(\delta_{\bx},*)\}\right)^{d}$. Without loss of generality, suppose $(\mm^{\tns}_1,\ldots, \mm^{\tns}_d)\in \left( \{(*,\delta_{1}),(\delta_{1},*)\}\right)^{d}$ and let $I:=\{1\leq i \leq d:\mm^{\tns}_i=(*,\delta_{1})\}$ be the indices that correspond to the non-forcing messages. Then, note that $(\mm_1,\ldots, \mm_d)$ that satisfy $\mm^{\tns}[\mm_i]=\mm^{\tns}_i$ for $1\leq i\leq d$ is determined by $(\hat{\mm}_i)_{i\in I}\subset \PPP(\{0,1\})\setminus \{\delta_0,\delta_1\}$. Also, $\dot{\varphi}(\hat{\mm}_1,\ldots,\hat{\mm}_d)=\prod_{i\in I}\hat{\mm}_i(1)$ holds by definition of $\dot{\varphi}$. Thus, we have
    \begin{equation}\label{eq:simplify:K:1}
    \begin{split}
    \dot{K}^{\tns}(\mm_1^{\tns},\ldots,\mm_d^{\tns})
    &= (\dot{Z})^{-1}\hat{\mu}_{\la^\star}(1)^{d-|I|}\cdot \prod_{i\in I}\Bigg(\sum_{\hat{\mm}_i\in \PPP(\{0,1\})\setminus \{\delta_0,\delta_1\}}\big(\hat{\mm}_i(1)\big)^{\la^\star}\hat{\mu}_{\la^\star}\big(\hat{\mm}_i(1)\big)\Bigg)\\
    &=\Big(\dot{Z}\big(1-\hat{q}^{\star}_{\la^\star}(\bb)\big)^d\Big)^{-1}\hat{q}^\star_{\la^\star}(\rr_1)^{d-|I|}\hat{q}^\star_{\la^\star}(\bb_1)^{|I|}\,,
    \end{split}
    \end{equation}
    where the last equality is due to Lemma \ref{lem:relation:bp}. Next, consider the case where the simplified messages are adjacent to a free variable, i.e. $\mm_1^{\tns},\ldots, \mm_d^{\tns}\in \big(\PPP(\{0,1\})\setminus \{\delta_0,\delta_1\}\big)^2$. Then, by Lemma \ref{lem:relation:bp}, we have
    \begin{equation}\label{eq:simplify:K:2}
    \begin{split}
    \dot{K}^{\tns}(\mm_1^{\tns},\ldots,\mm_d^{\tns})=\dot{K}(\mm_1^{\tns},\ldots,\mm_d^{\tns})=\Big(\dot{Z}\big(1-\hat{q}^{\star}_{\la^\star}(\bb)\big)^d\Big)^{-1}\sum_{\hat{\tau}_1,\ldots, \hat{\tau}_{d}\in \hat{\MMM}\setminus\{0,1,\star\}}\dot{\varphi}\big(\hat{\mm}[\hat{\tau}_1],\ldots, \hat{\mm}[\hat{\tau}_d]\big)^{\la^\star}\prod_{i=1}^{d}\hat{q}^\star_{\la^\star}(\hat{\tau}_i)\,.
    \end{split}
    \end{equation}
    Similarly, we can simplify $\hat{K}^{\tns}$ as follows. Consider the case where the simplified messages $(\mm^{\tns}_1,\ldots, \mm^{\tns}_k)$ are adjacent to a forcing clause, i.e. $(\mm^{\tns}_1,\ldots, \mm^{\tns}_k)\in \sqcup_{\bx\in \{0,1\}}\textnormal{Per}\big((*,\delta_{\bx}), (\delta_{1-\bx},*)^{k-1}\big)$. Without loss of generality, suppose $(\mm^{\tns}_1,\ldots, \mm^{\tns}_k)=\big((*,\delta_{1}), (\delta_{0},*)^{k-1}\big)$. Then, note that $(\mm_1,\ldots, \mm_k)$ that satisfy $\mm^{\tns}[\mm_i]=\mm^{\tns}_i$ for $1\leq i\leq k$ is determined by $\dot{\mm}_1\in \PPP(\{0,1\})$. Also, $\hat{\varphi}^{\lit}(\dot{\mm}_1,\ldots,\dot{\mm}_k)=\dot{\mm}_1(0)$ holds by definition of $\hat{\varphi}^{\lit}$. Thus, we have
    \begin{equation}\label{eq:simplify:K:3}
    \begin{split}
    \hat{K}^{\tns}(\mm^{\tns}_1,\ldots, \hat{\mm}^{\tns}_k)
    &=(\hat{Z})^{-1}\dot{\mu}_{\la^\star}(1)^{k-1}\cdot \Bigg(\sum_{\dot{\mm}_1\in \PPP(\{0,1\})}\big(\dot{\mm}_1(0)\big)^{\la^\star}\dot{\mu}_{\la^\star}\big(\dot{\mm}_1(1)\big)\Bigg)\\
    &=\Big(\hat{Z}\big(1-\dot{q}^{\star}_{\la^\star}(\rr)\big)^k\Big)^{-1}\dot{q}^\star_{\la^\star}(\bb_1)^{k-1}\dot{q}^\star_{\la^\star}(\rr_0)\,,
    \end{split}
    \end{equation}
    where the last equality is due to Lemma \ref{lem:relation:bp}. Next, consider the case where the simplified messages are adjacent to a non-forcing clause, i.e. $\mm^{\tns}_1,\ldots, \mm^{\tns}_k\in \big(\PPP(\{0,1\})\big)^2\sqcup_{\bx\in \{0,1\}} \big\{(\delta_{\bx},*)\big\}$. Then, note that for all $(\mm_1,\ldots, \mm_k)$ that satisfy $\mm^{\tns}[\mm_i]=\mm^{\tns}_i$ for $1\leq i\leq k$, $(\dot{\mm}_1,\ldots, \dot{\mm}_k)$ is determined by $(\mm^{\tns}_1,\ldots, \mm^{\tns}_k)$. With abuse of notation, set $\dot{q}^\star_{\la^\star}(\bx)\equiv \dot{q}^\star_{\la^\star}(\bb_{\bx})$ for $\bx\in \{0,1\}$. Then, by Lemma \ref{lem:relation:bp}, we have
    \begin{equation}\label{eq:simplify:K:4}
    \begin{split}
          \hat{K}^{\tns}(\mm^{\tns}_1,\ldots, \mm^{\tns}_k)
          &=\sum_{\mm_1,\ldots, \mm_k}\hat{K}(\mm_1,\ldots,\mm_k)\one\left\{\mm^{\tns}[\mm_i]=\mm^{\tns}_i~~\textnormal{for all}~~1\leq i \leq k \right\}\\
          &=\Big(\hat{Z}\big(1-\dot{q}^{\star}_{\la^\star}(\rr)\big)^k\Big)^{-1}\sum_{\dot{\tau}_1,\ldots, \dot{\tau}_{k}\in \dot{\MMM}\setminus\{\star\}}\hat{\varphi}^{\lit}\big(\dot{\mm}[\dot{\tau}_1],\ldots, \dot{\mm}[\dot{\tau}_k]\big)^{\la^\star}\prod_{i=1}^{d}\dot{q}^\star_{\la^\star}(\dot{\tau}_i)\,.
    \end{split}
    \end{equation}
    Now, since $\mm^{\circ}_t$ is obtained by a projection of $\bsig_t^{\circ}\sim \nu_t^{\star,\circ}(\,\cdot\,;\,\uL_t)$, $\mm^{\circ}_t$ is a broadcast model with channel $(\dot{K}^{\circ},\hat{K}^{\circ})$, where $(\dot{K}^{\circ},\hat{K}^{\circ})$ is given by
    \begin{equation}\label{eq:K:circ}
    \begin{split}
        &\dot{K}^{\circ}(\mm^{\circ}_1,\ldots, \mm^{\circ}_d)= \sum_{\sigma_1^{\circ},\ldots, \sigma^{\circ}_d}\dot{H}^\star(\sigma_1^{\circ},\ldots,\sigma_d^{\circ})\one\left\{\mm^{\circ}[\sigma_i^{\circ}]=\mm^{\circ}_i~~\textnormal{for all}~~1\leq i \leq d \right\}\,,\\
        &\hat{K}^{\circ}(\mm^{\circ}_1,\ldots, \mm^{\circ}_k)= \sum_{\sigma_1^{\circ},\ldots, \sigma^{\circ}_k}\hat{H}^{\star, \lit}(\sigma_1^{\circ},\ldots,\sigma_k^{\circ})\one\left\{\mm^{\circ}[\sigma_i^{\circ}]=\mm^{\circ}_i~~\textnormal{for all}~~1\leq i \leq k \right\}\,,
    \end{split}
    \end{equation}
    where $\hat{H}^{\star,\lit}$ is defined as
    \begin{equation*}
        \hat{H}^{\star,\lit}(\sig^{\circ})=\frac{\hat{\Phi}^{\lit}(\sig^{\circ})^{\la^\star}}{\hat{\mathfrak{Z}}^{\lit}}\prod_{i=1}^{k}\dot{q}^\star_{\la^\star}(\sigma^{\circ}_i)\,.
    \end{equation*}
    In the above equation, $\hat{\mathfrak{Z}}^{\lit}$ is a normalizing constant. 

    Observe that the expressions of $\dot{K}^{\tns}$ and $\hat{K}^{\tns}$ obtained in \eqref{eq:simplify:K:1}-\eqref{eq:simplify:K:4} show that $(\dot{K}^{\tns},\hat{K}^{\tns})$ equals the channel $(\hat{K}^{\circ},\hat{K}^{\circ})$ in \eqref{eq:K:circ}. Therefore, we conclude that the law of $\underline{\mm}^{\tns}_t$ and the law of $\underline{\mm}^{\circ}_t$ is the same, which concludes the proof of \eqref{eq:goal:coloring:to:message}.
\end{proof}
\begin{proof}[Proof of Proposition \ref{prop:equiv}]
To begin with, a standard dynamic programming (or belief propagation) calculation (e.g. see \cite[Chapter 14]{am06} or the proof of \cite[Lemma 2.9]{ssz22}) shows that we can express $p_{\uz_t}(\sig_t,\uL_t)$ as follows. Let $(\sig^{\circ}_t,\uL_t)=\Psi(\sig_t,\uL_t)$ be the coloring configuration corresponding to $(\sig_t,\uL_t)$. Note that $\sig^{\circ}_t\equiv (\sigma^{\circ}_e)_{e\in E(\TTT_{d,k,t})}$ determines the frozen configuration $\ux_t\equiv (x_v)_{v\in V(\TTT_{d,k,t})}\equiv \ux[\sig^{\circ}_t]\in \{0,1,\ff\}^{V(\TTT_{d,k,t})}$. Write $\uz_t\sim_{\uL_t}\sig^{\circ}_t$ when $\uz_t=(z_v)_{v\in V(\TTT_{d,k,t})}$ is a valid $\naesat$ solution w.r.t. the literals $\uL_t$ and $z_v=x_v$ whenever $x_v\in \{0,1\}$. Then,  
\begin{equation*}
    p_{\uz_t}(\sig_t,\uL_t)=\frac{\one\{\uz_t\sim_{\uL_t}\sig_t^{\circ}\}\prod_{e\in \partial \TTT_{d,k,t}:\sigma^{\circ}_e\not\in \{\rr,\bb\}}\hat{\mm}[\hat{\sigma}_e^{\circ}](z_{v(e)})}{\sum_{\uz^\prime_t\in \{0,1\}^{V(\TTT_{d,k,t})}}\one\{\uz_t^\prime \sim_{\uL_t}\sig_t^{\circ}\}\prod_{e\in \partial \TTT_{d,k,t}:\sigma^{\circ}_e\not\in \{\rr,\bb\}}\hat{\mm}[\hat{\sigma}_e^{\circ}](z_{v(e)}^\prime)}\,.
\end{equation*}
Note that in the expression above, $\sigma_e^{\circ}\notin \{\rr,\bb\}$ guarantees that $\hat{\mm}[\hat{\sigma}_e^{\circ}]\in \PPP(\{0,1\})$ is well defined. Thus, combining with Lemma \ref{lem:equiv}, $\sum_{\sig_{t}\in \CCC_t}\nu_{t}^\star(\sig_t,\uL_t)p_{\uz_t}(\sig_t,\uL_t)$ equals
\begin{equation*}
2^{-|E_{\sf in}(\TTT_{d,k,t})|}\cdot \E_{\bsig^{\circ}_t\sim \nu_t^{\star,\circ}(\,\cdot\,;\,\uL_t)}\left[\frac{\one\{\uz_t\sim_{\uL_t}\bsig_t^{\circ}\}\prod_{e\in \partial \TTT_{d,k,t}:\bsigma^{\circ}_e\not\in \{\rr,\bb\}}\hat{\mm}[\hat{\bsigma}_e^{\circ}](z_{v(e)})}{\sum_{\uz^\prime_t\in \{0,1\}^{V(\TTT_{d,k,t})}}\one\{\uz_t^\prime \sim_{\uL_t}\bsig_t^{\circ}\}\prod_{e\in \partial \TTT_{d,k,t}:\bsigma^{\circ}_e\not\in \{\rr,\bb\}}\hat{\mm}[\hat{\bsigma}_e^{\circ}](z_{v(e)}^\prime)}\right]\,.
\end{equation*}
Furthermore, by Lemma \ref{lem:coloring:to:message}, we can couple $\bsig_t^{\circ}\equiv (\bsigma_e^{\circ})_{e\in E(\TTT_{d,k,t})} \sim \nu_t^{\star,\circ}(\,\cdot\,;\,\uL_t)$ and $\underline{\mm}_t\equiv (\mm_e)_{e\in E(\TTT_{d,k,t})} \sim \nu_{\la^\star}(\,\cdot\,;\,\uL_t)$ so that almost surely, their frozen configurations are the same and the boundary messages $(\dot{\mm}[\dot{\bsigma}_e^{\circ}], \hat{\mm}[\hat{\bsigma}_e^{\circ}])$ and $\mm_e$ are the same for $e\in \partial \TTT_{d,k,t}$ whenever $\bsigma_e^{\circ}\notin \{\rr,\bb\}$. Here, note that $\dot{\mm}[\dot{\bsigma}_e^{\circ}], \hat{\mm}[\hat{\bsigma}_e^{\circ}]\notin \{\delta_0, \delta_1\}$ holds if $\bsigma_e^{\circ}\notin \{\rr,\bb\}$, so for such $e\in \partial\TTT_{d,k,t}$, $\dot{\mm}_e,\hat{\mm}_e \not\in \{\delta_0,\delta_1\}$. Also, if $\bsigma_e^{\circ}\in \{\rr,\bb\}$, $\dot{\mm}_e\in \{\delta_0, \delta_1\}$ or $\hat{\mm}_e\in \{\delta_0, \delta_1\}$ by our construction of the coupling. Thus, the expression in the display above equals
\begin{equation*}
2^{-|E_{\sf in}(\TTT_{d,k,t})|}\cdot \E_{\underline{\mm}_t\sim  \nu_{\la^\star}(\,\cdot\,;\,\uL_t)}\left[\frac{\one\{\uz_t\sim_{\uL_t}\ux(\underline{\mm}_t)\}\prod_{e\in \partial \TTT_{d,k,t}:\dot{\mm}_e,\hat{\mm}_e \notin \{\delta_0,\delta_1\}}\hat{\mm}_e(z_{v(e)})}{\sum_{\uz^\prime_t\in \{0,1\}^{V(\TTT_{d,k,t})}}\one\{\uz_t^\prime \sim_{\uL_t}\ux(\underline{\mm}_t)\}\prod_{e\in \partial \TTT_{d,k,t}:\dot{\mm}_e,\hat{\mm}_e \notin \{\delta_0,\delta_1\}}\hat{\mm}_e(z_{v(e)}^\prime)}\right]\,.
\end{equation*}
Finally, observe that given $\underline{\mm}_t$ and $e\in \partial \TTT_{d,k,t}$, the condition $\dot{\mm}_e\in \{\delta_0,\delta_1\}$ or $\hat{\mm}_e\in \{\delta_0,\delta_1\}$ implies that $v(e)$ is frozen in the frozen configuration $\ux(\underline{\mm}_t)$. Thus, given $\underline{\mm}_t$ and $e\in \partial \TTT_{d,k,t}$ such that $\dot{\mm}_e\in \{\delta_0,\delta_1\}$ or $\hat{\mm}_e\in \{\delta_0,\delta_1\}$, $\hat{\mm}_e(z^\prime_v(e))$ is the same for all $\uz^\prime_t\in \{0,1\}^{V(\TTT_{d,k,t})}$ such that $\uz_t^\prime \sim_{\uL_t}\ux(\mm_t)$. Therefore, we have that
\begin{equation*}
\sum_{\sig_{t}\in \CCC_t}\nu_{t}^\star(\sig_t,\uL_t)p_{\uz_t}(\sig_t,\uL_t)=2^{-|E_{\sf in}(\TTT_{d,k,t})|} \E_{\underline{\mm}_t\sim  \nu_{\la^\star}(\,\cdot\,;\,\uL_t)}\left[\frac{\one\{\uz_t\sim_{\uL_t}\ux(\underline{\mm}_t)\}\prod_{e\in \partial \TTT_{d,k,t}}\hat{\mm}_e(z_{v(e)})}{\sum_{\uz^\prime_t\in \{0,1\}^{V(\TTT_{d,k,t})}}\one\{\uz_t^\prime \sim_{\uL_t}\ux(\underline{\mm}_t)\}\prod_{e\in \partial \TTT_{d,k,t}}\hat{\mm}_e(z_{v(e)}^\prime)}\right]\,,
\end{equation*}
which concludes the proof.
\end{proof}
\end{document}